\documentclass[a4paper,11pt]{article}
\usepackage[utf8]{inputenc}
\usepackage{amsmath,amssymb}
\usepackage{array,multirow} 
\usepackage{lmodern}
\usepackage[osf]{mathpazo} 
\usepackage[T1]{fontenc}
\usepackage{enumerate}
\usepackage{xcolor}
\usepackage{colortbl}
\usepackage{thmbox}
\usepackage{graphicx}
\usepackage[english]{babel}
\usepackage[nottoc, notlof, notlot]{tocbibind}
\usepackage{hyperref}
\usepackage{numprint} 
\usepackage{hhline} 
\usepackage[normalem]{ulem} 
\usepackage{bigints} 
\usepackage{soul} 

\newcommand{\dst}{\displaystyle}

\newcommand{\Acal}{\mathcal{A}}
\newcommand{\Bcal}{\mathcal{B}}
\newcommand{\Ccal}{\mathcal{C}}
\newcommand{\Fcal}{\mathcal{F}}

\newcommand{\Lcal}{\mathcal{L}}
\newcommand{\Mcal}{\mathcal{M}}

\newcommand{\Pcal}{\mathcal{P}}

\newcommand{\Scal}{\mathcal{S}}

\newcommand{\IE}{\mathbb{E}}
\newcommand{\IN}{\mathbb{N}}
\newcommand{\IP}{\mathbb{P}}
\newcommand{\IQ}{\mathbb{Q}}
\newcommand{\IR}{\mathbb{R}}
\newcommand{\IZ}{\mathbb{Z}}

\newcommand{\drm}{\mathrm{d}}
\newcommand{\erm}{\mathrm{e}}
\newcommand{\Irm}{\mathrm{I}}
\newcommand{\Lrm}{\mathrm{L}}

\newcommand{\vbar}{\overline{v}}

\newcommand{\Ybar}{\overline{Y}}
\newcommand{\Zbar}{\overline{Z}}

\newcommand{\lambdabar}{\overline{\lambda}}

\newcommand{\zetabar}{\overline{\zeta}}

\newcommand{\chat}{\widehat{c}}

\newcommand{\nuhat}{\widehat{\nu}}

\newcommand{\Ctilde}{\widetilde{C}}

\newcommand{\utilde}{\widetilde{u}}

\newcommand{\vtilde}{\widetilde{v}}

\newcommand{\Wtilde}{\widetilde{W}}

\newcommand{\Ytilde}{\widetilde{Y}}
\newcommand{\Ztilde}{\widetilde{Z}}
\newcommand{\IEtilde}{\widetilde{\IE}}
\newcommand{\IPtilde}{\widetilde{\IP}}

\newcommand{\Deltatilde}{\widetilde{\Delta}}

\newcommand{\lambdatilde}{\widetilde{\lambda}}

\newcommand{\psitilde}{\widetilde{\psi}}

\newcommand{\zetatilde}{\widetilde{\zeta}}

\renewcommand{\leq}{\leqslant}
\renewcommand{\geq}{\geqslant}
\renewcommand{\epsilon}{\varepsilon}

\newcommand{\GL}{\mathrm{GL}}
\newcommand{\ind}{\mathbb{1}}
\newcommand{\loc}{\mathrm{loc}}
\newcommand{\lip}{\mathrm{lip}}

\makeatletter
\def\newboxtheorem{%
  \@ifstar\thmbox@new@star\newboxtheorem@nostar}%
\def\newboxtheorem@nostar{%
  \@ifnextchar[{\thmbox@newA}{\thmbox@newA[]}}%
\def\thmbox@new@star{%
  \@ifnextchar[{\thmbox@new@star@A}{\thmbox@new@star@A[]}}%
\def\thmbox@new@star@A[#1]#2#3{%
  \expandafter\def\csname#2\endcsname{%
    \setkeys{thmbox}{#1}%
    \@ifnextchar[{\thmbox@star@beginA{#3}{}}{%
      \thmbox@star@begin{#3}{}}}%
  \expandafter\def\csname end#2\endcsname{%
    \endthmbox\smallbreak}}
\def\thmbox@star@beginA#1[#2]{%
  \thmbox@star@begin{#1}{\thmbox@titlestyle{#2}}{}}
\def\thmbox@star@begin#1#2{%
  \medbreak
  \thmbox{\thmbox@headstyle{#1}{\unskip}#2}%
  \thmbox@bodystyle\ignorespaces}
\let\newtheorem\newboxtheorem
\makeatother

\newtheorem[style=L,bodystyle=\noindent\upshape]{thm}{Theorem}
\newtheorem[style=M,bodystyle=\noindent\upshape]{cor}[thm]{Corollary}
\newtheorem[style=M,bodystyle=\noindent\upshape]{lem}[thm]{Lemma}
\newtheorem[style=M,bodystyle=\noindent\upshape]{prop}[thm]{Proposition}
\newtheorem[style=S,bodystyle=\noindent\upshape]{de}[thm]{Definition}
\newtheorem[style=S,bodystyle=\noindent\upshape]{ass}[thm]{Assumption}
\newenvironment{rmq}{\stepcounter{thm} \noindent \small \textbf{Remark \thethm. }}{\normalsize}
\newtheorem*[style=L,bodystyle=\noindent\upshape]{theo}{Theorem}
\newtheorem[style=S,bodystyle=\noindent\upshape]{hyp}[thm]{Hypothesis}

\newcommand{\limit}[2]{
\underset{#1 \rightarrow #2}{\longrightarrow}}

\newcommand{\limite}[3]{
\underset{#2 \rightarrow #3}{\overset{#1}{\longrightarrow}}}

\def\restriction#1#2{\mathchoice
              {\setbox1\hbox{${\displaystyle #1}_{\scriptstyle #2}$}
              \restrictionaux{#1}{#2}}
              {\setbox1\hbox{${\textstyle #1}_{\scriptstyle #2}$}
              \restrictionaux{#1}{#2}}
              {\setbox1\hbox{${\scriptstyle #1}_{\scriptscriptstyle #2}$}
              \restrictionaux{#1}{#2}}
              {\setbox1\hbox{${\scriptscriptstyle #1}_{\scriptscriptstyle #2}$}
              \restrictionaux{#1}{#2}}}
\def\restrictionaux#1#2{{#1\,\smash{\vrule height .8\ht1 depth .85\dp1}}_{\,#2}} 

\makeatletter
\newcommand{\pushright}[1]{\ifmeasuring@#1\else\omit\hfill$\displaystyle#1$\fi\ignorespaces}
\newcommand{\pushleft}[1]{\ifmeasuring@#1\else\omit$\displaystyle#1$\hfill\fi\ignorespaces}
\makeatother

\begin{document}

\title{\normalsize \bf Ergodic BSDE with an unbounded and multiplicative underlying diffusion\\
 and application to large time behavior of viscosity solution of HJB equation}

\date{}
\author{Ying Hu\thanks{Univ Rennes, CNRS, IRMAR - UMR 6625, F-35000 Rennes, France (ying.hu@univ-rennes1.fr) and School of
Mathematical Sciences, Fudan University, Shanghai 200433, China.
Partially supported by Lebesgue center of mathematics ``Investissements d'avenir"
program - ANR-11-LABX-0020-01,  by CAESARS - ANR-15-CE05-0024 and by MFG - ANR-16-CE40-0015-01.} \and
Florian Lemonnier\thanks{Univ Rennes, CNRS, IRMAR - UMR 6625, F-35000 Rennes, France (Florian.Lemonnier@ens-rennes.fr). 
Partially supported by Lebesgue center of mathematics ``Investissements d'avenir"
program - ANR-11-LABX-0020-01,  by CAESARS - ANR-15-CE05-0024 and by MFG - ANR-16-CE40-0015-01. }}

\maketitle

\normalsize

\textbf{Abstract.} In this paper, we study ergodic backward stochastic differential equations (EBSDEs for short), for which the underlying diffusion is assumed to be multiplicative and of at most linear growth. The fact that the forward process has an unbounded diffusion is balanced with an assumption of weak dissipativity for its drift. Moreover, the forward equation is assumed to be non-degenerate. Like in \cite{PDE}, we show that the solution of a BSDE in finite horizon $T$ behaves basically as a linear function of $T$, with a shift depending on the solution of the associated EBSDE, with an explicit rate of convergence. Finally, we apply our results to an ergodic optimal control problem. In particular, we show the large time behaviour of viscosity solution of Hamilton-Jacobi-Bellman equation with an exponential rate of convergence when the undelrying diffusion is multiplicative and unbounded.



\section{Introduction}

We study the following EBSDE in finite dimension and infinite horizon: for all $t,T \in \IR_+$ such that $0 \leq t \leq T < \infty$,
\begin{equation}
 Y_t^x = Y_T^x + \int_t^T \left\{ \psi\left(X_s^x, Z_s^x\right) - \lambda\right\} \, \drm s - \int_t^T Z_s^x \, \drm W_s,
\label{eqintro:EBSDE}
\end{equation}
where the unknown is the triplet $\left(Y^x, Z^x, \lambda\right)$, with:
\begin{itemize}
\item $Y^x$ a real-valued and progressively measurable process;
\item $Z^x$ a $\left(\IR^d\right)^*$-valued and progressively measurable process;
\item $\lambda$ a real number.
\end{itemize}

The given data of our equation consists in:
\begin{itemize}
\item $W$, a $\IR^d$-valued standard Brownian motion;
\item $x \in \IR^d$;
\item $X^x$, a $\IR^d$-valued process, starting from $x$, and solution of the SDE: for all $t \in \IR_+$,
\begin{equation}
X_t^x = x + \int_0^t \Xi\left(X_s^x\right) \, \drm s + \int_0^t \sigma\left(X_s^x\right) \, \drm W_s;
\label{eqintro:SDE}
\end{equation}
\item $\psi : \IR^d \times \left(\IR^d\right)^* \rightarrow \IR$ a measurable function.
\end{itemize}

This class of ergodic BSDEs was first introduced by Fuhrman, Hu and Tessitore in \cite{EBSDE1}, in order to study optimal ergodic control problem. In that paper, the main assumption is the strong dissipativity of $\Xi$, that is to say:
\[ \exists \eta > 0,\, \forall x,x' \in \IR^d,\, \langle \Xi(x) - \Xi(x'), x-x' \rangle \leq -\eta |x-x'|^2.\]

The strong dissipativity assumption was then dropped off in \cite{EBSDE2} and replaced by a weak dissipativity assumption: in other words, $\Xi$ can be written as the sum of a dissipative function and of a bounded function.

A few years later, in \cite{PDE}, under similar assumptions, the large time behaviour of BSDEs in finite horizon $T$: for all $t \in [0,T]$, 
\begin{equation}
Y_t^{T,x} = g\left(X_T^x\right) + \int_t^T \psi\left(X_s^x, Z_s^{T,x}\right) \, \drm s - \int_t^T Z_s^{T,x} \, \drm W_s,
\label{eqintro:BSDE}
\end{equation}
was studied and linked with ergodic BSDEs. The authors prove the existence of a constant $L \in \IR$, such that for all $x \in \IR^d$,
\[ Y_0^{T,x} - \lambda T - Y_0^x \underset{T \rightarrow \infty}{\longrightarrow} L;\]
moreover, they obtain an exponential rate on convergence.

Those BSDEs are linked with Hamilton-Jacobi-Bellman equations. Indeed, the solution of equation (\ref{eqintro:BSDE}) can be written as $Y_t^{T,x} = u\left(T-t, X_t^x\right)$, where $u$ is the solution of the Cauchy problem
\begin{equation}
 \left\{ \begin{array}{llll}
           \partial_t u(t,x) &=& \Lcal u(t,x) + \psi\left(x,\nabla u(t,x) \sigma(x)\right) & \forall (t,x) \in \IR_+ \times \IR^d, \\
           u(0,x) &=& g(x) & \forall x \in \IR^d,
           \end{array} \right.
\label{eqintro:PDE}
\end{equation}
and where $\Lcal$ is the generator of the Kolmogorov semigroup of $X^x$, solution of (\ref{eqintro:SDE}). Also, the ergodic BSDE (\ref{eqintro:EBSDE}) admits a solution such that $Y_t^x = v\left(X_t^x\right)$ satisfies
\begin{equation}
\Lcal v(x) + \psi\left(x,\nabla v(x) \sigma(x)\right) - \lambda = 0 \ \ \ \ \ \forall x \in \IR^d.
\label{eqintro:EPDE}
\end{equation}

Because the functions $u$ and $v$ built by solving BSDEs or EBSDEs are not, in general, of class $\Ccal^2$, we only solve these equations in a weak sense. In \cite{PDE}, the authors prove that, under their assumptions, the function $v$ is of class $\Ccal^1$ and they are able to work with mild solutions. But in this article, we only prove the continuity of $v$, so we link the solution of the EBSDE (\ref{eqintro:EBSDE}) with the viscosity solution of the ergodic PDE (\ref{eqintro:EPDE}). Viscosity solutions of PDEs have already been widely studied (see \cite{Ish06}).

The large time behaviour of such PDEs has already been widely studied, for example in \cite{Cag15} and the references therein: but we do not make here the same assumptions. The authors work on the torus $\IR^n / \IZ^n$, and they need the Hamiltonian $\psi$ to be uniformly convex; on the other side, they are less restrictive about the matrix $\sigma(x) \sigma(x)^*$ which only needs to be nonnegative definite. They prove the convergence of $u(t,\bullet) - v - \lambda t$ as $t \rightarrow \infty$, but they do not have any rate of convergence. The same method has been used in \cite{Mitake15} to study the large time behaviour of the solution to the obstacle problem for degenerate viscous Hamilton-Jacobi equations.

But in \cite{EBSDE1}, \cite{EBSDE2} or \cite{PDE}, the diffusion $\sigma$ of the forward process $X^x$ is always supposed to be constant. The main contribution in this article is that the function $\sigma$ is here assumed to be Lipschitz continuous, invertible, and such that $\sigma^{-1}$ is bounded. Moreover, we will need the linear growth of $\sigma$ to be small enough, regarding the weak dissipativity of the drift $\Xi$: consequently, in average, the forward process $X^x$ is attracted to the origin. The key point is a coupling estimate for a multiplicative noise obtained by the irreducibility of $X^x$ (see \cite{DaPrato96}, \cite{Cohen15} or \cite{Mattingly01}): the proof (see Theorem \ref{thm:coupling_estimate}) is different from \cite{Madec15}, where a bridge was used in equation (B.6), which required $\sigma$ to fulfill $\langle y, \sigma(x+y)-\sigma(x)\rangle \leq \Lambda |y|$ with $\Lambda \geq 0$ not too large. Then, we apply this result to auxiliary monotone BSDEs in infinite horizon:
\[ Y_t^{\alpha,x} = Y_T^{\alpha,x} + \int_t^T \left\{ \psi\left(X_s^x, Z_s^{\alpha,x}\right) - \alpha Y_s^{\alpha,x}\right\} \, \drm s - \int_t^T Z_s^{\alpha,x} \, \drm W_s,\]
where $\alpha \in \IR_+^*$ (see \cite{BH98} or \cite{Royer04}); we get that $Y_0^{\alpha,x}$ is locally Lipschitz with respect to $x$, and it allows us to prove the convergence of $\left(Y^{\alpha,x} - Y^{\alpha,0}_0, Z^{\alpha,x}, \alpha Y_0^{\alpha,0}\right)$ to a solution $\left(Y^x, Z^x, \lambda\right)$ of the EBSDE, for every $x \in \IR^d$. We can prove uniqueness for $\lambda$, but we can only expect uniqueness of $\left(Y^x, Z^x\right)$ as measurable functions of $X^x$.

In \cite{PDE}, the large time behaviour is obtained when $\psi$ is assumed to be Lipschitz continuous with respect to $z$ only. But we could not extend this result when $\sigma$ is unbounded: we need here $\psi$ to be Lipschitz continuous with respect to $z \sigma(x)^{-1}$ (it is not equivalent to being Lipschitz continuous with respect to $z$ when $\sigma$ is unbounded). The price to pay is that we also require the existence of a constant $K_x$ such that:
\[ |\psi(x,z) - \psi(x',z')| \leq K_x |x-x'| + K_z \left|z \sigma(x)^{-1} - z' \sigma(x')^{-1}\right|,\]
in order to keep the implication ``$Y_0^{\alpha,x}$ has quadratic growth, so its increments have also quadratic growth''. This part of our work is somewhat technical and is presented in the appendix.

The paper is organised as follows. Some notations are introduced in section 2. In section 3, we study a SDE slightly more general than the one satisfied by $X^x$: indeed, it will be the SDE satisfied by $X^x$ after a change of probability space, due to Girsanov's theorem. In section 4, we prove that, in a way, our EBSDE admits a unique solution, which allows us also to solve an ergodic PDE. The link between this solution and the solutions of some finite horizon BSDEs is presented in section 5. The section 6 is devoted to an application of our results to an optimal ergodic control problem. Finally, the appendix presents how we keep the estimates of the increments of $Y_0^{\alpha,x}$ with respect to $x$, despite our twisted Lipschitz assumption on $\psi$.

\section{Notations}

Throughout this paper, $\left(W_t\right)_{t \geq 0}$ will denote a $d$-dimensional Brownian motion, defined on a parobability space $(\Omega, \Fcal, \IP)$. For $t \geq 0$, set $\Fcal_t$ the $\sigma$-algebra generated by $W_s$, $0 \leq s \leq t$, and augmented with the $\IP$-null sets of $\Fcal$. We write elements of $\IR^d$ as vectors, and the star ${}^*$ stands for transposition. The Euclidean norm on $\IR^d$ and $\left(\IR^d\right)^*$ will be denoted by $|\bullet|$. For a matrix of $\Mcal_d(\IR)$, we denote by $|\bullet|_F$ its Frobenius norm, that is to say the square root of the sum of the square of its coefficients. The Lipschitz constant of a function is written $\|\bullet\|_\lip$ and we also write $\|f\|_\infty := \sup_x |f(x)|$ for a bounded function $f$. Finally, for any process, $|\bullet|^{*,p}_{t,T}$ stands for $\dst \sup_{s \in [t,T]} \left|\bullet_s\right|^p$, and the exponent $p$ is omitted when equal to $1$.

\section{The SDE}

We consider the SDE:
\begin{equation}
\left\{ \begin{array}{rcl}
  \drm X_t^x &=& \Xi\left(t,X_t^x\right) \, \drm t + \sigma\left(X_t^x\right) \, \drm W_t, \\
  X_0^x &=& x,
\end{array} \right.
\label{eq:SDE}
\end{equation}

\begin{ass}
\begin{itemize}
\item $\Xi : \IR_+ \times \IR^d \rightarrow \IR^d$ is Lipschitz continuous w.r.t. $x$ uniformly in $t$: precisely, there exists some constants $\xi_1$ and $\xi_2$ such that for all $t \in \IR_+$ and $x \in \IR^d$, $|\Xi(t,x)| \leq \xi_1 + \xi_2|x|$;
\item $\Xi$ is weakly dissipative, \textit{i.e.} $\Xi(t,x) = \xi(x) + b(t,x)$, where $\xi$ is dissipative (that is to say $\langle \xi(x),x \rangle \leq - \eta|x|^2$, with $\eta > 0$) and locally Lipschitz and $b$ is bounded; this way, $\exists \eta_1, \eta_2 >0,\, \forall t \in \IR_+,\, \forall x \in \IR^d,\, \langle \Xi(t,x),x  \rangle \leq \eta_1 -\eta_2 |x|^2$;
\item $\sigma : \IR^d \rightarrow \GL_d(\IR)$ is Lipschitz continuous, the function $x \mapsto \sigma(x)^{-1}$ is bounded and $\forall x \in \IR^d,\, |\sigma(x)|_F^2 \leq r_1 + r_2|x|^2$.
\end{itemize}
\label{ass_SDE}
\end{ass}

\begin{rmq}
If we know that $\sigma$ has a sublinear growth, \emph{i.e.} $|\sigma(x)|_F \leq C\left(1+|x|^\alpha\right)$ with $\alpha \in (0,1)$, then, by Young's inequality, we have $|\sigma(x)|_F^2 \leq r_1+r_2|x|^2$, and $r_2$ can be chosen as close to $0$ as possible.
\end{rmq}

\begin{thm}
Under Assumption \ref{ass_SDE}, for all $p \in [2,+\infty)$ and for all $T \in (0,+\infty)$, there exists a unique process $X^x \in \Lrm_\Pcal^p\left(\Omega, \Ccal\left([0,T],\IR^d\right)\right)$ strong solution to (\ref{eq:SDE}).
\label{thm:solution_SDE}
\end{thm}

\begin{proof}~\\
See Theorem 7.4 of \cite{DaPrato92}. The main idea for the existence is to use a fix point theorem.
\end{proof}

\begin{prop}
Under Assumption \ref{ass_SDE}, for all $p \in (0,+\infty)$, and $T > 0$, we have:
\[ \sup_{0 \leq t \leq T} \IE\left[\left|X_t^x\right|^p\right] \leq \IE\left[\sup_{0 \leq t \leq T} \left|X_t^x\right|^p\right] \leq C \left(1+|x|^p\right),\]
where $C$ only depends on $p$, $T$, $r_1$, $r_2$, $\xi_1$ and $\xi_2$.
\label{majo_X_Sp}
\end{prop}

\begin{proof}~\\
This is a straightforward consequence of Burkholder-Davis-Gundy's inequality and Gronwall's lemma.
\end{proof}

\begin{prop}
Suppose that Assumption \ref{ass_SDE} holds true. Let $p \in (0,+\infty)$, $T>0$ and $\gamma : \IR_+ \times \IR^d \rightarrow \IR^d$ a bounded function. Thanks to Girsanov theorem, the process $\dst \Wtilde_t^x = W_t - \int_0^t \gamma\left(s,X_s^x\right) \, \drm s$ is a Brownian motion under the probability $\IPtilde^{x,T}$, on the interval $[0,T]$. 
When $\dst \sqrt{r_2} \|\gamma\|_\infty + \frac{(p \vee 2)-1}{2} r_2 < \eta_2$, we get: \vspace{-0.2mm}
\[ \sup_{T \geq 0} \widetilde{\IE}^{x,T}\left[\left|X_T^x\right|^p\right] \leq C\left(1+|x|^p\right),\]
where $C$ only depends on $p$, $\eta_1$, $\eta_2$, $r_1$, $r_2$ and $\|\gamma\|_\infty$.
\label{prop:sup_fini}
\end{prop}

\begin{proof}
The case $p = 2$ is easier, we will treat the case $p > 2$.\\
We will suppose that we have $\sqrt{r_2} \|\gamma\|_\infty + (p-1)\dfrac{r_2}{2} < \eta_2$.\\
By Itô's formula, we get (where $0 \leq t \leq T$):
\begin{align*}
\frac{\drm}{\drm t}& \widetilde{\IE}^{x,T}\left[\left|X_t^x\right|^p\right] \\
  &= p \widetilde{\IE}^{x,T}\left[\left|X_t^x\right|^{p-2} \langle X_t^x, \Xi\left(t,X_t^x\right) + \sigma\left(X_t^x\right) \gamma\left(t, X_t^x\right)\rangle\right] + \frac{p}{2} \widetilde{\IE}^{x,T}\left[\left|X_t^x\right|^{p-2} \left|\sigma\left(X_t^x\right)\right|_F^2\right] \\   
  &\pushright{+ \frac{p(p-2)}{2} \widetilde{\IE}^{x,T}\left[\left|X_t^x\right|^{p-4} \sum_{i=1}^d \left(\sum_{j=1}^d \left(X_t^x\right)_j \sigma\left(X_t^x\right)_{i,j}\right)^2\right]} \\
  &\leq p \widetilde{\IE}^{x,T}\left[\eta_1 \left|X_t^x\right|^{p-2} -\eta_2 \left|X_t^x\right|^p + \left|X_t^x\right|^{p-1} \left|\sigma\left(X_t^x\right)\right|_F \left|\gamma\left(t, X_t^x\right)\right|\right] + \frac{p}{2} \widetilde{\IE}^{x,T}\left[\left|X_t^x\right|^{p-2} \left|\sigma\left(X_t^x\right)\right|_F^2 \right] \\
  &\pushright{+ \frac{p(p-2)}{2} \widetilde{\IE}^{x,T}\left[\left|X_t^x\right|^{p-2} \left|\sigma\left(X_t^x\right)\right|_F^2 \right]} \\
  &\leq \left(-p\eta_2 + p \sqrt{r_2} \|\gamma\|_\infty + \frac{p(p-1)}{2} r_2\right) \widetilde{\IE}^{x,T}\left[\left|X_t^x\right|^p\right] + p\sqrt{r_1} \|\gamma\|_\infty \widetilde{\IE}^{x,T}\left[\left|X_t^x\right|^{p-1}\right] \\
  & \hspace{10.1cm}+ \left(p \eta_1 + \frac{p(p-1)}{2} r_1\right) \widetilde{\IE}^{x,T}\left[\left|X_t^x\right|^{p-2}\right].
\end{align*}
But, using Young's inequality, we can show, for every $\varepsilon > 0$:
\[ \frac{\drm}{\drm t} \widetilde{\IE}^{x,T}\left[\left|X_t^x\right|^p\right] \leq \lambda_\varepsilon \widetilde{\IE}^{x,T}\left[\left|X_t^x\right|^p\right] + \sqrt{r_1} \|\gamma\|_\infty \varepsilon^{-p} + \left(2\eta_1 + (p-1)r_1\right) \varepsilon^{-\frac{p}{2}}. \]
We set $\lambda_\varepsilon := -p\eta_2 + p\sqrt{r_2}\|\gamma\|_\infty + \frac{p(p-1)}{2} r_2 + (p-1) \sqrt{r_1} \|\gamma\|_\infty \varepsilon^{\frac{p}{p-1}} + \left((p-2) \eta_1 + \frac{(p-2)(p-1)}{2} r_1\right) \varepsilon^{\frac{p}{p-2}}$ (for $\varepsilon$ small enough) and this quantity is negative. Hence, for $\varepsilon$ small enough,
\begin{align*}
\frac{\drm}{\drm t} \left(\erm^{-\lambda_\varepsilon t} \widetilde{\IE}^{x,T}\left[\left|X_t^x\right|^p\right]\right)
  &\leq \erm^{-\lambda_\varepsilon t} \left( \sqrt{r_1} \|\gamma\|_\infty \varepsilon^{-p} + \left(2\eta_1 + (p-1)r_1\right) \varepsilon^{-\frac{p}{2}} \right) \\
\widetilde{\IE}^{x,T}\left[\left|X_T^x\right|^p\right]
  &\leq |x|^p + \frac{1}{\left|\lambda_\varepsilon\right|} \left( \sqrt{r_1} \|\gamma\|_\infty \varepsilon^{-p} + \left(2\eta_1 + (p-1)r_1\right) \varepsilon^{-\frac{p}{2}} \right).
\end{align*}
We have been able to conclude the case $p > 2$ because $\lambda_\varepsilon < 0$. When $p \in (0,2)$, and under the Assumption \linebreak $\dst \sqrt{r_2} \|\gamma\|_\infty + \frac{r_2}{2} < \eta_2$, we have $\dst \sup_{T \geq 0} \widetilde{\IE}^{x,T}\left[\left|X_T^x\right|^p\right] \leq C\left(1+|x|^p\right)$. We use Jensen's inequality (with the concavity of $x \mapsto |x|^{\frac{p}{2}}$) and the inequality $\left(1+z\right)^\alpha \leq 1+z^\alpha$ for $z \geq 0$ and $\alpha \in (0,1)$.
\end{proof}

\begin{prop}
Under Assumption \ref{ass_SDE}, the process $X^x$ is irreducible, that is to say:
\[ \forall t > 0,\, \forall x,z \in \IR^d,\, \forall r > 0,\, \IP\left(\left|X_t^x - z\right| < r\right) > 0.\]
\end{prop}

\begin{proof}~ \\
The proof is very similar to Theorem 7.3.1 of \cite{DaPrato96}.
\end{proof}

\begin{thm}
Suppose Assumption \ref{ass_SDE} holds true. Let $\mu \geq 2$ be such that $\dfrac{\mu-1}{2} r_2 < \eta_2$ and $\phi : \IR^d \rightarrow \IR^d$ a measurable function, satisfying $|\phi(x)| \leq c_\phi \left(1+|x|^\mu\right)$.
Then we have,
\[\forall x,y \in \IR^d,\, \forall t \geq 0,\, 
\left| \Pcal_t[\phi](x) - \Pcal_t[\phi](y)\right| \leq \chat c_\phi \left(1+|x|^\mu+|y|^\mu\right) \erm^{-\nuhat t},\]
where $\Pcal_t[\phi](x) = \IE\left[\phi\left(X_t^x\right)\right]$ (Kolmogorov semigroup of $X$) and $\chat$ and $\nuhat$ only depend on $\eta_1$, $\eta_2$, $r_1$, $r_2$ and $\mu$.
\label{thm:coupling_estimate}
\end{thm}

\begin{proof}~\\
This proof is based on the application of the Theorem A.2 of \cite{Mattingly01}. Hypoellipticity is a consequence of the boundedness of $\sigma^{-1}$. In order to show that the monomial function $V: x \mapsto |x|^\mu$ is a Lyapunov function, it suffices to show that it satisfies for all $x \in \IR^d$: $LV(x) <-aV(x)+b \ind_C(x)$, where $a$ and $b$ are positive constants, $C$ a compact set and $L$ the generator of $X$. We have:
\begin{align*}
LV(x) 
   &= \mu |x|^{\mu-2} \langle x, \Xi(0,x) \rangle + \frac{\mu}{2} |x|^{\mu-2} |\sigma(x)|_F^2 + \frac{\mu(\mu-2)}{2} |x|^{\mu-4} \sum_{i=1}^d \left(\sum_{j=1}^d x_j \sigma(x)_{i,j}\right)^2 \\
   &\leq \left(-\mu \eta_2 + \frac{\mu(\mu-1)}{2} r_2\right) |x|^\mu + \left(\mu\eta_1 + \frac{\mu(\mu-1)}{2} r_1\right) |x|^{\mu-2}
\end{align*}
One can check that the required inequality can be obtained by setting $C = \overline{\Bcal}(0,R)$, $\dst R = \sqrt{\frac{\eta_1 + \frac{\mu-1}{2} r_1}{\eta_2 - \frac{\mu-1}{2} r_2}}+1$, \linebreak $\dst a = \mu\eta_2 - \frac{\mu(\mu-1)}{2} r_2 - \frac{b}{R^2}$ and $\dst b = \mu \eta_1 + \frac{\mu(\mu-1)}{2} r_1$.\\
The last thing we need to apply Theorem A.2 of \cite{Mattingly01} is a consequence of  irreducibility of $X^x$. Then, we get:
\[ \left|\Pcal_t[\phi](x) - \Pcal_t[\phi](y)\right| \leq \chat c_\phi\left(1+|x|^\mu+|y|^\mu\right) \erm^{- \nuhat t}.\]
See also \cite{MT1} and \cite{MT3} to prove that $\chat$ and $\nuhat$ only depend on $\eta_1$, $\eta_2$, $r_1$, $r_2$ and $\mu$.
\end{proof}

\begin{ass}
We can write $\Xi(t,x) = \xi(x) + \rho(t,x)$ where $\xi$ and $\rho$ satisfy the following:
\begin{itemize}
\item $\rho$ is a bounded function;
\item $\xi$ is Lipschitz continuous and satisfies $\langle \xi(x),x\rangle \leq \eta_1 - \eta_2 |x|^2$, with $\eta_1$ and $\eta_2$ two positive constants;
\item $\rho$ is the pointwise limit of a sequence $\left(\rho_n\right)$ of $\Ccal^1$ functions, with bounded derivatives w.r.t. $x$ and uniformly bounded by $\|\rho\|_\infty$.
\end{itemize}
As before, we still require on $\sigma$ the following:
\begin{itemize}
\item $\sigma : \IR^d \rightarrow \GL_d(\IR)$ is Lipschitz continuous;
\item the function $x \mapsto \sigma(x)^{-1}$ is bounded;
\item $\forall x \in \IR^d,\, |\sigma(x)|_F^2 \leq r_1 + r_2|x|^2$.
\end{itemize}
\label{ass_SDE_nonLip}
\end{ass}

\begin{cor}
Suppose Assumption \ref{ass_SDE_nonLip} holds true. Let $\mu \geq 2$ be such that $\sqrt{r_2} \|\rho\|_\infty + \dfrac{\mu-1}{2} r_2 < \eta_2$ and $\phi : \IR^d \rightarrow \IR^d$ measurable with $|\phi(x)| \leq c_\phi \left(1+|x|^\mu\right)$.
Then we have,
\[\forall x,y \in \IR^d,\, \forall t \geq 0,\, 
\left| \Pcal_t[\phi](x) - \Pcal_t[\phi](y)\right| \leq \chat c_\phi \left(1+|x|^\mu+|y|^\mu\right) \erm^{-\nuhat t},\]
where $\chat$ and $\nuhat$ only depend on $\eta_1$, $\eta_2$, $r_1$, $r_2$, $\mu$ and $\|\rho\|_\infty$.
\label{cor:semigroup}
\end{cor}

\begin{proof}~\\
We set $\sigma_n$ a function close to $\sigma$ on the centered ball of $\IR^d$ of radius $n$, equal to $\Irm_d$ outside the centered ball of radius $n+1$ and of class $\Ccal^1$ with bounded derivatives on $\IR^d$; on the ring between the radius $n$ and $n+1$, $\sigma_n$ is chosen in such a way that $\sigma^{-1} \sigma_n$ is bounded, independently from $n$.
This way, the function $\Xi_n : (t,x) \mapsto \xi(x) + \sigma_n(x) \rho_n(t,x)$ is Lipschitz continuous w.r.t. $x$.
We denote $X^{n,x}$ the solution of the SDE:
\[ \left\{ \begin{array}{l}
           \drm X_t^{n,x} = \Xi_n\left(t, X_t^{n,x}\right) \, \drm t + \sigma\left(X_t^{n,x}\right) \, \drm W_t, \\
           X_0^{n,x} = x.
           \end{array} \right. \]
We can write, for every $\varepsilon > 0$:
\[ \langle \Xi_n(t,x), x \rangle \leq \eta_1 - \eta_2 |x|^2 + \left(\sqrt{r_1} + \sqrt{r_2}|x|\right) \|\rho\|_\infty |x| \leq \left[\eta_1 + \frac{\|\rho\|_\infty^2 r_1}{2 \varepsilon^2}\right] - \left[\eta_2 - \sqrt{r_2} \|\rho\|_\infty - \frac{\varepsilon^2}{2}\right] |x|^2.\]
Then, when $\varepsilon$ is small enough, for any $\phi$ with polynomial growth of degree $\mu$, Theorem \ref{thm:coupling_estimate} tells us that $\chat_n$ and $\nuhat_n$ are independent of $n$, and we have: 
\[ \exists \chat, \nuhat >0,\, \forall n \in \IN^*,\, \forall x,y \in \IR^d,\, \forall t \geq 0,\, \left| \IE\left[\phi\left(X_t^{n,x}\right)\right] - \IE\left[\phi\left(X_t^{n,y}\right)\right] \right| \leq \chat c_\phi \left(1+|x|^\mu+|y|^\mu\right) \erm^{-\nuhat t}.\]
Our goal is to take the limit ; let us show that $\IE\left[\phi\left(X_t^{n,x}\right)\right] \limit{n}{\infty} \IEtilde\left[\phi\left(X_t^x\right)\right]$.
Let $U$ be the solution of the SDE 
\[ U_t^x = x + \int_0^t{\xi\left(U_s^x\right) \, \drm s} + \int_0^t{\sigma\left(U_s^x\right)\, \drm W_s}. \]
We can write:
$\dst X_t^{n,x} = x + \int_0^t{\xi\left(X_s^{n,x}\right) \, \drm s} + \int_0^t{\sigma\left(X_s^{n,x}\right) \, \drm W_s^{(n)}}$,
where $\dst W_t^{(n)} = W_t + \int_0^t{\sigma\left(X_s^{n,x}\right)^{-1} \sigma_n\left(X_s^{n,x}\right)\rho_n\left(s,X_s^{n,x}\right) \, \drm s}$ is a Brownian motion under the probability $\IP^{(n)} = p_T^n\left(X^{n,x}\right) \IP$ on $[0,T]$ and where
\[ p_t^n\left(X^{n,x}\right) = \exp\left(-\int_0^t{\langle \sigma\left(X_s^{n,x}\right)^{-1}\sigma_n\left(X_s^{n,x}\right)\rho_n\left(s,X_s^{n,x}\right), \drm W_s \rangle} - \frac{1}{2} \int_0^t{\left|\sigma\left(X_s^{n,x}\right)^{-1} \sigma_n\left(X_s^{n,x}\right) \rho_n\left(s,X_s^{n,x}\right)\right|^2 \, \drm s}\right). \]
Similarly,
$\dst X_t^x = x + \int_0^t{\xi\left(X_s^x\right) \, \drm s} + \int_0^t{\sigma\left(X_s^x\right) \, \drm W_s^{(\infty)}}$,
where $\dst W_t^{(\infty)} = \Wtilde_t + \int_0^t{\rho\left(s,X_s^x\right)\, \drm s}$ is a Brownian motion under the probability $\IP^{(\infty)} = p_T^\infty\left(X^x\right) \IPtilde$ on $[0,T]$ and where
\[p_t^\infty\left(X^x\right) = \exp\left(-\int_0^t{\langle \rho\left(s,X_s^x\right), \drm \Wtilde_s \rangle} - \frac{1}{2} \int_0^t{\left|\rho\left(s,X_s^x\right)\right|^2 \, \drm s}\right).\]
By uniqueness in law of the solutions of the SDEs, we get the equalities:
\[ \IE\left[\phi\left(X_t^{n,x}\right)\right] = \IE^{(n)}\left[p^n_t\left(X^{n,x}\right)^{-1} \phi\left(X_t^{n,x}\right)\right] = \IE\left[p^n_t\left(U^x\right)^{-1} \phi\left(U_t^x\right)\right] ;\]
\[ \IEtilde\left[\phi\left(X_t^x\right)\right] = \IE^{(\infty)}\left[p^\infty_t\left(X^x\right)^{-1} \phi\left(X_t^x\right)\right] = \IE\left[p^\infty_t\left(U^x\right)^{-1} \phi\left(U_t^x\right)\right]. \]
But $\rho_n(t,x) \limite{}{n}{\infty} \rho(t,x)$, so we have $p_t^n\left(U^x\right)^{-1} \limite{\IP}{n}{\infty} p_t^\infty\left(U^x\right)^{-1}$.
We just have to show that the sequence $\left(p_t^n\left(U^x\right)^{-1}\right)_{n \in \IN^*}$ is uniformly integrable, in order to show that $\IE\left[\phi\left(X_t^{n,x}\right)\right] \limit{n}{\infty} \IEtilde\left[\phi\left(X_t^x\right)\right]$.
We have:
\begin{align*}
\IE&\left[\left(p_t^n\left(U^x\right)^{-1}\right)^2\right] \\
  &= \IE\left[\exp\left(2 \int_0^t{\langle \sigma\left(X_s^{n,x}\right)^{-1} \sigma_n\left(X_s^{n,x}\right) \rho_n\left(s,X_s^{n,x}\right),\drm W_s\rangle} + \int_0^t{\left|\sigma\left(X_s^{n,x}\right)^{-1} \sigma_n\left(X_s^{n,x}\right) \rho_n\left(s,X_s^{n,x}\right)\right|^2 \, \drm s}\right)\right] \\
	&\leq \IE\left[\exp\left(\int_0^t{\langle 4 \sigma\left(X_s^{n,x}\right)^{-1} \sigma_n\left(X_s^{n,x}\right) \rho_n\left(s,X_s^{n,x}\right) , \drm W_s \rangle} - \frac{1}{2} \int_0^t{\left|4\sigma\left(X_s^{n,x}\right)^{-1} \sigma_n\left(X_s^{n,x}\right) \rho_n\left(s,X_s^{n,x}\right)\right|^2 \, \drm s}\right)\right]^{\frac{1}{2}} \\
	&\pushright{\IE\left[\exp\left(10 \int_0^t{\left|\sigma\left(X_s^{n,x}\right)^{-1} \sigma_n\left(X_s^{n,x}\right) \rho_n\left(s,X_s^{n,x}\right)\right|^2 \, \drm s}\right)\right]^{\frac{1}{2}}} \\
	&\leq \exp\left(5t \left(d^2 + \left\|\sigma^{-1}\right\|_\infty^2\right) \|\rho\|_\infty\right) < \infty.
\end{align*}
\end{proof}

\section{The EBSDE}

We consider the following EBSDE:
\begin{equation}
\forall 0 \leq t \leq T < \infty,\, Y_t^x = Y_T^x + \int_t^T \left[\psi\left(X_s^x, Z_s^x\right) - \lambda\right] \, \drm s - \int_t^T Z_s^x \, \drm W_s
\label{eq:EBSDE}
\end{equation}
and we make on $\psi : \IR^d \times \left(\IR^d\right)^* \rightarrow \IR$ and the SDE (\ref{eq:SDE}) satisfied by $X^x$ the following assumptions.

\begin{ass}
\begin{itemize}
\item $\forall x \in \IR^d,\, |\psi(x,0)| \leq M_\psi (1+|x|)$;
\item $\forall x, x' \in \IR^d,\, \forall z,z' \in \IR^d,\, \left|\psi(x,z) - \psi\left(x',z'\right)\right| \leq K_x \left|x-x'\right| + K_z \left|z\sigma(x)^{-1}-z'\sigma(x')^{-1}\right|$;
\item $\Xi : \IR^d \rightarrow \IR^d$ is Lipschitz continuous, $|\Xi(x)| \leq \xi_1 + \xi_2 |x|$;
\item $\Xi(x) = \xi(x) + b(x)$, with $\xi$ dissipative and locally Lipschitz and $b$ bounded, and $\langle \Xi(x),x \rangle \leq \eta_1 - \eta_2|x|^2$ for two positive constants $\eta_1$, $\eta_2$;
\item $\sigma : \IR^d \rightarrow \GL_d(\IR)$ is Lipschitz continuous;
\item $x \mapsto \sigma(x)^{-1}$ is bounded and $|\sigma(x)|_F^2 \leq r_1 + r_2|x|^2$;
\item $\dst \sqrt{r_2} K_z \left\|\sigma^{-1}\right\|_\infty + \frac{r_2}{2} < \eta_2$.
\end{itemize}
\label{ass_EBSDE}
\end{ass}

\begin{rmq}
In most papers, the function $\psi$ is assumed to be Lipschitz continuous. We make a slight modification of this assumption in order to have some information about the second and third behaviour (see \cite{PDE}). But $\psi$ is still Lipschitz continuous w.r.t. $z$, with a constant equal to $K_z \left\|\sigma^{-1}\right\|_\infty$. Moreover, $\psi$ is still continuous w.r.t. $x$.
\end{rmq}~\\

\subsection{Existence of a solution}

Theorem \ref{thm:solution_SDE} ensures that the process $X^x$ is well defined. We introduce a new parameter $\alpha > 0$ and we consider a new BSDE of infinite horizon:
\begin{equation}
\forall 0 \leq t \leq T < \infty,\, Y_t^{\alpha,x} = Y_T^{\alpha,x} + \int_t^T \left[\psi\left(X_s^x, Z_s^{\alpha,x}\right) - \alpha Y_s^{\alpha,x}\right] \, \drm s - \int_t^T Z_s^{\alpha,x} \, \drm W_s.
\label{eq:BSDE}
\end{equation}

\begin{lem}
Under Assumption \ref{ass_EBSDE}, for every $x \in \IR^d$ and $\alpha > 0$, there exists a unique solution $\left(Y^{\alpha,x}, Z^{\alpha,x}\right)$ to the BSDE (\ref{eq:BSDE}), such that $Y^{\alpha,x}$ is a continuous process bounded in $\Lrm^1$ and $Z^{\alpha,x} \in \Lrm^2_{\Pcal, \text{loc}}\left(\Omega, \Lrm^2\left(0, \infty; \left(\IR^d\right)^*\right)\right)$. Also, for every $t \geq 0$, $\left|Y_t^{\alpha,x}\right| \leq \dfrac{C}{\alpha} \left(1+\left|X_t^x\right|\right)$, $\IP$-a.s., where $C$ only depends on $M_\psi$, $\eta_1$, $\eta_2$, $r_1$, $r_2$, $K_z$ and $\left\|\sigma^{-1}\right\|_\infty$.
The function $v^\alpha : x \mapsto Y_0^{\alpha,x}$ is continuous, and for every $t \geq 0$, $Y_t^{\alpha,x} = v^\alpha\left(X_t^x\right)$ $\IP$-a.s.
\label{lem:EU_BSDE}
\end{lem}

\begin{proof}~\\
For the upperbound for $Y^{\alpha,x}$, see Theorem 2.1 of \cite{Royer04}. The main difference is that we do not require $\psi(\bullet,0)$ to be bounded. Uniqueness only needs boundedness in $\Lrm^1$ and not almost surely boundedness. Also, in the proof of existence, the constant $C$ is harder to write, because we do the same as in the end of the proof of proposition \ref{prop:sup_fini}. The continuity of $v^\alpha$ is a consequence of a straightforward adaptation of the Theorem 2.1 of \cite{Royer04} when $\psi(\bullet,0)$ is no more assumed to be bounded and proposition 2.1 of \cite{EPQ97}. The representation of $Y^{\alpha,x}$ by $X^x$ comes from the Lemma 4.4 of \cite{Royer04}.
\end{proof}

\begin{lem}
Under Assumption \ref{ass_EBSDE}, for every $\alpha \in (0,1]$, we have:
\begin{equation}
\forall x,x' \in \IR^d,\, \left|v^\alpha(x) - v^\alpha(x') \right| \leq C \left(1+|x| + |x'| \right) |x-x'|,
\label{eq:valpha_lipschitz}
\end{equation}
where $C$ only depends on $K_x$, $K_z$, $M_\psi$, $\alpha$, $\eta_1$, $\eta_2$, $r_1$, $r_2$, $\left\|\sigma^{-1}\right\|_\infty$, $\xi_1$, $\|\Xi\|_\lip$ and $\|\sigma\|_\lip$.
\label{lem:valpha_lipschitz}
\end{lem}

\begin{proof}~\\
See Theorem \ref{thm:u_Lipschitz}, given in the appendix, with $f:(x,y,z) \mapsto \psi(x,z) - \alpha y$ and $g = v^\alpha$. It uses the estimate of Lemma \ref{lem:EU_BSDE}.
\end{proof}

\begin{rmq}
\begin{enumerate}
\item If $\sigma$ and $\psi(\bullet,0)$ are bounded and $\Xi$ and $\sigma$ are of class $\Ccal^1$, then, the Theorem 3.2 of \cite{HT07} tells us that $v^\alpha$ is of class $\Ccal^1$. Then, we can apply the Theorem 3.1 of \cite{MZ02} and write $Z_t^{\alpha,x} = \partial_x v^\alpha\left(X_t^x\right) \sigma\left(X_t^x\right)$.
\item The Lipschitz constant given by the last lemma goes to infinity as $\alpha$ appoximates $0$.
\end{enumerate}
\end{rmq}

\begin{lem}
Let $\psi : \IR^d \times \left(\IR^d\right)^* \rightarrow \IR$ continuous w.r.t. the first variable and Lipschitz continuous w.r.t. the second one.
Let $\zeta, \zeta' : \IR_+ \times \IR^d \rightarrow \left(\IR^d \right)^*$ be such that for every $t \geq 0$, $\zeta(t,\bullet)$ and $\zeta'(t,\bullet)$ are continuous, and set:
\begin{equation*}
\widetilde{\Gamma}(t,x) = \left\{ \begin{array}{ll}
  \dfrac{\psi(x,\zeta(t,x)) - \psi(x,\zeta'(t,x))}{|\zeta(t,x)-\zeta'(t,x)|^2} (\zeta(t,x)-\zeta'(t,x)), & \text{if } \zeta(t,x) \neq \zeta'(t,x), \\
  0, & \text{if } \zeta(t,x) = \zeta'(t,x).
\end{array} \right.
\end{equation*}
Then, there exists a uniformly bounded sequence of $\Ccal^1$ functions w.r.t. $x$ with bounded derivatives $\left(\widetilde{\Gamma}_n\right)_{n \geq 1}$ (\emph{i.e.} for all $n$, $\widetilde{\Gamma}_n$ has bounded derivatives w.r.t. $x$ -- the bound of derivatives can depend on $n$ -- and $\sup_{n \geq 1} \left\|\widetilde{\Gamma}_n(t,\bullet)\right\|_\infty < \infty$ for every $t \geq 0$), such that $\widetilde{\Gamma}_n \underset{n \rightarrow \infty}{\longrightarrow} \widetilde{\Gamma}$ pointwise.
\label{lem:approx_lipschitz}
\end{lem}

\begin{proof}~\\
See the Lemma 3.7 of \cite{PDE}: we can approximate the Lipschitz functions by $\Ccal^1$ functions with bounded derivatives and construct a new sequence having the required regularity.
\end{proof}

\begin{prop}
Under Assumption \ref{ass_EBSDE}, there exists a constant $C$, such that for every $\alpha \in (0,1]$, we have:
\[ \forall x,x' \in \IR^d,\, \left|v^\alpha(x) - v^\alpha(x')\right| \leq C \left(1+|x|^2+|x'|^2\right).\]
The constant $C$ only depends on $\eta_1$, $\eta_2$, $r_1$, $r_2$, $K_z$, $\left\|\sigma^{-1}\right\|_\infty$ and $M_\psi$.
\label{prop:valpha_quadratic}
\end{prop}

\begin{proof}~
\begin{enumerate}
\item We approximate $\sigma$ by a sequence $\left(\sigma^\varepsilon\right)_{\varepsilon > 0}$ of functions satisfying:
      \begin{itemize}
      \item $\sigma^\varepsilon$ converges pointwise towards $\sigma$ over $\IR^d$;
      \item $\sigma^\varepsilon$ is bounded (the bound can depend on $\varepsilon$) and $\left|\sigma^\varepsilon(x)\right|_F^2 \leq r_1 + r_2|x|^2$;
      \item $\sigma^\varepsilon$ is of class $\Ccal^1$ and $\left\|\sigma^\varepsilon\right\|_\lip \leq \|\sigma\|_\lip$;
      \item $x \mapsto \sigma^\varepsilon(x)^{-1}$ is bounded, the bound is independent of $\varepsilon$.
      \end{itemize}
      We also approximate $\Xi$ by a sequence $\left(\Xi^\varepsilon\right)_{\varepsilon > 0}$ of $\Ccal^1$ functions which converges uniformly. One can check that the functions $\Xi^\varepsilon$ are ``uniformly weakly dissipative'', because the functions $\Xi^\varepsilon - \xi$ are uniformly bounded. Moreover, $\psi$ is approximated by a sequence $\left(\psi^\varepsilon\right)$ of functions satisfying:
      \begin{itemize}
      \item $\left|\psi^\varepsilon(x,z) - \psi^\varepsilon(x',z')\right| \leq K_x |x-x'| + K_z \left|z \sigma^\varepsilon(x)^{-1} - z' \sigma^\varepsilon(x')^{-1}\right|$;
      \item $\psi^\varepsilon(\bullet,0)$ is bounded;
      \item $\left|\psi^\varepsilon(x,0)\right| \leq M_\psi (1+|x|)$;
      \item $\left(\psi^\varepsilon\right)$ converges pointwise towards $\psi$.
      \end{itemize}
      We consider the BSDE:
      \[ Y_t^{\varepsilon,\alpha,x} = Y_T^{\varepsilon,\alpha,x} + \int_t^T \left\{ \psi^\varepsilon\left( X_s^{\varepsilon,x}, Z_s^{\varepsilon,\alpha,x}\right) - \alpha Y_s^{\varepsilon,\alpha,x}\right\} \, \drm s - \int_t^T Z_s^{\varepsilon,\alpha,x} \, \drm W_s ,\]
      where the process $X^{\varepsilon,x}$ satisfies the following equation:
      \[ X_t^{\varepsilon,x} = x + \int_0^t \Xi^\varepsilon\left(X_s^{\varepsilon,x}\right) \, \drm s + \int_0^t \sigma^\varepsilon\left(X_s^{\varepsilon,x}\right) \, \drm W_s.\]
      
      This BSDE has a unique solution (see Lemma \ref{lem:EU_BSDE}), and $\dst \left|Y_t^{\varepsilon,\alpha,x}\right| \leq \frac{C}{\alpha}\left(1+\left|X_t^x\right|\right)$ $\IP$-a.s. and for every $t \geq 0$. Thanks to the previous remark, we can write $Z_t^{\varepsilon,\alpha,x} = \zeta^{\varepsilon,\alpha} \left(X_t^{\varepsilon,x}\right)$ $\IP$-a.s. and for a.e. $t \geq 0$, and $\zeta^{\varepsilon,\alpha}$ is continuous. We define:
      \[ \Gamma^{\varepsilon,\alpha}(x) = \begin{cases} 
        \dfrac{\psi^\varepsilon\left(x,\zeta^{\varepsilon,\alpha}(x)\right) - \psi^\varepsilon(x,0)}{\left|\zeta^{\varepsilon,\alpha}(x)\right|^2} \zeta^{\varepsilon, \alpha}(x), & \text{if } \zeta^{\varepsilon,\alpha}(x) \neq 0, \\ 0, & \text{otherwise.} \end{cases} \]
        
      Then, thanks to Lemma \ref{lem:approx_lipschitz}, we can approximate $\Gamma^{\varepsilon,\alpha}$ in such a way that we can use Corollary \ref{cor:semigroup}.
      
      We can rewrite:
      \[ -\drm Y_t^{\varepsilon,\alpha,x} = \left\{\psi^\varepsilon\left(X_t^{\varepsilon,x},0\right) + Z_t^{\varepsilon,\alpha,x} \Gamma^{\varepsilon,\alpha}\left(X_t^{\varepsilon,x}\right)^*-\alpha Y_t^{\varepsilon,\alpha,x}\right\} \, \drm t - Z_t^{\varepsilon,\alpha,x} \, \drm W_t.\]
      
      But $\Gamma^{\varepsilon,\alpha}$ is bounded by $K_z \left\|\sigma^{-1}\right\|_\infty$, and there exists a probability $\widehat{\IP}^{\varepsilon,\alpha,x,T}$ under which $\dst \widehat{W}^{\varepsilon,\alpha,x}_t = W_t - \int_0^t \Gamma^{\varepsilon,\alpha}\left(X_s^{\varepsilon,x}\right)^* \, \drm s$ is a Brownian motion on $[0,T]$. Finally, we get the equality: $\dst v^{\varepsilon,\alpha}(x) = \widehat{\IE}^{\varepsilon,\alpha,x,T}\left[\erm^{-\alpha T} v^{\varepsilon,\alpha}\left(X_T^{\varepsilon,x}\right) + \int_0^T \erm^{-\alpha s} \psi^\varepsilon\left(X_s^{\varepsilon,x},0\right)\, \drm s\right]$.
      
      On the one hand, using proposition \ref{prop:sup_fini}: $\dst \left|\widehat{\IE}^{\varepsilon,\alpha,x,T}\left[\erm^{-\alpha T} v^{\varepsilon,\alpha}\left(X_T^{\varepsilon,x}\right)\right]\right| \leq \erm^{-\alpha T} \frac{C}{\alpha} \widehat{\IE}^{\varepsilon,\alpha,x,T}\left[\left|X_T^{\varepsilon,x}\right|\right] \underset{T \rightarrow \infty}{\longrightarrow} 0$.
      
      On the other hand, $X^\varepsilon$ satisfies the following SDE under $\widehat{\IP}^{\varepsilon,\alpha,x,T}$:
      \[ \drm X_t^{\varepsilon,x} = \left[\Xi^\varepsilon\left(X_t^{\varepsilon,x}\right) + \sigma^\varepsilon\left(X_t^{\varepsilon,x}\right)\Gamma^{\varepsilon,\alpha}\left(X_t^{\varepsilon,x}\right)^*\right] \, \drm t + \sigma^\varepsilon\left(X_t^{\varepsilon,x}\right) \, \drm \widehat{W}^{\varepsilon,\alpha,x}_t.\]
      
      Thanks to corollary \ref{cor:semigroup}, we get:
      \[ \left|\widehat{\IE}^{\varepsilon,\alpha,x,T} \left[\psi^\varepsilon\left(X_t^{\varepsilon,x},0\right)\right] - \widehat{\IE}^{\varepsilon,\alpha,x',T}\left[\psi^\varepsilon\left(X_t^{\varepsilon,x'},0\right)\right]\right| \leq 2 M_\psi \chat \erm^{-\nuhat t}\left(1+|x|^2+|x'|^2\right),\]
      where $\chat$ and $\nuhat$ only depend on $\eta_1$, $\eta_2$, $r_1$, $r_2$, $K_z$ and $\left\|\sigma^{-1}\right\|_\infty$. As a consequence, we get:
      \begin{align}
      \left|v^{\varepsilon,\alpha}(x)-v^{\varepsilon,\alpha}(x')\right|
      &= \lim_{T \rightarrow \infty} \left|\int_0^T \erm^{-\alpha t} \left(\widehat{\IE}^{\varepsilon,\alpha,x,t}\left[\psi^\varepsilon\left(X_t^{\varepsilon,x},0\right)\right] - \widehat{\IE}^{\varepsilon,\alpha,x',t}\left[\psi^\varepsilon\left(X_t^{\varepsilon,x'},0\right)\right]\right) \, \drm t\right| \nonumber\\
      &\leq \int_0^\infty \erm^{-\alpha t} \left|\widehat{\IE}^{\varepsilon,\alpha,x,t}\left[\psi^\varepsilon\left(X_t^{\varepsilon,x},0\right)\right] - \widehat{\IE}^{\varepsilon,\alpha,x',t}\left[\psi^\varepsilon\left(X_t^{\varepsilon,x'},0\right)\right]\right| \, \drm t \nonumber \\
      &\leq 2 \int_0^\infty \erm^{-\alpha t} \chat M_\psi \erm^{-\nuhat t}\left(1+|x|^2+|x'|^2\right) \, \drm t 
       \leq 2 \frac{\chat M_\psi}{\nuhat} \left(1+|x|^2+|x'|^2\right).  
      \label{eq:vlocLip} 
      \end{align}

\item Now, our goal is to take the limit when $\varepsilon \rightarrow 0$. Let $D$ be a dense and countable subset of $\IR^d$. By a diagonal argument, there exists a positive sequence $\left(\varepsilon_n\right)_n$ such that $\left(v^{\varepsilon_n,\alpha}\right)_n$ converges pointwise over $D$ to a function $\vbar^\alpha$. Because the constant $C$ in equation (\ref{eq:valpha_lipschitz}) does not depend on $\varepsilon$, $\vbar^\alpha$ satisfies the same inequality. Let $\left(K_n\right)$ be a sequence of compact sets whose diameter goes to infinity. The function $\vbar^\alpha$ is uniformly continuous on $K_n \cap D$, so it has an extension $\vbar$ which is continuous on $K_n$. Passing to the limit as $n$ goes to infinity, we get a continuous function on $\IR^d$, and it is the pointwise limit of the sequence $\left(v^{\varepsilon_n,\alpha}\right)_n$. We denote $\Ybar_t^{\alpha,x} = \vbar^\alpha\left(X_t^x\right)$; we have $\dst \left|\Ybar_t^{\alpha,x}\right| \leq \frac{C}{\alpha} \left(1+ \left|X_t^x\right|\right)$. Using the Lemma 2.1 of \cite{MZ02}, we have, for every $T > 0$ :
      \[ \IE\left[\left|X^{\varepsilon_n,x} - X^x\right|^{*,2}_{0,T}\right] \leq C \left\{\left\|\Xi-\Xi^{\varepsilon_n}\right\|_\infty^2 T + \IE\left[\int_0^T \left|\sigma^{\varepsilon_n}\left(X_t^x\right) - \sigma\left(X_t^x\right) \right|^2 \, \drm t\right]\right\}.\]
      
      By dominated convergence, this quantity goes to $0$ as $n$ goes to infinity. Also by dominated convergence, we get:
      \[ \IE\left[\int_0^T \left|Y_t^{\varepsilon_n,\alpha,x} - \Ybar_t^{\alpha,x}\right|^2 \, \drm t\right] \underset{n \rightarrow \infty}{\longrightarrow} 0 \text{ and } \IE\left[\left|Y_T^{\varepsilon_n,\alpha,x}- \Ybar_T^{\alpha,x}\right|^2\right] \underset{n \rightarrow \infty}{\longrightarrow} 0.\]

\item We will show that there exists a process denoted $\Zbar^{\alpha,x}$ belonging to $\Lrm^2_{\Pcal,\loc}\left(\Omega,\Lrm^2\left(0,\infty;\left(\IR^d\right)^*\right)\right)$ which satisfies $\dst \IE\left[\int_0^T \left|Z_t^{\varepsilon_n,\alpha,x} - \Zbar_t^{\alpha,x}\right|^2 \, \drm t\right] \underset{n \rightarrow \infty}{\longrightarrow} 0$, for every $T > 0$. Indeed, $\left(\Ybar^{\alpha,x}, \Zbar^{\alpha,x}\right)$ is solution of the BSDE (\ref{eq:BSDE}) ; by uniqueness of the solution, $v^\alpha \equiv \vbar^\alpha$ and taking the limit in the equation (\ref{eq:vlocLip}) gives the result.

      Let $n \leq m \in \IN$, we define $\Ytilde = Y^{\varepsilon_n,\alpha,x} - Y^{\varepsilon_m,\alpha,x}$ and $\Ztilde = Z^{\varepsilon_n,\alpha,x} - Z^{\varepsilon_m,\alpha,x}$. 
      We have:
      \[ \drm \Ytilde_t = \alpha \Ytilde_t \, \drm t + \underbrace{\left(\psi^{\varepsilon_m}\left(X_t^{\varepsilon_m,x},Z_t^{\varepsilon_m,\alpha,x}\right) - \psi^{\varepsilon_n}\left(X_t^{\varepsilon_n,x},Z_t^{\varepsilon_n,\alpha,x}\right)\right)}_{=: \psi_t} \, \drm t + \Ztilde_t \, \drm W_t.\]
      
      Thanks to Itô's formula, we obtain:
      $ \dst \IE\left[\int_0^T \left|\Ztilde_t\right|^2 \, \drm t\right] 
      \leq \IE\left[\left|\Ytilde_T\right|^2\right] - 2 \IE\left[\int_0^T \Ytilde_t \psi_t \, \drm t\right].$
      
      But, we have: $\dst \left|\psi_t\right| \leq M_\psi \left(2+\left|X_t^{\varepsilon_m,x}\right| + \left|X_t^{\varepsilon_n,x}\right|\right) + K_z \left(\left|Z_t^{\varepsilon_m,x} \sigma^{\varepsilon_m}\left(X_t^{\varepsilon_m,x}\right)^{-1}\right| + \left|Z_t^{\varepsilon_n,x} \sigma^{\varepsilon_n}\left(X_t^{\varepsilon_n,x}\right)^{-1}\right|\right)$.
      
      So:
      \begin{align*}
      \IE\left[\int_0^T \left|\Ztilde_t\right|^2 \, \drm t\right] 
          &\leq 2 \IE\left[\left|\Ytilde_T\right|^2\right] + 2TM_\psi \IE\left[\left|\Ytilde\right|^*_{0,T} \left(2 + \left|X^{\varepsilon_m,x}\right|^*_{0,T} + \left|X^{\varepsilon_n,x}\right|^*_{0,T}\right)\right] \\
          &\pushright{\hspace{3cm} + 2 K_z \left\|\sigma^{-1}\right\|_\infty \IE\left[\left|\Ytilde\right|^*_{0,T} \int_0^T \left(\left|Z_t^{\varepsilon_m,x}\right| + \left|Z_t^{\varepsilon_n,x}\right|\right) \, \drm t\right].}
      \end{align*}
      
      Now, we need to bound $\IE\left[\left|X^{\varepsilon_n,x}\right|^{*,2}_ {0,T}\right]$ and $\dst \IE\left[\int_0^T \left|Z_t^{\varepsilon_n,\alpha,x}\right|^2 \, \drm t\right]$ independently of $n$. Thanks to proposition \ref{majo_X_Sp}, we have $\IE\left[\left|X^{\varepsilon_n,x}\right|^{*,2}_ {0,T}\right] \leq C \left(1+|x|^2\right)$, where $C$ only depends on $T$, $r_1$, $r_2$, $\xi_1$ and $\xi_2$. Also, we have:
      \[ \IE\left[\int_0^T \left|Z_t^{\varepsilon_n,\alpha,x}\right|^2 \, \drm t\right] \leq \IE\left[\left|Y_T^{\varepsilon_n,\alpha,x}\right|^2\right] + 2 \IE\left[\int_0^T \left|Y_t^{\varepsilon_n,\alpha,x} \psi^{\varepsilon_n}\left(X_t^{\varepsilon_n,x}, Z_t^{\varepsilon_n,\alpha,x}\right)\right| \, \drm t\right].\]
      But, using the estimate of Lemma \ref{lem:EU_BSDE},
      \begin{align*}
      \left|Y_t^{\varepsilon_n,\alpha,x} \psi\left(X_t^{\varepsilon_n,x}, Z_t^{\varepsilon_n,\alpha,x}\right)\right| &\leq \frac{C}{\alpha} \left(1+\left|X_t^{\varepsilon_n,x}\right|\right) \left(M_\psi\left(1+\left|X_t^{\varepsilon_n,x}\right|\right) +K_z \left\|\sigma^{-1}\right\|_\infty \left|Z_t^{\varepsilon_n,\alpha,x}\right|\right) \\
      &\leq \Ccal \left(1+\left|X_t^{\varepsilon_n,x}\right|^2\right) + \frac{1}{4} \left|Z_t^{\varepsilon_n,\alpha,x}\right|^2.
      \end{align*}
      
      Finally, $\dst \IE\left[\int_0^T \left|Z_t^{\varepsilon_n,\alpha,x}\right|^2 \, \drm t\right] \leq \Ccal' \left(1+|x|^2\right)$, with $\Ccal'$ depending only on $M_\psi$, $\eta_1$, $\eta_2$, $r_1$, $r_2$, $K_z$, $\left\|\sigma^{-1}\right\|_\infty$, $\alpha$, $T$, $\xi_1$ and $\xi_2$, and $\left(Z^{\varepsilon_n,\alpha,x}\right)_n$ is a Cauchy sequence in $\Lrm^2_\Pcal\left(\Omega, \Lrm^2\left([0,T], \left(\IR^d\right)^*\right)\right)$; we can define its limit process $\Zbar^{\alpha,x} \in \Lrm^2_{\Pcal,\loc}\left(\Omega, \Lrm^2\left(0,\infty;\left(\IR^d\right)^*\right)\right)$, and it satisfies the convergence we claimed.
\end{enumerate}
\end{proof}

\begin{prop}
Under Assumption \ref{ass_EBSDE}, there exists a constant $C$, such that for every $\alpha \in (0,1]$, we have:
\[ \forall x,x' \in \IR^d,\, \left|v^\alpha(x) - v^\alpha(x')\right| \leq C \left(1+|x|^2+|x'|^2\right)|x-x'|.\]
The constant $C$ only depends on $\eta_1$, $\eta_2$, $r_1$, $r_2$, $K_z$, $\left\|\sigma^{-1}\right\|_\infty$, $M_\psi$, $K_x$, $\|\Xi\|_\lip$ and $\|\sigma\|_\lip$.
\label{prop:valpha_locLipschitz}
\end{prop}

\begin{proof}
See Theorem \ref{thm:u_locLipschitz} given in the appendix, with $f : (x,y,z) \mapsto \psi(x,z) - \alpha y$ and $g = v^\alpha$.
\end{proof}

\begin{thm}[Existence of solutions to the EBSDE]
Under Assumption \ref{ass_EBSDE}, there exists a real number $\lambdabar$, a locally Lipschitz function $\vbar$, which satisfies $\vbar(0)=0$, and a process $\Zbar^x \in \Lrm^2_{\Pcal,\loc}\left(\Omega, \Lrm^2\left(0,\infty; \left(\IR^d\right)^*\right)\right)$ such that if we define $\Ybar_t^x = \vbar\left(X_t^x\right)$, then the EBSDE (\ref{eq:EBSDE}) is satisfied by $\left(\Ybar^x, \Zbar^x, \lambdabar\right)$ $\IP$-a.s. and for all $0 \leq t \leq T < \infty$. Moreover, there exists $C > 0$ such that for all $x \in \IR^d$, $\left|\vbar(x)\right| \leq C \left(1+|x|^2\right)$, and there exists $\zetabar$ measurable such that $\Zbar_t^x = \zetabar\left(X_t^x\right)$ $\IP$-a.s. and for a.e. $t \geq 0$.
\label{thm:existence_EBSDE}
\end{thm}

\begin{proof}~\\
The strategy is the same as in the Theorem 4.4 of \cite{EBSDE1}. We give a sketch of the proof here for completeness.
\begin{description}
\item[Step 1:] Construction of $\vbar$ by a diagonal procedure.\\
      For every $\alpha > 0$, we define $\vbar^\alpha (x) = v^\alpha(x) - v^\alpha(0)$; we recall that $\left|\vbar^\alpha(x)\right| \leq C \left(1+|x|^2\right)$ and $\left|\alpha v^\alpha(0)\right| \leq C$, with $C$ independent of $\alpha$. Let $D$ be a countable dense set in $\IR^d$; by a diagonal argument, we can construct a sequence $\left(\alpha_n\right)$, such that $\left(\vbar^{\alpha_n}\right)_n$ converges pointwise over $D$ to a function $\vbar$ and $\alpha_n v^{\alpha_n}(0) \underset{n \rightarrow \infty}{\longrightarrow} \lambdabar$, for a convenient real number $\lambdabar$. Moreover, thanks to the previous proposition:
      \[ \exists C > 0,\, \forall \alpha \in (0,1],\, \forall x,x' \in \IR^d,\, \left|\vbar^\alpha(x) - \vbar^\alpha(x')\right| \leq C \left(1+|x|^2+|x'|^2\right)|x-x'|.\]
      Because it is uniformly continuous on every compact subset of $D$, $\vbar$ has an extension which is continuous on $\IR^d$. Then, we can show that $\vbar$ is the pointwise limit of the functions $\vbar^{\alpha_n}$ on $\IR^d$, and then $\vbar$ is locally Lipschitz continuous and has quadratic growth.
\item[Step 2:] Construction of the process $\Zbar^x$.\\
      We will show that $\left(Z^{\alpha_n,x}\right)_n$ is Cauchy in $\Lrm^2_\Pcal\left(\Omega;\Lrm^2\left([0,T];\left(\IR^d\right)^*\right)\right)$ for every $T > 0$. Then, we will be able to define $\Zbar^x \in \Lrm^2_{\Pcal,\loc}\left(\Omega;\Lrm^2\left(0,\infty;\left(\IR^d\right)^*\right)\right)$. When $n \leq m \in \IN$, we set $\Ytilde = Y^{\alpha_n,x} - Y^{\alpha_m,x}$ and $\Ztilde = Z^{\alpha_n,x} - Z^{\alpha_m,x}$; we have:
      \[ \drm \Ytilde_t = -\psitilde_t \, \drm t + \left(\alpha_n Y_t^{\alpha_n,x} - \alpha_m Y_t^{\alpha_m,x}\right) \, \drm t + \Ztilde_t \, \drm W_t,\]
      where $\psitilde_t = \psi\left(X_t^x, Z_t^{\alpha_n,x}\right) - \psi\left(X_t^x,Z_t^{\alpha_m,x}\right)$. Thanks to Itô's formula:
      \begin{align*}
      \IE\left[\int_0^T \left|\Ztilde_t\right|^2 \, \drm t\right] 
      &= \IE\left[\left|\Ytilde_T\right|^2\right] - \left|Y_0\right|^2 + 2\IE\left[\int_0^T \psitilde_t \Ytilde_t \, \drm t\right] - 2\IE\left[\int_0^T \left(\alpha_n Y_t^{\alpha_n, x} - \alpha_m Y_t^{\alpha_m, x} \right) \Ytilde_t \, \drm t\right] \\
      &\leq \IE\left[\left|\Ytilde_T\right|^2\right] + 2 K_z \left\|\sigma^{-1}\right\|_\infty \IE\left[\int_0^T \left|\Ytilde_t\right| \left|\Ztilde_t\right| \, \drm t\right] + 4 M_\psi \IE\left[\int_0^T \left|\Ytilde_t\right| \, \drm t\right].
      \end{align*}
      By Cauchy-Schwarz, and noting that $\dst \left|\Ytilde_t \right| \left|\Ztilde_t\right| \leq K_z \left\|\sigma^{-1}\right\|_\infty \left|\Ytilde_t\right|^2 + \frac{1}{4 K_z \left\|\sigma^{-1}\right\|_\infty} \left|\Ztilde_t\right|^2$, we obtain:
      \[ \IE\left[\int_0^T \left|\Ztilde_t\right|^2 \, \drm t\right] \leq 2 \IE\left[\left|\Ytilde_T\right|^2\right] + 4 K_z^2 \left\|\sigma^{-1}\right\|_\infty^2 \IE\left[\int_0^T \left|\Ytilde_t\right|^2 \, \drm t\right] + 8 M_\psi \sqrt{T} \IE\left[\int_0^T \left|\Ytilde_t\right|^2\, \drm t\right]^{\frac{1}{2}}.\]
      Dominated convergences (using propositions \ref{prop:valpha_quadratic} and \ref{majo_X_Sp}) give us the result claimed.
\item[Step 3:] $\left(\Ybar^x, \Zbar^x, \lambdabar\right)$ is a solution to the EBSDE (\ref{eq:EBSDE}).\\
      Taking the limit in the BSDE satisfied by $\left(Y^{\alpha_n, x}, Z^{\alpha_n, x}\right)$ gives us:
      \[ \Ybar_t^x = \Ybar_T^x + \int_t^T \left[\psi\left(X_s^x, \Zbar_s^x\right) - \lambdabar\right] \, \drm s - \int_t^T \Zbar^x_s \, \drm W_s.\]
\item[Step 4:] $\Zbar^x$ can be represented as a measurable function of $X^x$. We fix $T > 0$. We denote $\Delta \bullet := \bullet^{\alpha,x} - \bullet^{\alpha,x'}$. By standard calculations, for every $\alpha \in (0,1]$, $x,x' \in \IR^d$:
      \begin{align*}
      \IE\left[\int_0^T \left|\Delta Z_t\right|^2 \, \drm t\right]
        &\leq \IE\left[\left|\Delta Y\right|^{*,2}_{0,T}\right] + 2 K_x T \IE\left[\left|\Delta Y \right|^*_{0,T} \left|\Delta X\right|^*_{0,T}\right] + 2 K_z \left\|\sigma^{-1}\right\|_\infty \IE\left[\int_0^T \left|\Delta Y \right|^*_{0,T} \left|\Delta Z_t\right| \, \drm t\right] \\
        &\pushright{+ 4 K_z \left\|\sigma^{-1}\right\|_\infty \IE\left[\int_0^T \left|\Delta Y \right|^*_{0,T} \left|Z_t^{\alpha,x'}\right| \, \drm t\right]} \\
        &\leq 2 \IE\left[\left|\Delta Y\right|^{*,2}_{0,T}\right] + 4 K_x T \IE\left[\left|\Delta Y \right|^*_{0,T} \left|\Delta X\right|^*_{0,T}\right] + 2 K_z^2 \left\|\sigma^{-1}\right\|_\infty^2 T \IE\left[\left|\Delta Y \right|^{*,2}_{0,T}\right]\\
        &\pushright{ + 8 K_z \left\|\sigma^{-1}\right\|_\infty \IE\left[\int_0^T \left|\Delta Y \right|^*_{0,T} \left|Z_t^{\alpha,x'}\right| \, \drm t\right]}
      \end{align*}
      
      But, using Cauchy-Schwarz inequality and the estimate of the end of the proof of proposition \ref{prop:valpha_quadratic}, we get:
      \[ \IE\left[\int_0^T \left|\Delta Y \right|^*_{0,T} \left|Z_t^{\alpha,x'}\right| \, \drm t\right] \leq C \IE\left[\left|\Delta Y\right|^{*,2}_{0,T}\right]^{\frac{1}{2}} (1+|x|).\]
      
      Using Lemma 2.1 of \cite{MZ02} and propositions \ref{majo_X_Sp} and \ref{prop:valpha_locLipschitz}, we finally get:
      \begin{equation}
      \IE\left[\int_0^T \left|Z_t^{\alpha,x} - Z_t^{\alpha,x'}\right|^2 \, \drm t\right] \leq \Ccal \left(1+|x|^4 + |x'|^4\right) |x-x'|^2,
      \label{eq:Zalpha_locLipschitz}
      \end{equation}
      where $\Ccal$ is independent of $x$ and $x'$, but depends on $\alpha$ and $T$. For every $x \in \IR^d$, the sequence $\dst \left(\IE\left[\int_0^T \left|Z_t^{\alpha_n,x}-Z_t^{\alpha_m,x}\right|^2 \, \drm t\right]\right)_{n \leq m \in \IN}$ is bounded (it converges). By a diagonal procedure, there exists a subsequence $\left(\alpha'_n\right) \subset \left(\alpha_n\right)$ such that:
      \[ \forall x \in D,\, \forall n \leq m \in \IN,\, \IE\left[\int_0^T \left|Z_t^{\alpha'_n,x} - Z_t^{\alpha'_m,x}\right|^2 \, \drm t\right] \leq 2^{-n}.\]
      Equation (\ref{eq:Zalpha_locLipschitz}) extends this inequality to $\IR^d$. Borel-Cantelli theorem gives for a.e. $t \in [0,T]$, $Z_t^{\alpha'_n,x} \underset{n \rightarrow \infty}{\longrightarrow} \Zbar^x_t$ $\IP$-a.s. Set
      \[ \zetabar(x) = \begin{cases} \lim_n \zeta^{\alpha'_n}(x), & \text{if the limit exists}, \\
                                     0, & \text{elsewhere.}
                       \end{cases}\]
      For a.e. $t \in [0,T]$, $X_t^x$ belongs $\IP$-a.s. to the set where $\lim_n \zeta^{\alpha'_n}(x)$ exists ; $\Zbar_t^x = \zetabar\left(X_t^x\right)$ $\IP$-a.s. and for a.e. $t \in [0,T]$.
\end{description}
\end{proof}

\subsection{Uniqueness of the solution}

\begin{thm}[Uniqueness of the parameter $\lambda$]
Let $p > 0$; we suppose that $\sqrt{r_2} K_z \left\|\sigma^{-1}\right\|_\infty + [(p \vee 2) -1] \dfrac{r_2}{2} < \eta_2$ and that Assumption \ref{ass_EBSDE} holds true.
We suppose that, for some $x \in \IR^d$, $(Y',Z',\lambda')$ verifies the EBSDE (\ref{eq:EBSDE}) $\IP$-a.s. and for all $0 \leq t \leq T < \infty$, where $Y'$ is a progressively measurable continuous process, $Z'$ is a process in $\Lrm^2_{\Pcal,\loc}\left(\Omega, \Lrm^2\left(0,\infty; \left(\IR^d\right)^*\right)\right)$ and $\lambda' \in \IR$. Finally, we assume that there exists $c_x > 0$ (that may depend on $x$) such that
\[ \forall t \geq 0, |Y'_t| \leq c_x \left(1+\left|X_t^x\right|^p\right).\]
Then $\lambda' = \lambdabar$.
\label{thm:lambda_uniq}
\end{thm}

\begin{proof}
We define $\lambdatilde = \lambda' - \lambdabar$, $\Ytilde = Y' - \Ybar$ and $\Ztilde = Z' - \Zbar$.
We have:
\[ \lambdatilde = \frac{\Ytilde_T - \Ytilde_0}{T} + \frac{1}{T} \int_0^T \left[\psi\left(X_t^x, Z'_t\right) - \psi\left(X_t^x, \Zbar_t^x\right)\right] \, \drm t - \frac{1}{T} \int_0^T \Ztilde_t \, \drm W_t.\]
We denote:
\[ \gamma_t^x = \left\{ \begin{array}{ll}
  \dfrac{\psi\left(X_t^x, Z'_t\right) - \psi\left(X_t^x, \Zbar_t^x\right)}{\left|\Ztilde_t\right|^2} \Ztilde_t, & \text{if } \Ztilde_t \neq 0, \\ 
  0, & \text{otherwise.} 
  \end{array} \right. \]
There exists a probability $\IPtilde^{x,T}$ under which $\dst \Wtilde_t^x = W_t - \int_0^t \gamma_s^{x *} \, \drm s$ is a Brownian motion on $[0,T]$. For every $\delta > 0$:
\[ \lambdatilde = \frac{1}{T} \IEtilde^{x,T}\left[\Ytilde_T - \Ytilde_0\right] \leq \frac{1}{T} \left\{c_x \left(2+|x|^p + \IEtilde^{x,T}\left[\left|X_T^x\right|^p\right]\right) + C \left(2+|x|^2 + \IEtilde^{x,T}\left[\left|X_T^x\right|^2\right]\right)\right\}.\]
We conclude by taking the limit $T \rightarrow \infty$ and using Proposition \ref{prop:sup_fini}.
\end{proof}

\begin{thm}[Uniqueness of the functions $v$ and $\zeta$]
Let $p > 0$; we suppose that $\sqrt{r_2} K_z \left\|\sigma^{-1}\right\|_\infty + [(p \vee 2)-1] \dfrac{r_2}{2} < \eta_2$ and that Assumption \ref{ass_EBSDE} holds true.\\
Let $(v,\zeta)$ and $(\vtilde,\zetatilde)$ be two couples of functions with:
\begin{itemize}
\item $v,\vtilde : \IR^d \rightarrow \IR$ are continuous, $|v(x)| \leq C\left(1+|x|^p\right)$, $|\vtilde(x)| \leq C \left(1+|x|^p\right)$ and $v(0) = \vtilde(0) = 0$;
\item $\zeta, \zetatilde : \IR^d \rightarrow \left(\IR^d\right)^*$ are measurable.
\end{itemize}
We also assume that for some constants $\lambda$, $\lambdatilde$, and for all $x \in \IR^d$, the triplets $\left(v\left(X_t^x\right), \zeta\left(X_t^x\right), \lambda\right)$ and $\left(\vtilde\left(X_t^x\right),\zetatilde\left(X_t^x\right), \lambdatilde\right)$ verify the EBSDE (\ref{eq:EBSDE}).\\
Then $\lambda = \lambdatilde$, $v = \vtilde$ and $\zeta\left(X_t^x\right) = \zetatilde\left(X_t^x\right)$ $\IP$-a.s. and for a.e. $t \geq 0$ and for all $x \in \IR^d$.
\label{thm:uniq_EBSDE}
\end{thm}

\begin{proof}~\\
By Theorem \ref{thm:lambda_uniq}, we already know that $\lambda = \lambdatilde$. We denote $Y_t^x = v\left(X_t^x\right)$, $Z_t^x = \zeta\left(X_t^x\right)$, $\Ytilde_t^x = \vtilde\left(X_t^x\right)$ and $\Ztilde_t^x = \zetatilde\left(X_t^x\right)$. 
We approximate the functions $\psi$, $v$ and $\vtilde$ by sequences $\left(\psi^\varepsilon\right)$, $\left(v^\varepsilon\right)$ and $\left(\vtilde^\varepsilon\right)$ of $\Ccal^1$ functions with bounded derivatives that converge uniformly. We can say that for all $\varepsilon > 0$, the bound on the derivatives of $\psi^\varepsilon$ is independent of $\varepsilon$ and $v^\varepsilon$, $\vtilde^\varepsilon$ have polynomial growth (with constant $C$ and exponent $p$, all independent of $\varepsilon$). In the following, $x \in \IR^d$, $T > 0$ and $\varepsilon > 0$ are fixed.
Let us consider the following BSDEs in finite time horizon:
\[ \left\{ \begin{array}{l}
   \drm Y_t^{\varepsilon, x} = - \left[\psi^\varepsilon\left(X_t^x, Z_t^{\varepsilon, x}\right)-\lambda\right] \, \drm t + Z_t^{\varepsilon,x} \, \drm W_t, \\
   Y_T^{\varepsilon, x} = v^\varepsilon\left(X_T^x\right), \end{array} \right.
   \text{ and }
   \left\{ \begin{array}{l}
   \drm \Ytilde_t^{\varepsilon, x} = - \left[\psi^\varepsilon\left(X_t^x, \Ztilde_t^{\varepsilon, x}\right)-\lambda\right] \, \drm t + \Ztilde_t^{\varepsilon,x} \, \drm W_t, \\
   \Ytilde_T^{\varepsilon, x} = \vtilde^\varepsilon\left(X_T^x\right). \end{array} \right.\]
We denote $\Delta Y_t^x = Y_t^x - \Ytilde_t^x$, $\Delta Z_t^x = Z_t^x - \Ztilde_t^x$, $\Delta Y_t^{\varepsilon,x} = Y_t^{\varepsilon,x} - \Ytilde_t^{\varepsilon,x}$, $\Delta Z_t^{\varepsilon,x} = Z_t^{\varepsilon,x} - \Ztilde_t^{\varepsilon,x}$. This way, we get:
\[
\Delta Y_0^x - \IE\left[\Delta Y_T^x\right] = \IE\left[\int_0^T \left[\psi\left(X_t^x, Z_t^x\right) - \psi\left(X_t^x, \Ztilde_t^x\right)\right] \, \drm t\right] \text{ and }\] 
\[\Delta Y_0^{\varepsilon,x} - \IE\left[\Delta Y_T^{\varepsilon,x}\right] = \IE\left[\int_0^T \left[\psi^\varepsilon\left(X_t^x, Z_t^{\varepsilon,x}\right) - \psi^\varepsilon\left(X_t^x, \Ztilde_t^{\varepsilon,x}\right)\right] \, \drm t\right].
\]
By substraction, it leads us to:
\begin{align*}
\left| \Delta Y_0^x - \Delta Y_0^{\varepsilon,x} \right| 
  &\leq \IE\left[\int_0^T \left|\psi\left(X_t^x, Z_t^x\right) - \psi^\varepsilon\left(X_t^x, Z_t^{\varepsilon,x}\right)\right| \, \drm t\right] + \IE\left[\int_0^T \left|\psi\left(X_t^x, \Ztilde_t^x\right) - \psi^\varepsilon\left(X_t^x, \Ztilde_t^{\varepsilon,x}\right)\right| \, \drm t\right] \\
  &\pushright{+ \left\|v-v^\varepsilon\right\|_\infty + \left\|\vtilde-\vtilde^\varepsilon\right\|_\infty.}
\end{align*}
Set $\delta Y_t^{\varepsilon,x} = Y_t^x - Y_t^{\varepsilon,x}$ and $\delta Z_t^{\varepsilon,x} = Z_t^x - Z_t^{\varepsilon,x}$. We have:
$\dst \left|\psi\left(X_t^x, Z_t^x\right) - \psi^\varepsilon\left(X_t^x, Z_t^{\varepsilon,x}\right)\right| \leq K_z \left\|\sigma^{-1}\right\|_\infty \left|\delta Z_t^{\varepsilon,x}\right| + \left\|\psi-\psi^\varepsilon\right\|_\infty$.\linebreak
Our next goal is to estimate $\dst \IE\left[\int_0 ^T \left|\delta Z_t^{\varepsilon,x}\right| \, \drm t\right]$. We see that $\left( \delta Y_t^{\varepsilon,x}, \delta Z_t^{\varepsilon,x}\right)$ is solution of the BSDE:
\[ \left\{ \begin{array}{l}
   \drm \delta Y_t^{\varepsilon,x} = -\left[ \psi\left(X_t^x, Z_t^x\right) - \psi^\varepsilon\left(X_t^x, Z_t^{\varepsilon,x}\right)\right] \, \drm t + \delta Z_t^{\varepsilon,x} \, \drm W_t, \\
   \delta Y_T^{\varepsilon,x} = \left(v-v^\varepsilon\right)\left(X_T^x\right).
   \end{array} \right.\]  
Using Lemma 2.1 of \cite{MZ02}, there exists $C > 0$ only depending on $T$ and the Lipschitz constant of $\psi$, such that:
\[ \IE\left[\int_0^T \left|\delta Z_t^{\varepsilon,x}\right|^2 \, \drm t\right] \leq C \left(T \left\|\psi-\psi^\varepsilon\right\|_\infty^2 + \left\|v-v^\varepsilon\right\|_\infty^2 \right).\]
By Cauchy-Schwarz, we finally get:
$\dst \IE\left[\int_0^T \left|\delta Z_t^{\varepsilon,x}\right| \, \drm t\right] \leq \sqrt{CT} \sqrt{T \left\|\psi-\psi^\varepsilon\right\|_\infty^2 + \left\|v-v^\varepsilon\right\|_\infty^2}$.
We can do the same for $\dst \IE\left[\int_0^T \left|\psi\left(X_t^x, \Ztilde_t^x\right) - \psi^\varepsilon\left(X_t^x, \Ztilde_t^{\varepsilon,x}\right)\right| \, \drm t\right]$, and we have:
$\dst \left|\Delta Y_0^x - \Delta Y_0^{\varepsilon,x}\right| \underset{\varepsilon \rightarrow 0}{\longrightarrow} 0$.
By Theorem 3.1 of \cite{MZ02}, because $\psi^\varepsilon$, $v^\varepsilon$ and $\vtilde^\varepsilon$ have bounded derivatives, there exists continuous functions $\zeta_T^\varepsilon$ and $\zetatilde_T^\varepsilon$ such that: $ Z_t^{\varepsilon,x} = \zeta_T^\varepsilon\left(t,X_t^x\right)$ and $\Ztilde_t^{\varepsilon,x} = \zetatilde_T^\varepsilon\left(t,X_t^x\right)$ $\IP$-a.s. and for a.e. $t \in [0,T]$. As usual, we linearise our BSDE; so we set:
\[ \Gamma_T^\varepsilon(t,x) = 
  \left\{ \begin{array}{ll} 
          \dfrac{\psi^\varepsilon\left(x,\zeta_T^\varepsilon(t,x)\right) - \psi^\varepsilon\left(x,\zetatilde_T^\varepsilon(t,x)\right)}{\left|\zeta_T^\varepsilon(t,x) - \zetatilde_T^\varepsilon(t,x)\right|^2} \left(\zeta_T^\varepsilon(t,x) - \zetatilde_T^\varepsilon(t,x)\right), & \text{if } \zeta_T^\varepsilon(t,x) \neq \zetatilde_T^\varepsilon(t,x), \\
          0, & \text{otherwise.}
          \end{array} \right.\]
The process $\Gamma_T^\varepsilon\left(t,X_t^x\right)^*$ is bounded by $K_z \left\|\sigma^{-1}\right\|_\infty$; by Girsanov, $\dst W_t^{\varepsilon,x,T} = W_t - \int_0^t \Gamma_T^\varepsilon\left(s,X_s^x\right)^* \, \drm s$ is a Brownian motion on $[0,T]$ under the probability $\IQ_T^{\varepsilon,x}$. $\Delta Y^{\varepsilon, x}$ is a $\IQ_T^{\varepsilon,x}$-martingale and we get:
\[ \Delta Y_0^{\varepsilon,x} = \IE^{\IQ_T^{\varepsilon,x}}\left[\Delta Y_T^{\varepsilon,x}\right] = \IE^{\IQ_T^{\varepsilon,x}}\left[\left(v^\varepsilon - \vtilde^\varepsilon\right)\left(X_T^x\right)\right] = \Pcal^\varepsilon_T\left[v^\varepsilon - \vtilde^\varepsilon\right](x),\]
where $\Pcal^\varepsilon_T$ is the Kolmogorov semigroup of the SDE:
\[ \drm U_t^x = \left[\Xi\left(U_t^x\right) + \sigma\left(U_t^x\right) \Gamma_T^\varepsilon\left(t, U_t^x\right)^*\right] \, \drm t + \sigma\left(U_t^x\right) \, \drm W_t.\]
Using Lemma \ref{lem:approx_lipschitz} and Corollary \ref{cor:semigroup}, we get:
$\dst \left|\Pcal^\varepsilon_T\left[v^\varepsilon-\vtilde^\varepsilon\right](x) - \Pcal^\varepsilon_T\left[v^\varepsilon-\vtilde^\varepsilon\right](0)\right| \leq C\left(1+|x|^{p \vee 2} \right) \erm^{-\nu T}$, where $\nu$ and $C$ are independent of $\varepsilon$ (because the polynomial growths of $v^\varepsilon$ and $\vtilde^\varepsilon$ do not depend on $\varepsilon$).
Finally,
\[ \forall \varepsilon > 0,\, \forall T >0,\, \forall x \in \IR^d,\, \left|\Delta Y_0^{\varepsilon,x} - \Delta Y_0^{\varepsilon,0}\right| \leq C \left(1+|x|^{p \vee 2}\right) \erm^{-\nu T}.\]
By taking the limit as $\varepsilon$ goes to $0$:
\[ \forall T > 0,\, \forall x \in \IR^d,\, \left|\left(v-\vtilde\right)(x)\right| = \left|\Delta Y_0^x - \Delta Y_0^0\right| \leq C \left(1+|x|^{p \vee 2}\right) \erm^{-\nu T}.\]
Taking the limit as $T$ goes to infinity leads us to $v = \vtilde$. Then, uniqueness of $\zeta$ is the consequence of Itô's formula.
\end{proof}

\section{Large time behaviour}

In this section, we always suppose that $\sqrt{r_2} K_z \left\|\sigma^{-1}\right\|_\infty + \dfrac{\mu-1}{2}r_2 < \eta_2$, with $\mu \geq 2$, and we keep working under Assumption \ref{ass_EBSDE}. Indeed, we have seen in the previous section that, under this assumption, there exists a unique triplet $(v,\zeta,\lambda)$ such that $\left(v\left(X_t^x\right), \zeta\left(X_t^x\right), \lambda\right)$ is a solution to the EBSDE (\ref{eq:EBSDE}), $v$ is continuous with quadratic growth, $v(0)=0$ and $\zeta$ is measurable. Let $\xi^T$ be a real random variable $\Fcal_T$-measurable and such that $\left|\xi^T\right| \leq C \left(1+\left|X_T^x\right|^\mu\right)$. It will allow us to use proposition \ref{prop:sup_fini}. We denote by $\left(Y_t^{T,x}, Z_t^{T,x}\right)$ the solution of the BSDE in finite horizon:
\[ Y_t^{T,x} = \xi^T + \int_t^T \psi\left(X_s^x, Z_s^{T,x}\right) \, \drm s - \int_t^T Z_s^{T,x} \, \drm W_s.\]

\begin{thm}
We have the following inequality:
\[ \left|\frac{Y_0^{T,x}}{T}-\lambda\right| \leq \frac{C\left(1+|x|^\mu\right)}{T},\]
where the constant $C$ is independent of $x$ and $T$; and in particular:
$\dst \frac{Y_0^{T,x}}{T} \underset{T \rightarrow \infty}{\longrightarrow} \lambda$, uniformly in any bounded subset of $\IR^d$.
\label{thm:1st_behav}
\end{thm}

\begin{proof}~\\
For all $x \in \IR^d$ and $T > 0$, we write:
\[ \left|\frac{Y_0^{T,x}}{T}-\lambda\right| \leq \left| \frac{Y_0^{T,x} - Y_0^x - \lambda T}{T} \right| + \left|\frac{Y_0^x}{T}\right|.\]
First of all, $\left|Y_0^x\right| = |v(x)| \leq C \left(1+|x|^2\right)$. Also, by the usual linearisation technique, we have:
\[ Y_0^{T,x} - Y_0^x - \lambda T = \xi^T - v\left(X_T^x\right) + \int_0^T \left(Z_s^{T,x} - Z_s^x\right) \beta_s^{T *} \, \drm s - \int_0^T \left(Z_s^{T,x} - Z_s^x\right) \, \drm W_s,\]
where
\[ \beta_t^T = \left\{ \begin{array}{ll}
   \dfrac{\psi\left(X_t^x, Z_t^{T,x}\right)-\psi\left(X_t^x, Z_t^x\right)}{\left|Z_t^{T,x} - Z_t^x\right|^2} \left(Z_t^{T,x} - Z_t^x\right), & \text{if } Z_t^{T,x} \neq Z_t^x, \\
   0, & \text{otherwise.}
   \end{array} \right.\]
The process $\beta^T$ is bounded by $K_z \left\|\sigma^{-1}\right\|_\infty$ and by Girsanov's theorem, there exists a probability measure $\IQ^T$ under which $\dst \Wtilde_t^T = W_t - \int_0^t \beta_s^{T *} \, \drm s$ is a Brownian motion on $[0,T]$. This way, we can see that:
\[ \left|Y_0^{T,x} - Y_0^x - \lambda T\right| = \left| \IE^{\IQ^T}\left[ \xi^T - v\left(X_T^x\right) \right]\right| \leq \IE^{\IQ^T}\left[\left|\xi^T\right|\right] + \IE^{\IQ^T}\left[\left|v\left(X_T^x\right)\right|\right] \leq C \left(1+ \IE^{\IQ^T}\left[\left|X_T^x\right|^\mu\right]\right).\]
Thanks to Proposition \ref{prop:sup_fini}, we get:
$\dst \sup_{T \geq 0} \IE^{\IQ^T}\left[\left|X_T^x\right|^\mu\right] \leq \kappa\left( 1 +|x|^\mu\right)$, where $\kappa$ is independent of $x$.
\end{proof}

\begin{thm}
We suppose that $\xi^T = g\left(X_T^x\right)$, where $g : \IR^d \rightarrow \IR$ has polynomial growth: $\forall x \in \IR^d,\, |g(x)| \leq C\left(1+|x|^\mu\right)$ and satisfies: $\forall x,x' \in \IR^d,\, |g(x)-g(x')| \leq C\left(1+|x|^\mu+|x'|^\mu\right) |x-x'|$.\\
Then, there exists $L \in \IR$, such that:
$ \dst \forall x \in \IR^d,\, Y_0^{T,x} - \lambda T - Y_0^x \underset{T \rightarrow \infty}{\longrightarrow} L$.
Furthermore,
\[ \forall x \in \IR^d,\, \forall T > 0,\, \left|Y_0^{T,x} - \lambda T - Y_0^x - L\right| \leq C \left(1+|x|^{\mu}\right) \erm^{- \nu T}.\]
\label{thm:LTB}
\end{thm}

\begin{proof}~\\
We will consider the following equations:
\[ \left\{ \begin{array}{rcl}
           \drm Y_s^{T,t,x} &=& -\psi\left(X_s^{t,x},Z_s^{T,t,x}\right) \, \drm s + Z_s^{T,t,x} \drm W_s \\
           Y_T^{T,t,x} &=& g\left(X_T^{t,x}\right)
           \end{array} \right. \longrightarrow Y_s^{T,t,x} = u_T\left(s,X_s^{t,x}\right) \text{, see Theorem 4.1 of \cite{EPQ97}}; \]
\[ \left\{ \begin{array}{rcl}
           \drm Y_s^{t,x} &=& -\left\{\psi\left(X_s^{t,x},Z_s^{t,x}\right) - \lambda\right\} \, \drm s + Z_s^{t,x} \drm W_s \\
           Y_T^{t,x} &=& v\left(X_T^{t,x}\right)
           \end{array} \right. \longrightarrow Y_s^{t,x} = v\left(X_s^{t,x}\right) \text{, solution of the EBSDE}; \]
\[ \left\{ \begin{array}{rcl}
           \drm Y_s^{n,t,x} &=& -\psi^n\left(X_s^{n,t,x},Z_s^{n,t,x}\right) \, \drm s + Z_s^{n,t,x} \drm W_s \\
           Y_T^{n,t,x} &=& g^n\left(X_T^{n,t,x}\right)
           \end{array} \right. \longrightarrow Y_s^{n,t,x} = u_T^n\left(s,X_s^{n,t,x}\right) \text{, where $u_T^n$ is $\Ccal^1$ (see appendix)}; \]
\[ \left\{ \begin{array}{rcl}
           \drm \Ytilde_s^{n,t,x} &=& -\psi^n\left(X_s^{n,t,x},\Ztilde_s^{n,t,x}\right) \, \drm s + \Ztilde_s^{n,t,x} \drm W_s \\
           \Ytilde_T^{n,t,x} &=& v^n\left(X_T^{n,t,x}\right)
           \end{array} \right. \longrightarrow \Ytilde_s^{n,t,x} = \utilde_T^n\left(s,X_s^{n,t,x}\right) \text{, where $\utilde_T^n$ is $\Ccal^1$ (see appendix)}; \]
where $g^n, v^n \in \Ccal^1_b$ converge uniformly on every compact towards $g$ and $v$; also, we take $\left(\Xi^n\right)_n$ a sequence of $\Ccal^1_b$ functions that converges uniformly towards $\Xi$ and define $X^{n,t,x}$ as the solution of the SDE
\[ \left\{ \begin{array}{rcl}
           \drm X^{n,t,x}_s &=& \Xi^n\left(X_s^{n,t,x}\right) \, \drm s + \sigma\left(X_s^{n,t,x}\right) \, \drm W_s, \\
           X_t^{n,t,x} &=& x. \end{array} \right. \] 
\begin{description}
\item[Step 1:] Approximation of the function $\psi$.\\
      We set $\rho(x,p) = \psi(x,p \sigma(x))$ ; $\rho$ is Lipschitz, and we approximate it by a sequence of $\Ccal^1_b$ functions $\left(\rho^n\right)$ which converges uniformly on $\IR^d \times \left(\IR^d\right)^*$. Then, we define a function $\beta^n$ satisfying:
      \[ \beta^n(x,p) = \left\{ \begin{array}{ll}
        \rho^n(x,p), & \text{if } |p\sigma(x)| \leq n, \\
        0,           & \text{if } |p \sigma(x)| \geq f(n), \end{array} \right.\]
      where $f(n)$ is chosen such that $\beta^n$ and $\rho^n$ have the same Lipschitz constant.
      This way, $\beta^n \underset{n \rightarrow \infty}{\longrightarrow} \rho$, uniformly on every compact, and $\beta^n \in \Ccal^1_b$. We set $\psi^n(x,z) = \beta^n\left(x,z\sigma(x)^{-1}\right)$ and $\psi^n \in \Ccal^1_b$. Moreover, we have the following:
      \[
      \left|\left(\psi-\psi^n\right)(x,z)\right| 
      \leq \left\|\rho-\rho^n\right\|_\infty + \left(2M_\psi(1+|x|) + 2 K_z \left\|\sigma^{-1}\right\|_\infty |z|\right) \ind_{|z|>n}.
      \]
\item[Step 2:] Growth of $w_T^n$.\\
      We set $w_T^n(t,x) = u_T^n(t,x) - \utilde_T^n(t,x)$ and $w_T(t,x) = u_T(t,x) - \lambda(T-t) - v(x)$. By uniqueness of viscosity solutions (see \cite{EPQ97}), we have $u_T^n(0,x) = u_{T+S}^n(S,x)$ and $\utilde_T^n(0,x) = \utilde_{T+S}^n(S,x)$.
      \begin{lem}
      We have: $\forall \varepsilon > 0,\, \forall T > 0,\, \exists \Ctilde_T > 0,\, \forall x \in \IR^d,\, \forall n \in \IN^*,$
      \begin{align*}
      &\hspace{-1.5cm}\left|v(x) - \utilde_T^n(0,x) + \lambda T\right|^2 \\
      &\leq \Ctilde_T \left\{ \IE\left[\left|v\left(X_T^x\right) - v^n\left(X_T^{n,x}\right)\right|^2\right] + \left\|\rho-\rho^n\right\|_\infty^2 + \left(\left\|\Xi-\Xi^n\right\|_\infty^2 + \frac{1}{n^\varepsilon}\right) \left(1+|x|^{2(2+\varepsilon)}\right)\right\}.
      \end{align*}
      \end{lem}
      \begin{proof}
      We set $\delta \Ytilde_s^n = Y_s^x - \Ytilde_s^{n,x} + \lambda(T-s)$ and $\delta \Ztilde_s^n = Z_s^x - \Ztilde_s^{n,x}$; then,
      \[ \drm \delta \Ytilde_s^n = \drm Y_s^x - \drm \Ytilde_s^{n,x} - \lambda \drm s = \left\{\psi^n\left(X_s^{n,x}, \Ztilde_s^{n,x}\right)-\psi\left(X_s^x, Z_s^x\right)\right\} \, \drm s + \delta \Ztilde_s^{n,x} \, \drm W_s.\]
      We will use \cite{Rennes03}; to use the same notations, we set $f(s,y,z) = \psi\left(X_s^x,Z_s^x\right) - \psi^n\left(X_s^x, Z_s^x-z\right)$. We have:
      \begin{align*}
      |f(s,y,z)| &\leq \left\|\rho-\rho^n\right\|_\infty + \left(2 M_\psi\left(1+\left|X_s^x\right|\right) + 2 K_z \left\|\sigma^{-1}\right\|_\infty \left|Z_s^x\right|\right) \ind_{\left|Z_s^x\right| > n} + K_x \left|X_s^x - X_s^{n,x}\right|&  \\
      & \pushright{+ K_z \left|Z_s^x\right| \|\sigma\|_\lip \left\|\sigma^{-1}\right\|_\infty^2 \left|X_s^x - X_s^{n,x}\right| + K_z \left\|\sigma^{-1}\right\|_\infty |z|}
      & =: f_s + \lambda |z|.
      \end{align*}
      We define $\dst F = \int_0^T f_s \, \drm s$; because $F \in \Lrm^2$ and $\delta \Ytilde^n \in \Scal^2$, we get, $\IE\left[\left|\delta \Ytilde^n\right|^{*,2}_{0,T}\right] \leq C \erm^{2 \lambda^2 T} \IE\left[\left|\delta \Ytilde^n_T\right|^2 + F^2\right]$. We just need an upper bound for $\IE\left[F^2\right]$:
      \begin{align*}
      \IE\left[F^2\right]
      & \leq 5 T^2 K_x^2 \IE\left[\left|X^x - X^{n,x}\right|^{*,2}_{0,T}\right] + 5 K_z^2 \|\sigma\|_\lip^2 \left\|\sigma^{-1}\right\|_\infty^4 \IE\left[\left|X^x-X^{n,x}\right|^{*,2}_{0,T} \left(\int_0^T \left|Z_s^x\right| \, \drm s\right)^2\right] + 5 T^2 \left\|\rho-\rho^n\right\|_\infty^2 \\
      & \pushright{ + 40 M_\psi^2 \IE\left[\left(1+\left|X^x\right|^{*,2}_{0,T}\right) \left(\int_0^T \ind_{\left|Z_s^x\right| > n} \, \drm s\right)^2\right] + 20 K_z^2 \left\|\sigma^{-1}\right\|_\infty^2 \IE\left[\left(\int_0^T \left|Z_s^x\right| \ind_{\left|Z_s^x\right| > n} \, \drm s\right)^2\right].}
      \end{align*}
      Thanks to the Lemma 2.1 of \cite{MZ02}, we have:
      $\dst \forall p \geq 2,\, \exists C_{T,p} > 0,\, \IE\left[\left|X^x - X^{n,x}\right|^{*,p}_{0,T}\right] \leq C_{T,p} \left\|\Xi-\Xi^n\right\|_\infty^p$.
      Moreover, we have, for every $\alpha \in (0,2]$:
      \[ \IE\left[\left(\int_0^T \ind_{\left|Z_s^x\right| > n} \, \drm s\right)^2\right] \leq T \int_0^T \IP\left(\left|Z_s^x\right| > n\right) \, \drm s \leq \frac{T}{n^\alpha} \int_0^T \IE\left[\left|Z_s^x\right|^\alpha\right] \, \drm s \leq \frac{T^{2-\frac{\alpha}{2}}}{n^\alpha} \IE\left[\int_0^T \left|Z_s^x\right|^2 \, \drm s\right]^{\frac{\alpha}{2}}.\]
      We also recall -- see \cite{Rennes03} -- for $p > 1$:
      $\dst \IE\left[\left(\int_0^T \left|Z_s^x\right|^2 \, \drm s\right)^{\frac{p}{2}}\right] \leq C_{p,T} \left(1+|x|^{2p}\right)$,
      because $v$ has quadratic growth.
      Using Cauchy-Schwarz inequality and the above remarks, we get:
      \[ \IE\left[F^2\right] \leq C_T \left(\left\|\rho-\rho^n\right\|_\infty^2 + \left(\left\|\Xi-\Xi^n\right\|_\infty^2 + \frac{1}{n^\varepsilon}\right) \left(1+|x|^{2(2+\varepsilon)}\right)\right),\]
      for every $\varepsilon > 0$.
      \end{proof}
      \begin{lem}
      We have: $\forall \varepsilon > 0,\, \forall T > 0,\, \exists C_T > 0,\, \forall x \in \IR^d,\, \forall n \in \IN^*,$
      \begin{align*}
      &\hspace{-1.5cm}\left|u_T(0,x) - u_T^n(0,x)\right|^2 \\
      &\leq C_T \left\{ \IE\left[\left|g\left(X_T^x\right) - g^n\left(X_T^{n,x}\right)\right|^2\right] + \left\|\rho-\rho^n\right\|_\infty^2 + \left(\left\|\Xi-\Xi^n\right\|_\infty^2 + \frac{1}{n^\varepsilon}\right) \left(1+|x|^{(2+\varepsilon)\mu}\right)\right\}.
      \end{align*}
      \end{lem}
      \begin{proof}~\\
      The proof is essentially the same; the main difference is that the exponent $\mu$ appears in the following inequality (see \cite{Rennes03}): $\dst \IE\left[\left(\int_0^T \left|Z_s^{T,0,x}\right|^2 \, \drm s\right)^{\frac{p}{2}}\right] \leq C_{p,T} \left(1 + \IE\left[\left|Y_T^{T,0,x}\right|^p\right]\right)$.
      \end{proof}
      We keep in mind the following results:
      \begin{align*}
      \left|w_T^n(0,x) - w_T(0,x)\right|^2 
      &\leq C_T \left\{ \IE\left[\left|g\left(X_T^x\right) - g^n\left(X_T^{n,x}\right)\right|^2 + \left|v\left(X_T^x\right) - v^n\left(X_T^{n,x}\right)\right|^2\right] \right.\\
      &\pushright{ \left. + \left\|\rho-\rho^n\right\|_\infty^2 + \left(\left\|\Xi-\Xi^n\right\|_\infty^2 + \frac{1}{n^\varepsilon}\right) \left(1+|x|^{(2+\varepsilon)\mu}\right)\right\},}\\
      \left|w_T(0,x)\right| &\leq C \left(1+|x|^\mu\right).
      \end{align*}
\item[Step 3:] Variation of $w_T^n(0,\bullet)$.\\
      We write:
      \[ Y_0^{n,x} - \Ytilde_0^{n,x} = Y_T^{n,x} - \Ytilde_T^{n,x} + \int_0^T \left\{ \psi^n\left(X_s^{n,x}, Z_s^{n,x}\right) - \psi^n\left(X_s^{n,x}, \Ztilde_s^{n,x}\right)\right\} \, \drm s - \int_0^T \left\{ Z_s^{n,x} - \Ztilde_s^{n,x}\right\} \, \drm W_s.\]
      We have $Z_s^{n,x} = \partial_x u_T^n\left(s, X_s^{n,x}\right) \sigma\left(X_s^{n,x}\right)$ and $\Ztilde_s^{n,x} = \partial_x \utilde_T^n\left(s,X_s^{n,x}\right) \sigma\left(X_s^{n,x}\right)$ (see the Theorem 3.1 of \cite{MZ02}). Consequently, $Z_s^{n,x} - \Ztilde_s^{n,x} = \partial_x w_T^n\left(s, X_s^{n,x}\right) \sigma\left(X_s^{n,x}\right)$. We define the following function:
      \[ \beta_T^n(t,x) = \left\{ \begin{array}{ll}
       \dfrac{\psi^n\left(x, \partial_x u_T^n(t,x) \sigma(x)\right) - \psi^n\left(x, \partial_x \utilde_T^n(t,x) \sigma(x)\right)}{\left|\partial_x w_T^n(t,x) \sigma(x)\right|^2} \left(\partial_x w_T^n(t,x) \sigma(x) \right)^*, & \text{if } t < T \text{ and } \partial_x w_T^n(t,x) \neq 0, \\ 
       \dfrac{\psi^n\left(x, \partial_x u_T^n(T,x) \sigma(x)\right) - \psi^n\left(x, \partial_x \utilde_T^n(T,x) \sigma(x)\right)}{\left|\partial_x w_T^n(T,x) \sigma(x)\right|^2} \left(\partial_x w_T^n(T,x) \sigma(x) \right)^*, & \text{if } t \geq T \text{ and } \partial_x w_T^n(T,x) \neq 0, \\
       0, & \text{otherwise.} \end{array} \right.\] 
       The function $\beta_T^n$ is bounded by $K_z \left\|\sigma^{-1}\right\|_\infty$; $\dst \Wtilde_t = W_t - \int_0^t \beta_T^n\left(s,X_s^{n,x}\right) \, \drm s$ is a Brownian motion on $[0,T]$ under the probability $\IQ_T^n$. Because $\psi^n\left(X_s^{n,x}, Z_s^{n,x}\right) - \psi^n\left(X_s^{n,x}, \Ztilde_s^{n,x}\right) = \left(Z_s^{n,x} - \Ztilde_s^{n,x}\right) \beta_T^n\left(s,X_s^{n,x}\right)$, we have:
       \[w_T^n(0,x) = Y_0^{n,x} - \Ytilde_0^{n,x} = \IE^{\IQ_T^n}\left[Y_T^{n,x} - \Ytilde_T^{n,x}\right] = \IE^{\IQ_T^n}\left[g^n\left(X_T^{n,x}\right) - v^n\left(X_T^{n,x}\right)\right] = \Pcal_T^n\left[g^n-v^n\right](x),\]
       where $\Pcal^n$ is the Kolmogorov semigroup of the SDE: $\drm U_t = \Xi^n\left(U_t\right) \, \drm t + \sigma\left(U_t\right)\left(\drm W_t + \beta_T^n\left(t, U_t\right) \, \drm t\right)$. By Lemma \ref{lem:approx_lipschitz} and corollary \ref{cor:semigroup}, we get:
       \[\left|w_T^n(0,x) - w_T^n(0,y)\right| = \left|\Pcal_T^n\left[g^n-v^n\right](x) - \Pcal_T^n\left[g^n-v^n\right](y)\right| \leq \chat_n c_{g^n-v^n} \left(1+|x|^{\mu} + |y|^{\mu}\right) \erm^{-\nuhat_n T},\]
       where $\chat_n$ and $\nuhat_n$ depend only on $\eta_1$, $\eta_2$, $r_1$, $r_2$, $K_z$ and $\left\|\sigma^{-1}\right\|_\infty$; also $c_{g^n-v^n}$ is the constant appearing in the $\mu$-polynomial growth of $g^n-v^n$: it only depends on $g-v$ and is independent of $n$. So:
       \begin{equation}
       \exists \chat, \nuhat >0,\, \forall T > 0,\, \forall x,y \in \IR^d,\, \forall n \in \IN^*,\, \left|w_T^n(0,x) - w_T^n(0,y)\right| \leq \chat\left(1+|x|^{\mu} + |y|^{\mu}\right) \erm^{-\nuhat T}.
       \label{eq:LTB_step3}
       \end{equation}
\item[Step 4:] Upperbound for $\partial_x w_T^n(0,\bullet)$.\\
      We use the Theorem 4.2 of \cite{MZ02}; for every $T' \in (t,T]$, we have:
      \[ \partial_x u_T^n(t,x) = \IE\left[u_T^n\left(T', X_{T'}^{n,t,x}\right) N_{T'}^{n,t,x} + \int_t^{T'} \psi^n\left(X_s^{n,t,x}, Z_s^{n,t,x}\right) N_s^{n,t,x} \, \drm s \right],\]
      \[ \partial_x \utilde_T^n(t,x) = \IE\left[\utilde_T^n\left(T', X_{T'}^{n,t,x}\right) N_{T'}^{n,t,x} + \int_t^{T'} \psi^n\left(X_s^{n,t,x}, \Ztilde_s^{n,t,x}\right) N_s^{n,t,x} \, \drm s \right],\]
      where $\dst N_s^{n,t,x} = \frac{1}{s-t} \left(\int_t^s \left(\sigma\left(X_r^{n,t,x}\right)^{-1} \nabla X_r^{n,t,x}\right)^* \, \drm W_r\right)^*$. By substraction:
      \[ \partial_x w_T^n(t,x) = \IE\left[w_{T-T'}^n\left(0,X_{T'}^{n,t,x}\right) N_{T'}^{n,t,x} + \int_t^{T'} \left\{\psi^n\left(X_s^{n,t,x}, Z_s^{n,t,x}\right) - \psi^n\left(X_s^{n,t,x}, \Ztilde_s^{n,t,x}\right)\right\} N_s^{n,t,x} \, \drm s\right].\]
      By Itô's formula, we prove that:
      $\dst \drm \left|\nabla X_s^{n,t,x}\right|_F^2 \leq \left(2 \|\Xi\|_\lip + \|\sigma\|_\lip^2\right) \left|\nabla X_s^{n,t,x}\right|_F^2 \, \drm s + \drm M_s$,
      with $M$ a martingale starting at $0$. We set $\dst \lambda = 2 \|\Xi\|_\lip + \|\sigma\|_\lip^2$ and we get:
      $\dst \IE\left[\left|\nabla X_s^{n,t,x}\right|_F^2\right] \leq d^2 \erm^{\lambda (s-t)}$.
      Then, for any $s \in [t,T']$,
      \[\IE\left[\left|N_s^{n,t,x}\right|^2\right]  
        \leq \left(\frac{1}{s-t}\right)^2 \IE\left[\int_t^s \left|\sigma\left(X_r^{n,t,x}\right)^{-1} \nabla X_r^{n,t,x}\right|_F^2 \, \drm r\right] \\
        \leq \frac{d^2 \left\|\sigma^{-1}\right\|_\infty^2}{s-t} \erm^{\lambda(T'-t)}.\]
      Recalling that $(x,p) \mapsto \psi^n(x,p\sigma(x))$ is Lipschitz, we obtain:
      \begin{align*}
      \left| \partial_x w_T^n(t,x) \right|
      & \leq \IE\left[\left|w_{T-T'}^n\left(0,X_{T'}^{n,t,x}\right)\right| \left|N_{T'}^{n,t,x}\right|\right] + \int_t^{T'} \IE\left[\left|\psi^n\left(X_s^{n,t,x}, Z_s^{n,t,x}\right) - \psi^n\left(X_s^{n,t,x}, \Ztilde_s^{n,t,x}\right)\right| \left|N_s^{n,t,x} \right| \right] \, \drm s \\
      & \leq \IE\left[\left|w_ {T-T'}^n\left(0, X_{T'}^{n,t,x}\right)\right|^2\right]^{\frac{1}{2}} \IE\left[\left|N_{T'}^{n,t,x}\right|^2\right]^{\frac{1}{2}} + K_z \int_t^{T'} \IE\left[\left|\partial_x w_T^n\left(s,X_s^{n,t,x}\right)\right|^2\right]^{\frac{1}{2}} \IE\left[\left|N_s^{n,t,x}\right|^2\right]^{\frac{1}{2}} \, \drm s \\
      & \leq \IE\left[\left|w_ {T-T'}^n\left(0, X_{T'}^{n,t,x}\right)\right|^2\right]^{\frac{1}{2}} \frac{d \left\|\sigma^{-1}\right\|_\infty}{\sqrt{T'-t}} \erm^{\frac{\lambda}{2} (T'-t)} + K_z d \left\|\sigma^{-1}\right\|_\infty \erm^{\frac{\lambda T'}{2}} \int_t^{T'} \IE\left[\left|\partial_x w_T^n\left(s,X_s^{n,t,x}\right)\right|^2\right]^{\frac{1}{2}} \, \frac{\drm s}{\sqrt{s-t}}.
      \end{align*}
      Thanks to step 2 and proposition \ref{majo_X_Sp}, we have the following upperbound (for any $\varepsilon > 0$):
      \[ \IE\left[\left|w_{T-T'}^n\left(0, X_{T'}^{n,t,x}\right)\right|^2\right] \leq C_{T'} \left(1+|x|^{2\mu}\right) + C_{T-T'}\left( \Delta_t^{n,x} + \left\|\rho-\rho^n\right\|^2_\infty + C_{T'}\left( \left\|\Xi - \Xi^n\right\|_\infty^2 + \frac{1}{n^\varepsilon}\right) \left(1+|x|^{(2+\varepsilon)\mu}\right)\right),\]
      where $\Delta_t^{n,x} = \IE\left[\left|v\left(X_{T-T'}^{0,X_{T'}^{n,t,x}}\right) - v^n\left(X_{T-T'}^{n,0,X_{T'}^{n,t,x}}\right)\right|^2 + \left|g\left(X_{T-T'}^{0,X_{T'}^{n,t,x}}\right) - g^n\left(X_{T-T'}^{n,0,X_{T'}^{n,t,x}}\right)\right|^2\right]$. So, for every $\varepsilon > 0$,
      \begin{align*}
      \left|\partial_x w_T^n(t,x)\right|
      & \leq \left\{C_{T'} \left(1+|x|^{\mu}\right) + C_{T-T'}\left( \sqrt{\Delta_t^{n,x}} + \left\|\rho-\rho^n\right\|_\infty + C_{T'} \left(\left\|\Xi-\Xi^n\right\|_\infty + n^{-\frac{\varepsilon}{2}} \right) \left(1+|x|^{(1+\varepsilon)\mu}\right)\right)\right\}  \\
      & \pushright{ \times \frac{d \left\|\sigma^{-1}\right\|_\infty}{\sqrt{T'-t}} \erm^{\frac{\lambda}{2} (T'-t)} + K_z d \left\|\sigma^{-1}\right\|_\infty \erm^{\frac{\lambda T'}{2}} \int_t^{T'} \IE\left[\left|\partial_x w_T^n\left(s,X_s^{t,x}\right)\right|^2\right]^{\frac{1}{2}} \, \frac{\drm s}{\sqrt{s-t}}.}
      \end{align*}
      Let us take $\zeta > \mu$; we set $\dst \varphi_T^n(t) = \sup_{x \in \IR^d} \frac{\left| \partial_x w_T^n(t,x)\right|}{1+|x|^\zeta}$. Using the appendix, we see that:
      \[ \left|\partial_x w_T^n(t,x)\right| \leq \left|\partial_x u_T^n(t,x)\right| + \left|\partial_x \utilde_T^n(t,x)\right| \leq C\left(1+|x|^\mu\right).\]
      This proves that $\varphi_T^n$ is well defined on $[0,T)$; moreover, it is bounded on every set $[0,T']$, for every $T' < T$. Also, for every $s \in [t, T']$, $\dst \IE\left[\left|\partial_x w_T^n\left(s, X_s^{t,x}\right)\right|^2 \right]^{\frac{1}{2}} \leq \varphi_T^n(s) C_{T'} \left(1+|x|^{\zeta}\right)$,
      where $C_{T'}$ is independent of $T$ and $x$. We get:
      \begin{align*}
      \frac{\left|\partial_x w_T^n(t,x)\right|}{1+|x|^\zeta}
       &\leq \left\{ C_{T'} + C_{T-T'} \left(\frac{\sqrt{\Delta_t^{n,x}}}{1+|x|^\zeta} + \left\|\rho-\rho^n\right\|_\infty + C_{T'} \left(\left\|\Xi-\Xi^n\right\|_\infty + n^{-\frac{\varepsilon}{2}} \right)\right)\right\} \frac{d \left\|\sigma^{-1}\right\|_\infty}{\sqrt{T'-t}} \erm^{\frac{\lambda}{2} (T'-t)} \\
       &\pushright{ + C_{T'} \int_t^{T'} \frac{\varphi_T^n(s)}{\sqrt{s-t}} \, \drm s.}
      \end{align*}
      For $\varepsilon$ small enough, we can take the supremum when $x \in \IR^d$, and by a change of variable, it can be rewritten as:
      \begin{align*}
       \varphi_T^n(T'-t) &\leq \underbrace{\left\{ C_{T'} + C_{T-T'} \left(\sup_{x \in \IR^d} \frac{\sqrt{\Delta_{T'-t}^{n,x}}}{1+|x|^\zeta} + \left\|\rho-\rho^n\right\|_\infty + C_{T'} \left(\left\|\Xi-\Xi^n\right\|_\infty + n^{-\frac{\varepsilon}{2}} \right)\right)\right\} \frac{d \left\|\sigma^{-1}\right\|_\infty}{\sqrt{t}} \erm^{\frac{\lambda}{2}t}}_{= a(t)} \\
       & \pushright{ + C_{T'} \int_0^t \frac{\varphi_T^n(T'-u)}{\sqrt{t-u}} \, \drm u,}
       \end{align*}
      where $0 < t \leq T' < T$. We use the Lemma 7.1.1 of \cite{Henry}; indeed, the function $a$ is locally integrable on $(0,T')$, since $\dst \sup_{x \in \IR^d} \frac{\sqrt{\Delta_{T'-t}^{n,x}}}{1+|x|^\zeta}$ is bounded independently of $t$. We get $\dst \varphi_T^n(T'-t) \leq a(t) + C_{T'} \int_0^t E'\left(C_{T'}(t-u)\right) a(u) \, \drm u$, and the integral is well defined, since $\dst E(z) = \sum_{n=0}^\infty \frac{z^{\frac{n}{2}}}{\Gamma\left(\frac{n}{2}+1\right)}$ is $\Ccal^1$ and $\dst E'(z) \underset{z \rightarrow 0^+}{\sim} \frac{\Gamma\left(\frac{1}{2}\right)}{\sqrt{z}}$. Choose $t = T'$ and get:
      \[
      \left|\partial_x w_T^n(0,x)\right|
      \leq \left[C_{T'} + C_{T-T', T'} \gamma^n_{T'}(0) + C_{T'} \int_0^{T'} E'\left(C_{T'}(T'-u)\right) \left\{ C_{T'} + C_{T-T'} \gamma^n_{T'}(T'-u)\right\} \frac{\drm u}{\sqrt{u}} \right] \left(1+|x|^{\zeta}\right),
      \]
      where we denoted $\dst \gamma^n_{T'}(t) = \sup_{x \in \IR^d} \frac{\sqrt{\Delta_t^{n,x}}}{1+|x|^\zeta} + \left\|\rho-\rho^n\right\|_\infty + C_{T'} \left(\left\|\Xi-\Xi^n\right\|_\infty + n^{-\frac{\varepsilon}{2}} \right)$.
\item[Step 5:] Taking the limit when $n \rightarrow \infty$.
      From Theorem \ref{thm:1st_behav} and equation (\ref{eq:LTB_step3}), we get easily:
      \begin{equation}
            \exists C > 0,\, \forall T >0,\, \forall x \in \IR^d,\, \left|w_T(0,x)\right| \leq C\left(1+|x|^\mu\right),
            \label{eq:LTB_step5.1}
      \end{equation}
      \begin{equation}
            \exists \chat, \nuhat >0,\, \forall T > 0,\, \forall x,y \in \IR^d,\, \left|w_T(0,x) - w_T(0,y)\right| \leq \chat\left(1+|x|^{\mu} + |y|^{\mu}\right) \erm^{-\nuhat T}.
            \label{eq:LTB_step5.2}
      \end{equation}
      The function $w_T$ has no reason to be $\Ccal^1$, so we use the mean value theorem:
      \[
      \frac{\left|w_T^n(0,x) - w_T^n(0,y)\right|}{\left(1+|x|^{\zeta} + |y|^\zeta\right) |x-y|} \leq \left[C_{T'} \int_0^{T'} E'\left(C_{T'}(T'-u)\right) \left\{ C_{T'} + C_{T-T'} \gamma^n_{T'}(T'-u)\right\} \frac{\drm u}{\sqrt{u}} + C_{T'} + C_{T-T', T'} \gamma^n_{T'}(0)\right].\]
      Our next goal will be to get an upperbound for $\gamma^n_{T'}$. Let $R > 0$; we denote $\alpha = \zeta-\mu$;
      \begin{align*}
      \IE\left[\left|\left(g-g^n\right)\left(X_{T-T'}^{n,0,X_{T'}^{n,t,x}}\right)\right|^2 \ind_{\left|X_{T-T'}^{n,0,X_{T'}^{n,t,x}}\right| > R}\right]
      & \leq \IE\left[\left|\left(g-g^n\right)\left(X_{T-T'}^{n,0,X_{T'}^{n,t,x}}\right)\right|^4\right]^{\frac{1}{2}} \IP\left(\left|X_{T-T'}^{n,0,X_{T'}^{n,t,x}}\right| > R\right)^{\frac{1}{2}} \\
      & \leq C \IE\left[\left(1+\left|X_{T-T'}^{n,0,X_{T'}^{n,t,x}}\right|^\mu\right)^4\right]^{\frac{1}{2}} \frac{\IE\left[\left|X_{T-T'}^{n,0,X_{T'}^{n,t,x}}\right|^{2\alpha}\right]^{\frac{1}{2}}}{R^\alpha} \\
      &\leq C_{T-T', T'} \frac{1+|x|^{2\mu+\alpha}}{R^\alpha}.
      \end{align*}
      By our proposition \ref{majo_X_Sp} and the Lemma 2.1 of \cite{MZ02}, we get:
      \[ \IE\left[\left|g\left(X_{T-T'}^{0,X_{T'}^{n,t,x}}\right) - g\left(X_{T-T'}^{n,0,X_{T'}^{n,t,x}}\right)\right|^2\right] \leq C_{T-T', T'} \left(1+|x|^{2\mu}\right)\left\|\Xi - \Xi^n\right\|_\infty^2.\]
      We can do exactly the same with $g$ and $g^n$ replaced by $v$ and $v^n$. This way, for any $R > 0$:
      \[ \sup_{x \in \IR^d} \frac{\sqrt{\Delta_{T'-t}^{n,x}}}{1+|x|^\zeta} \leq C_{T-T',T'} \left(R^{-\frac{\alpha}{2}} + \left\|\Xi-\Xi^n\right\|_\infty\right) + \left\|g-g^n\right\|_{\infty, \Bcal(0,R)} + \left\|v-v^n\right\|_{\infty, \Bcal(0,R)}.\]
      To sum up, we have:
      $\forall \alpha > 0,\, \forall T'>0,\, \exists C_{T'} >0,\, \forall T>T',\, \exists C_{T-T',T'}>0,\, \forall x,y \in \IR^d,\, \forall R > 0,\, \forall n \in \IN^*$,
      \begin{align*}
      & \left|w_T^n(0,x) - w_T^n(0,y)\right| \left(1+|x|^{\mu+\alpha} + |y|^{\mu+\alpha}\right)^{-1} |x-y|^{-1}\\
      & \leq \left[C_{T'} + C_{T-T', T'} \left(R^{-\frac{\alpha}{2}} + \left\|\Xi-\Xi^n\right\|_\infty + \left\|g-g^n\right\|_{\infty, \Bcal(0,R)} + \left\|v-v^n\right\|_{\infty, \Bcal(0,R)}\right) \right.\\
      & \pushright{ \left. +\ C_{T'} \int_0^{T'} E'\left(C_{T'}(T'-u)\right) \left\{ C_{T'} + C_{T-T', T'} \left(R^{-\frac{\alpha}{2}} + \left\|\Xi-\Xi^n\right\|_\infty + \left\|g-g^n\right\|_{\infty, \Bcal(0,R)} + \left\|v-v^n\right\|_{\infty, \Bcal(0,R)}\right)\right\} \frac{\drm u}{\sqrt{u}}\right]}
      \end{align*}
      We take the limit as $n$ and then $R$ go to the infinity: for every $\alpha > 0$,
      \[ \forall T'>0,\, \exists \Ccal_{T'} >0,\, \forall T>T',\, \forall x,y \in \IR^d,\,
      \left|w_T(0,x) - w_T(0,y)\right| 
      \leq \Ccal_{T'} \left(1+|x|^{\mu+\alpha} + |y|^{\mu+\alpha}\right)|x-y|.\]
\item[Step 6:] Convergence of $w_T(0,\bullet)$ when $T \rightarrow \infty$.\\
      Thanks to equation (\ref{eq:LTB_step5.1}), by a diagonal argument, there exists an increasing sequence $\left(T_i\right)_{i \in \IN}$, with limit $+\infty$; a function $w : D \rightarrow \IR$, with $D \subset \IR^d$ countable and dense, such that $\left(w_{T_i}\right)_i$ converges pointwise over $D$ to $w$. The equation (\ref{eq:LTB_step5.2}), with $x,y \in D$, tells us that $w$ is equal to a constant $L$ on $D$. The same equation, with $x \in \IR^d$ and $y \in D$, gives us that $w = L$ on the whole $\IR^d$. Let $K$ be a compact in $\IR^d$; we define $\Acal = \left\{ \restriction{w_T(0,\bullet)}{K} \middle| T > 1\right\} \subset \Ccal(K, \IR)$; it is equicontinuous, because $w_T(0,\bullet)$ is uniformly Lipschitz. For all $x \in K$, the set $\Acal(x) = \left\{ w_T(0,x) \middle| T > 1\right\}$ is bounded (see equation (\ref{eq:LTB_step5.1}). By Ascoli's theorem, $\Acal$ is relatively compact in $\Ccal(K,\IR)$. Let $x \in K$,
      \[ \left|w_{T_i}(0,x) - L\right| \leq \left|w_{T_i}(0,x) - w_{T_i}(0,0)\right| + \left|w_{T_i}(0,0) - L\right| \leq \chat \left(1+|x|^{\mu}\right) \erm^{-\nuhat T_i} + \left|w_{T_i}(0,0) - L\right|,\]
      which proves that the functions $w_{T_i}(0,\bullet)$ converge uniformly to $L$ on $K$. So $L$ is an accumulation point of $\Acal$. Since $\Acal$ is relatively compact, if we prove that $L$ is the unique accumulation point of $\Acal$, then, we get the uniform convergence of $w_T(0,\bullet)$ to $L$ on $K$. It works for any compact $K \subset \IR^d$, and $L$ does not depend on $K$. We conclude:
      \[ \forall x \in \IR^d,\, w_T(0,x) \underset{T \rightarrow \infty}{\longrightarrow} L.\]
\item[Step 7:] $L$ is the only accumulation point of $\Acal$.\\
      Let $w_{\infty, K}$ an accumulation point of $\Acal$, \emph{i.e.} $\left\| \restriction{w_{T'_i}(0,\bullet)}{K} - w_{\infty,K}\right\|_\infty \underset{i \rightarrow \infty}{\longrightarrow} 0$, for a sequence $\left(T'_i\right)$. Again, the equation (\ref{eq:LTB_step5.2}) tells us that $w_{\infty,K}$ is equal to a constant $L_K$ on $K$. Our goal is now to show that $L = L_K$. Be careful: now, $Y^n$ and $\Ytilde^n$ will follow the same equations as before, but their terminal conditions will be given at time $T+S$. So,
      \[ w_{T+S}^n(0,x) 
      = Y_0^n - \Ytilde_0^n 
      = Y_S^n - \Ytilde_S^n + \int_0^S \left\{ \psi^n\left(X_r^x, Z_r^n\right) - \psi^n\left(X_r^x, \Ztilde_r^n\right)\right\} \, \drm r - \int_0^S \left\{ Z_r^n - \Ztilde_r^n \right\} \, \drm W_r.\]
      We define the function:
      \[ \beta_{T,S}^n(t,x) = \left\{ \begin{array}{ll}
       \dfrac{\psi^n\left(x, \partial_x u_{T+S}^n(t,x) \sigma(x)\right) - \psi^n\left(x, \partial_x \utilde_{T+S}^n(t,x) \sigma(x)\right)}{\left|\partial_x w_{T+S}^n(t,x) \sigma(x)\right|^2} \left(\partial_x w_{T+S}^n(t,x) \sigma(x) \right)^*, \\ &\hspace{-4cm}\text{if } t < S \text{ and } \partial_x w_{T+S}^n(t,x) \neq 0, \\ 
       \dfrac{\psi^n\left(x, \partial_x u_{T+S}^n(S,x) \sigma(x)\right) - \psi^n\left(x, \partial_x \utilde_{T+S}^n(S,x) \sigma(x)\right)}{\left|\partial_x w_{T+S}^n(S,x) \sigma(x)\right|^2} \left(\partial_x w_{T+S}^n(S,x) \sigma(x) \right)^*, \\ &\hspace{-4cm} \text{if } t \geq S \text{ and } \partial_x w_{T+S}^n(S,x) \neq 0, \\
       0, & \hspace{-4cm} \text{otherwise.} \end{array} \right.\]
      This function is bounded by $K_z \left\|\sigma^{-1}\right\|_\infty$; $\dst\Wtilde_t = W_t - \int_0^t \beta_{T,s}^n\left(r,X_r^x\right) \, \drm r$ is a Brownian motion on $[0,S]$ under the probability $\IQ^{T,S,n,x}$.
      This way, we can write:
      \[ w_{T+S}^n(0,x) = \IE^{\IQ^{T,S,n,x}}\left[Y_S^n - \Ytilde_S^n\right] = \IE^{\IQ^{T,S,n,x}} \left[w_{T+S}^n\left(S,X_S^x\right)\right] = \IE^{\IQ^{T,S,n,x}}\left[w_T^n\left(0, X_S^x\right)\right] = \Pcal_S^n\left[w_T^n(0,\bullet)\right](x),\]
      where $\Pcal^n$ is the Kolmogorov semigroup of the SDE: $\drm U_t = \Xi\left(U_t\right) \, \drm t + \sigma\left(U_t\right)\left(\drm W_t + \beta_{T,S}^n\left(t,U_t\right) \, \drm t\right)$. Without any loss of generality, we can suppose that $T_i > T'_i$, for every $i \in \IN$. We take $T = T_i'$ and $S = T_i-T'_i$. This way:
      \[ \forall x \in \IR^d,\, \Pcal^n_{T_i-T_i'}\left[w_{T_i'}^n(0,\bullet)\right](x) = w_{T_i}^n(0,x) \underset{n \rightarrow \infty}{\longrightarrow} w_{T_i}(0,x) \underset{i \rightarrow \infty}{\longrightarrow} L.\]
      Now, to prove that $L=L_K$, it suffices to show that $\dst \lim_{i \rightarrow \infty} \limsup_{n \rightarrow \infty} \Pcal_{T_i-T'_i}^n\left[w_{T_i'}^n(0,\bullet)\right](x) = L_K$. We have:
      \[ \left|\Pcal_{T_i-T_i'}^n\left[w_{T_i'}^n(0,\bullet)\right](x)-L_K\right| \leq \left| \Pcal_{T_i-T_i'}^n\left[w_{T_i'}^n(0,\bullet)\right](x) - w_{T_i'}^n(0,x)\right| + \left|w_{T_i'}^n(0,x) - w_{T_i'}(0,x)\right| + \left|w_{T_i'}(0,x) - L_K\right|,\]
      \[ \limsup_{n \rightarrow \infty} \left|\Pcal_{T_i-T_i'}^n\left[w_{T_i'}^n(0,\bullet)\right](x)-L_K\right| \leq \limsup_{n \rightarrow \infty} \left| \Pcal_{T_i-T_i'}^n\left[w_{T_i'}^n(0,\bullet)\right](x) - w_{T_i'}^n(0,x)\right| + \left|w_{T_i'}(0,x) - L_K\right|.\]
      Thanks to Lemma \ref{lem:approx_lipschitz}, $\beta_{T,S}^n$ is the limit of $\Ccal^1$-functions, uniformly bounded, with bounded derivatives $\left(\beta_{T,S}^{n,m}\right)_{m \geq 0}$. We define $U^{m,x}$ as the solution of the SDE: $\drm U_t^{m,x} = \Xi\left(U_t^{m,x}\right) \, \drm t + \sigma\left(U_t^{m,x}\right) \left( \drm W_t + \beta^{n,m}_{T,S}\left(t,U_t^{m,x}\right) \, \drm t\right)$. Thus,
      \begin{align*}
      \left|\Pcal^n_{T_i-T_i'}\left[w_{T_i'}^n(0,\bullet)\right](x) - w_{T_i'}^n(0,x)\right|
      & = \left|\lim_{m \rightarrow \infty} \IE\left[w_{T_i'}^n\left(0, U_{T_i-T_i'}^{m,x}\right)\right] - w_{T_i'}^n(0,x)\right| \\
      & \leq \limsup_{m \rightarrow \infty} \IE\left[\left|w_{T_i'}^n\left(0,U_{T_i-T_i'}^{m,x}\right) - w_{T_i'}^n(0,x)\right|\right] \\
      & \leq \chat \left(1+ \limsup_{m \rightarrow \infty} \IE\left[\left|U_{T_i-T_i'}^{m,x}\right|^{\mu}\right] + |x|^{\mu}\right) \erm^{-\nuhat T_i'},
      \end{align*}
      and the constants $\chat$ and $\nuhat$ are independent of $n$, $m$ and $T_i'$. Thanks to proposition \ref{prop:sup_fini}, we have $\IE\left[\left|U_t^{m,x}\right|^{\mu}\right] \leq C\left(1+|x|^{\mu}\right)$, with $C$ independent of $n$, $m$ and $t$. So,
      \[ \left|\Pcal^n_{T_i-T_i'}\left[w_{T_i'}^n(0,\bullet)\right](x) - w_{T_i'}^n(0,x)\right| \leq \Ctilde \left(1+|x|^{\mu}\right) \erm^{-\nu T_i'},\]
      with $\Ctilde$ and $\nu$ independent of $n$ and $T_i'$. And we can conclude that $L=L_K$.
\item[Step 8:] Speed of convergence.
      Let $x \in \IR^d$ and $T > 0$. We have:
      \begin{align*}
      \left|w_T(0,x) - L \right|
      & = \lim_{V \rightarrow \infty} \left|w_T(0,x) - w_V(0,x)\right|
        = \lim_{V \rightarrow \infty} \lim_{n \rightarrow \infty} \left|w_T^n(0,x) - w_V^n(0,x)\right| \\
      & = \lim_{V \rightarrow \infty} \lim_{n \rightarrow \infty} \left|w_T^n(0,x) - \Pcal^n_{V-T}\left[w_T^n(0,\bullet)\right](x)\right|,
      \end{align*}
      where $\Pcal^n$ is the Kolmogorov semigroup of the SDE: $\drm U_t = \Xi\left(U_t\right) \, \drm t + \sigma\left(U_t\right) \left(\drm W_t + \beta_{T, V-T}^n\left(t,U_t\right) \, \drm t\right)$. Like before, we use Lemma \ref{lem:approx_lipschitz} and proposition \ref{prop:sup_fini}, and we get, for $\delta$ small enough:
      \begin{align*}
      \left|w_T(0,x)-L\right|
      & = \lim_{V \rightarrow \infty} \lim_{n \rightarrow \infty} \left|w_T^n(0,x) - \lim_{m \rightarrow \infty} \IE\left[w_T^n\left(0, U_{V-T}^{m,x}\right)\right]\right| \\
      & \leq \limsup_{V \rightarrow \infty} \limsup_{n \rightarrow \infty} \limsup_{m \rightarrow \infty} \IE\left[\left|w_T^n(0,x) - w_T^n\left(0, U_{V-T}^{m,x}\right)\right|\right] 
        \leq \Ctilde \left(1+|x|^{\mu}\right) \erm^{-\nu T}.
      \end{align*}
\end{description}
\end{proof}

\section{Application to Optimal ergodic problem and HJB equation}

\subsection{Optimal ergodic problem}

In this section, we apply our results to an ergodic control problems. The proofs of the following results are so similar to the ones of \cite{PDE} that we omit them. As before, we consider a process $X^x$ satisfying:
\[ X_t^x = x + \int_0^t \Xi\left(X_s^x\right) \, \drm s + \int_0^t \sigma\left(X_s^x\right) \, \drm W_s,\]
where $\Xi$ and $\sigma$ are Lipschitz continuous, $\Xi$ is weakly dissipative with $\langle \Xi(x),x \rangle \leq \eta_1 - \eta_2 |x|^2$, the function $x \mapsto \sigma(x)^{-1}$ is bounded and $|\sigma(x)|_F^2 \leq r_1 + r_2|x|^2$. Moreover, we pick $\mu \geq 2$ such that $\dfrac{\mu-1}{2} r_2 < \eta_2$. Let $U$ be a separable metric space; we call ``control'' any progressively measurable $U$-valued process. We consider some measurable functions satisfying the following assumptions:
\begin{itemize}
\item $R : U \rightarrow \IR^d$ is bounded;
\item $L : \IR^d \times U \rightarrow \IR$ is uniformly Lipschitz continuous in $x \in \IR^d$ w.r.t. $a \in U$;
\item $g : \IR^d \rightarrow \IR$ is continuous, has polynomial growth and is locally Lipschitz continuous: $\left|g(x)\right| \leq C \left(1+|x|^\mu\right)$ and $\left|g(x)-g(x')\right| \leq C\left(1+|x|^\mu + |x'|^\mu\right) |x-x'|$.
\end{itemize}
Let $X^x$ be the solution of the SDE (\ref{eq:SDE}); for any control $a$ and horizon $T > 0$, we set:
\[ \rho_T^{x,a} = \exp\left(\int_0^T \sigma \left(X_t^x\right)^{-1} R\left(a_t\right) \, \drm W_t - \frac{1}{2} \int_0^T \left|\sigma\left(X_t^x\right)^{-1} R\left(a_t\right)\right|^2 \, \drm t\right) \text{ and } \IP_T^{x,a} = \rho_T^{x,a} \IP \text{ over } \Fcal_T.\]
We define the finite horizon cost as: $\dst J^T(x,a) = \IE_T^{x,a}\left[\int_0^T L\left(X_t^x,a_t\right) \, \drm t\right] + \IE_T^{x,a}\left[g\left(X_t^x\right)\right]$, and the associated optimal control problem is to minimise $J^T(x,a)$ over the set of controls $a^T : \Omega \times [0,T] \rightarrow U$.
We also define an other cost, called ``ergodic cost'': $\dst J(x,a) = \limsup_{T \rightarrow \infty} \frac{1}{T} \IE_T^{x,a}\left[\int_0^T L\left(X_t^x, a_t\right) \, \drm t\right]$, and the associated optimal control problem is to minimise $J(x,a)$ over the set of controls $a : \Omega \times \IR_+ \rightarrow U$.
Due to Girsanov's theorem, $\dst W_t^{a,x} = W_t - \int_0^t \sigma\left(X_s^x\right)^{-1} R\left(a_s\right) \, \drm s$ is a Brownian motion on $[0,T]$ under $\IP_T^{x,a}$ and
\[ \drm X_t^x = \left[\Xi\left(X_t^x\right) + R\left(a_t\right)\right] \, \drm t + \sigma\left(X_t^x\right) \, \drm W_t^{x,a}.\]
We define the Hamiltonian in the following way:
\begin{equation}
 \psi(x,z) = \inf_{a \in U} \left\{ L(x,a) + z \sigma(x)^{-1} R(a) \right\}.
\label{eq:Filippov}
\end{equation}
We recall that, if this infimum is attained for every $x$ and $z$, by Filippov's theorem (see \cite{SW67}), there exists a measurable function $\gamma : \IR^d \times \left(\IR^d\right)^* \rightarrow U$ such that: $\dst \psi(x,z) = L(x,\gamma(x,z))+z\sigma(x)^{-1} R(\gamma(x,z))$.

\begin{lem}
The Hamiltonian $\psi_0$ satisfies:
\begin{itemize}
\item $\forall x \in \IR^d,\, \left|\psi_0(x,0)\right| \leq \left(\|L\|_\lip + |L(0,0)|\right) (1+|x|)$;
\item $\forall x,x' \in \IR^d,\, \forall z,z' \in \left(\IR^d\right)^*,\, \left|\psi_0(x,z)- \psi_0(x',z')\right| \leq \|L\|_\lip |x-x'| + \|R\|_\infty \left|z\sigma(x)^{-1} - z'\sigma(x')^{-1}\right|$.
\end{itemize}
\end{lem}

\begin{lem}
For every control $a$, we have: $J^T(x,a) \geq Y_0^{T,x}$, where $Y^{T,x}$ is part of the solution of the finite horizon BSDE:
\begin{equation}
 Y_t^{T,x} = g\left(X_T^x\right) + \int_t^T \psi\left(X_s^x, Z_s^{T,x}\right) \, \drm s - \int_t^T Z_s^{T,x} \, \drm W_s, \ \ \ \forall t \in [0,T].
\label{eq:BSDE_ceo}
\end{equation}
Moreover, if the infimum is attained for every $x$ and $z$ in equation (\ref{eq:Filippov}), we have $J\left(x,\overline{a}^T\right) = Y_0^{T,x}$, where we set $\overline{a}^T_t = \gamma\left(X_t^{x,\overline{a}^T},\nabla u\left(t,X_t^{x,\overline{a}^T}\right) \sigma\left(X_t^{x, \overline{a}^T}\right)\right).$
\end{lem}

\begin{lem}
For every control $a$, we have: $J(x,a) \geq \lambda$, where $\lambda$ is part of the solution of the EBSDE:
\[ Y_t^x = Y_T^x + \int_t^T \left\{\psi\left(X_s^x, Z_s^x\right) - \lambda\right\} \, \drm s - \int_t^T Z_s^x \, \drm W_s, \ \ \ \forall 0 \leq t \leq T < \infty.\]
Moreover, if the infimum is attained for every $x$ and $z$ in equation (\ref{eq:Filippov}), we have $J\left(x,\overline{a}^T\right) = \lambda$, where we set $\overline{a}_t = \gamma\left(X_t^{x,\overline{a}},\nabla v\left(X_t^{x,\overline{a}}\right) \sigma\left(X_t^{x, \overline{a}}\right)\right).$
\end{lem}

\begin{thm}
For every control $a$, we have: 
\[\liminf_{T \rightarrow \infty} \frac{J^T(x,a)}{T} \geq \lambda.\]
Moreover, if the infimum is attained for every $x$ and $z$ in equation (\ref{eq:Filippov}), we have 
\[\left| J^T\left(x,\overline{a}^T\right) - J\left(x,\overline{a}\right) T + Y_0^x + L\right| \leq C\left(1+|x|^2\right) \erm^{-\nu T}.\]
\end{thm}

\begin{proof}~\\
This is a straightforward consequence of the previous lemmas and Theorem \ref{thm:LTB}.
\end{proof}

\begin{rmq}
All the results of this subsection can be rephrased in terms of viscosity solution of PDEs (\ref{eq:EPDE}) and (\ref{eq:PDE}).
\end{rmq}

\subsection{Large time behaviour of viscosity solution of HJB equation}

We consider the ergodic PDE:
\begin{equation}
\Lcal v(x) + \psi\left(x, \nabla v(x) \sigma(x) \right) - \lambda = 0,
\label{eq:EPDE}
\end{equation}
where $\Lcal$ is the generator of the Kolmogorov semigroup of $X^x$, solution of (\ref{eqintro:SDE}). We recall that the couple $(v,\lambda)$ is a viscosity subsolution (resp. supersolution) if:
\begin{itemize}
\item $v : \IR^d \rightarrow \IR$ is a continuous function with polynomial growth;
\item for any function $\phi \in \Ccal^2\left(\IR^d, \IR\right)$, for every $x \in \IR^d$ of local maximum (resp. minimum) of $v-\phi$:
\[ \Lcal \phi(x) + \psi(x, \nabla \phi(x) \sigma(x)) - \lambda \geq 0 \ \ \ \text{(resp. $\leq 0$)}.\]
\end{itemize}

\begin{prop}[Existence of ergodic viscosity solution]
Under Assumption \ref{ass_EBSDE}, the couple $\left(\vbar, \lambdabar\right)$ obtained with the solution given in Theorem \ref{thm:existence_EBSDE} is a viscosity solution of equation (\ref{eq:EPDE}).
\end{prop}

\begin{proof}~\\
Note that we already know that $\vbar$ is continuous and has quadratic growth. The proof of this result is classical and can easily be adapted from Theorem 4.3 of \cite{Pard96}.
\end{proof}

\begin{prop}[Uniqueness of ergodic viscosity solution]
Let $p > 0$; we suppose that $\sqrt{r_2} K_z \left\|\sigma^{-1}\right\|_\infty + [(p \vee 2)-1] \dfrac{r_2}{2} < \eta_2$ and that Assumption \ref{ass_EBSDE} holds true.\\
Then uniqueness holds for viscosity solutions $(v,\lambda)$ of equation (\ref{eq:EPDE}) in the class of viscosity solutions such that $\exists a \in \IR^d, v(a) = \vtilde(a)$ and $v$ and $\vtilde$ have polynomial growth of at most degree $p$.
\end{prop}

\begin{proof}~\\
The proof is quite the same as in Lemma 3.18 of \cite{HM16}. Let $(v,\lambda)$ and $(\vtilde, \lambdatilde)$ be two viscosity solutions of (\ref{eq:EPDE}). We fix $T > 0$, and we consider the solution $\left(Y^{T,x}, Z^{T,x}\right)$ of the BSDE
\[ Y_t^{T,x} = v\left(X_T^x\right) + \int_t^T \left\{ \psi\left(X_s^x, Z_s^{T,x}\right) - \lambda\right\} \, \drm s - \int_t^T Z_s^{T,x} \, \drm W_s, \ \ \ t \in [0,T].\]
Because this BSDE has a unique solution, we can claim that $v(x) = Y_0^{T,x}$. We define a couple $\left(\Ytilde^{T,x}, \Ztilde^{T,x}\right)$ by replacing in the previous equation $(v,\lambda)$ by $(\vtilde, \lambdatilde)$.
Then, for any $T > 0$ and $x \in \IR^d$:
\[ \left(v-\vtilde\right)(x) = \left(v-\vtilde\right)\left(X_T^x\right) + \int_0^T \left\{ \psi\left(X_s^x, Z_s^{T,x}\right) - \psi\left(X_s^x, \Ztilde_s^{T,x}\right) \right\} \, \drm s - \left(\lambda-\lambdatilde\right) T - \int_0^T \left\{ Z_s^{T,x} - \Ztilde_s^{T,x}\right\} \, \drm W_s.\]
We set
\[ \beta_s = 
  \left\{ \begin{array}{ll} 
          \dfrac{\psi\left(X_s^x, Z_s^{T,x}\right) - \psi\left(X_s^x, \Ztilde_s^{T,x}\right)}{\left|Z_s^{T,x} - \Ztilde_s^{T,x}\right|^2} \left(Z_s^{T,x} - \Ztilde_s^{T,x}\right)^*, & \text{if } Z_s^{T,x} \neq \Ztilde_s^{T,x}, \\
          0, & \text{otherwise.}
          \end{array} \right.\]
Since the process $\beta$ is bounded by $K_z \left\|\sigma^{-1}\right\|_\infty$, by Girsanov's theorem, there exists a new probability measure $\IQ^T$ equivalent to $\IP$ and under which $W - \int_0^\bullet \beta_s \, \drm s$ is a Brownian motion. Then:
\[ \frac{\left(v-\vtilde\right)(x)}{T} = \frac{\IE^{\IQ^T}\left[\left(v-\vtilde\right)\left(X_T^x\right)\right]}{T} - \left(\lambda-\lambdatilde\right).\]
Thanks to proposition \ref{prop:sup_fini} and the polynomial growth of $v$ and $\vtilde$, letting $T \rightarrow \infty$ gives us $\lambda = \lambdatilde$.
Applying the same argument as that in Theorem \ref{thm:uniq_EBSDE}, we deduce the uniqueness claimed.
\end{proof}

We recall the Hamilton-Jacobi-Bellman equation:
\begin{equation}
 \left\{ \begin{array}{ll}
           \partial_t u(t,x) + \Lcal u(t,x) + \psi\left(x,\nabla u(t,x) \sigma(x)\right) = 0 & \forall (t,x) \in \IR_+ \times \IR^d, \\
           u(T,x) = g(x) & \forall x \in \IR^d,
           \end{array} \right.
\label{eq:PDE}
\end{equation}
whose viscosity solution is linked to the BSDE (\ref{eq:BSDE_ceo}) via $Y_t^{T,x} = u\left(T-t, X_t^x\right)$. We can rephrase Theorem \ref{thm:LTB} as:

\begin{thm}
We consider the equations (\ref{eq:EPDE}) and (\ref{eq:PDE}). We suppose that Assumption \ref{ass_EBSDE} holds true; moreover, we assume that $g : \IR^d \rightarrow \IR^d$ satisfies $\forall x,x' \in \IR^d,\, |g(x)| \leq C\left(1+|x|^\mu\right)$ and $|g(x)-g(x')| \leq C\left(1+|x|^\mu+|x'|^\mu\right)|x-x'|$ with $\dst \sqrt{r_2} K_z \left\|\sigma^{-1}\right\|_\infty + \frac{\mu-1}{2}r_2 < \eta_2$ and $\mu \geq 2$.\\
Then, there exists $L \in \IR$, such that: $\forall x \in \IR^d,\, u(T,x) - \lambda T - v(x) \underset{T \rightarrow \infty}{\longrightarrow} L$. Furthermore,
\[ \forall x \in \IR^d,\, \forall T > 0,\, \left|u(T,x)-\lambda T - v(x) - L\right| \leq C\left(1+|x|^\mu\right)\erm^{-\nu T}.\]
\end{thm}

\bibliographystyle{alpha}
\bibliography{bibli}

\appendix

\section{Appendix about the Lipschitz continuity of $v^\alpha$}

We consider the forward-backward system, where $X$ is a vector in $\IR^d$, $Y$ a real, $Z$ a row with $d$ real coefficients, and $W$ a Brownian motion in $\IR^d$:
\begin{equation}
X_s^{t,x} = x + \int_t^s{b\left(X_r^{t,x}\right) \, \drm r} + \int_t^s{\sigma\left(X_r^{t,x}\right) \, \drm W_r},
\label{SDE_app}
\end{equation}
\begin{equation}
Y_s^{t,x} = g\left(X_T^{t,x}\right) + \int_s^T{f\left(X_r^{t,x}, Y_r^{t,x}, Z_r^{t,x}\right) \, \drm r} - \int_s^T{Z_r^{t,x} \, \drm W_r}.
\label{BSDE_app}
\end{equation}
The first result is an adaptation to our case of the Theorem 3.1 of \cite{MZ02}.

\begin{prop}
We make the following assumptions:
\begin{itemize}
\item $\sigma \in \Ccal^1\left(\IR^d, \GL_d(\IR)\right)$ and $b \in \Ccal^1\left(\IR^d, \IR^d\right)$ have bounded derivatives;
\item the function $\sigma(\bullet)^{-1}$ is bounded;
\item $g \in \Ccal^1\left(\IR^d,\IR\right)$ has bounded derivatives;
\item if we define $h : (x,y,p) \mapsto f(x,y,p\sigma(x))$, then $h \in \Ccal^1\left(\IR^d \times \IR \times \left(\IR^d\right)^*, \IR\right)$ has bounded derivatives (the bounds of the derivatives of $h$ will be denoted $H_x$, $H_y$ and $H_p$ w.r.t. $x$, $y$ and $p$).
\end{itemize}
Under those assumptions, we already know that the forward-backward system (\ref{SDE_app})-(\ref{BSDE_app}) admits a unique solution $\left(Y^{t,x}, Z^{t,x}\right)$.\\
Moreover, we assume:
\begin{itemize}
\item $\forall x \in \IR^d,\, |f(x,0,0)| \leq M_f\left(1+|x|^\mu\right)$;
\item $\dst \forall p \geq 2,\, \forall x \in \IR^d, \sup_{|\varepsilon| \leq 1} \IE\left[\left|Y^{t,x+\varepsilon}\right|^{*,p}_{t,T}\right] < \infty$.
\end{itemize}
Then, the function $u : (t,x) \mapsto Y_t^{t,x}$ is of class $\Ccal^1$ w.r.t. $x$ and 
\[ \forall s \in [t,T],\, Z_s^{t,x} = \partial_x u\left(s,X_s^{t,x}\right) \sigma\left(X_s^{t,x}\right) \ \IP\text{-a.s.}\]
\end{prop}

\begin{proof} We will prove the proposition in the case $d=1$, to simplify presentation.
\begin{description}
\item[Step 1:] we have, for all $p \geq 1$:
      \begin{equation}
      \IE\left[\left(\int_t^T \left|Z_r^{t,x}\right|^2 \, \drm r\right)^p\right] \leq C \erm^{2ap(T-t)} \left\{\IE\left[\left|Y^{t,x}\right|^{*,2p}_{t,T}\right] + C_T\left(1+|x|^{2p\mu}\right)\right\},
      \label{eq:Z_Lp}
      \end{equation}
      where $C$ comes from the Lemma 3.1 of \cite{Rennes03} and $C_T$ depends only on $p$, $\mu$, $T$, $T-t$, $r_1$, $\|\sigma\|_\lip$, $|b(0)|$, $\|b\|_\lip$ and $M_f$. The last assumption of our proposition allows us to bound $\dst \IE\left[\left(\int_t^T \left|Z_r^{t,x+\varepsilon}\right|^2 \, \drm r\right)^p\right]$ independently of $\varepsilon$.
\item[Step 2:] we will show that: 
$\dst \forall p \geq 2,\, \forall x \in \IR^d,\, \IE\left[\left|Y^{t,x+\varepsilon}-Y^{t,x}\right|^{*,p}_{t,T} + \left( \int_t^T \left|Z_s^{t,x+\varepsilon}-Z_s^{t,x}\right|^2 \, \drm s\right)^{\frac{p}{2}}\right] \underset{\varepsilon \rightarrow 0}{\longrightarrow} 0$.\\
      Let $(t,x) \in [0,T] \times \IR$ be fixed. We define, for any $\varepsilon \neq 0$,
      \[ \nabla X^\varepsilon = \frac{X^{t,x+\varepsilon}-X^{t,x}}{\varepsilon},\, \nabla Y^\varepsilon = \frac{Y^{t,x+\varepsilon}-Y^{t,x}}{\varepsilon},\, \nabla Z^\varepsilon = \frac{Z^{t,x+\varepsilon}-Z^{t,x}}{\varepsilon}. \]
      We also define the solution of the variational system:
      \begin{equation}
      \nabla X_s = 1 + \int_t^s \partial_x b\left(X_r^{t,x}\right) \nabla X_r \, \drm r + \int_t^s \partial_x \sigma\left(X_r^{t,x}\right) \nabla X_r \, \drm W_r;
      \label{eq:var_SDE}
      \end{equation}
      \begin{equation}
      \nabla Y_s = \partial_x g\left(X_T^{t,x}\right) \nabla X_T + \int_s^T \left[\partial_x f\left(\Theta_r^{t,x}\right) \nabla X_r + \partial_y f\left(\Theta_r^{t,x}\right) \nabla Y_r + \partial_z f\left(\Theta_r^{t,x}\right) \nabla Z_r \right] \, \drm r - \int_s^T \nabla Z_r \, \drm W_r,
      \label{eq:var_BSDE}
      \end{equation}
      where $\Theta$ stands for $(X,Y,Z)$. For every $s \in [t,T]$, we have:
      \begin{equation}
      \nabla Y_s^\varepsilon = g_x^\varepsilon \nabla X_T^\varepsilon + \int_s^T \left[f_x^\varepsilon(r) \nabla X_r^\varepsilon + f_y^\varepsilon(r) \nabla Y_r^\varepsilon + f_z^\varepsilon(r) \nabla Z_r^\varepsilon \right] \, \drm r - \int_s^T \nabla Z_r^\varepsilon \, \drm W_r,
      \label{eq:var_app_BSDE}
      \end{equation} 
      \[\text{where } g_x^\varepsilon = \int_0^1 \partial_x g\left(X_T^{t,x}+w\varepsilon \nabla X_T^\varepsilon\right) \, \drm w,\,
         f_x^\varepsilon(r) = \int_0^1 \partial_x f\left(\Theta_r^{t,x}+w\varepsilon \nabla \Theta_r^\varepsilon\right) \, \drm w,\]
      \[ f_y^\varepsilon(r) = \int_0^1 \partial_y f\left(\Theta_r^{t,x}+w\varepsilon \nabla \Theta_r^\varepsilon\right) \, \drm w \, \text{ and } \,
         f_z^\varepsilon(r) = \int_0^1 \partial_z f\left(\Theta_r^{t,x}+w\varepsilon \nabla \Theta_r^\varepsilon\right) \, \drm w.\]
      In order to apply the proposition 3.2 of \cite{Rennes03} later; we set 
      \[ i(r,y,z) := \int_0^1 \partial_x f\left(\Theta_r^{t,x}+w\varepsilon \nabla \Theta_r^\varepsilon\right) \, \drm w \, \nabla X_r^\varepsilon + \int_0^1 \partial_y f\left(\Theta_r^{t,x}+w\varepsilon \nabla \Theta_r^\varepsilon\right) \, \drm w\, y + \int_0^1 \partial_z f\left(\Theta_r^{t,x}+w\varepsilon \nabla \Theta_r^\varepsilon \right)\, \drm w\, z.\]
      This way, we can check that:
      \[ |i(r,y,z)| \leq \underbrace{\left(H_x+H_p \left(\left|Z_r^{t,x}\right| + \left|Z_r^{t,x+\varepsilon}\right|\right) \left\|\sigma^{-1}\right\|_\lip\right) \left|\nabla X_r^\varepsilon\right|}_{=i_r} + H_y |y| + H_p \left\|\sigma^{-1}\right\|_\infty |z|,\]
      \[ I := \int_t^T i_r \, \drm r \leq C \left(1 + \int_t^T \left|Z_r^{t,x}\right| \, \drm r + \int_t^T \left|Z_r^{t,x+\varepsilon}\right| \, \drm r\right) \left|\nabla X^\varepsilon\right|^*_{t,T},\]
      \[ \IE\left[|I|^p\right] \leq C \left(1+\IE\left[\left(\int_t^T \left|Z_r^{t,x}\right|^2 \, \drm r\right)^p\right]^{\frac{1}{2}} + \IE\left[\left(\int_t^T \left|Z_r^{t,x+\varepsilon}\right|^2 \, \drm r\right)^p\right]^{\frac{1}{2}}\right) \IE\left[\left|\nabla X^\varepsilon\right|^{*,2p}_{t,T}\right]^{\frac{1}{2}}.\]
      By the step 1 and the converence of $\nabla X^\varepsilon$ in $\Scal^p$, $\IE\left[|I|^p\right]$ can clearly be bounded independently of $\varepsilon$. Using the proposition 3.2 of \cite{Rennes03}, we are now able to say that, for $a$ great enough,
      \[ \IE\left[\left|\nabla Y^\varepsilon\right|^{*,p}_{t,T} + \left(\int_t^T \left|\nabla Z_r^\varepsilon\right|^2 \, \drm r\right)^{\frac{p}{2}}\right] \leq C \erm^{ap(T-t)}\left(\left\|\partial_x g\right\|_\infty \IE\left[\left|\nabla X_T^\varepsilon\right|^p\right] + \IE\left[I^p\right]\right).\]
\item[Step 3:] some dominated convergences.
      Thanks to the previous step, for all $p \geq 2$, 
      \[ \IE\left[\left|g_x^\varepsilon - g_x^0\right|^p\right] \underset{\varepsilon \rightarrow 0}{\longrightarrow} 0, \hspace{5mm}
         \IE\left[\left(\int_t^T \left|f_x^\varepsilon(r)-f_x^0(r)\right|^2 \, \drm r\right)^{\frac{p}{2}}\right]\underset{\varepsilon \rightarrow 0}{\longrightarrow} 0,\]
      \[ \IE\left[\int_t^T \left|f_y^\varepsilon(r)-f_y^0(r)\right|^p \, \drm r\right]\underset{\varepsilon \rightarrow 0}{\longrightarrow} 0, \hspace{5mm}
         \IE\left[\int_t^T \left|f_z^\varepsilon(r)-f_z^0(r)\right|^p \, \drm r \right]\underset{\varepsilon \rightarrow 0}{\longrightarrow} 0.\]
      The first, third and fourth convergences are easy to prove since $f$ have bounded derivatives. Also, the inequality $\left|\partial_x f\left(\Theta_r^{t,x} + w\varepsilon \nabla \Theta_r^\varepsilon\right) - \partial_x f \left(\Theta_r^{t,x}\right)\right| \leq 2 H_x + H_p \left\|\sigma^{-1}\right\|_\lip \left(2 \left|Z_r^{t,x}\right| + \left|Z_r^{t,x+\varepsilon}\right|\right)$ let us use dominated convergence.
\item[Step 4:] we will show that: 
$\dst \forall x \in \IR^d,\, \IE\left[\left|\nabla Y^\varepsilon-\nabla Y\right|^{*,2}_{t,T} + \left( \int_t^T \left|\nabla Z_s^\varepsilon- \nabla Z_s\right|^2 \, \drm s\right)\right] \underset{\varepsilon \rightarrow 0}{\longrightarrow} 0$.\\
      Combining equations (\ref{eq:var_BSDE}) and (\ref{eq:var_app_BSDE}), we get:
      \begin{align*}
      \Delta Y_s^\varepsilon &= g_x^\varepsilon \Delta X_T^\varepsilon + \left(g_x^\varepsilon - g_x^0\right) \nabla X_T + \int_s^T \left[f_x^\varepsilon(r) \nabla X_r^\varepsilon - f_x^0(r) \nabla X_r + f_y^\varepsilon(r) \Delta Y_r^\varepsilon + f_z^\varepsilon(r) \Delta Z_r^\varepsilon + \beta^\varepsilon(r)\right] \, \drm r \\
      &\pushright{- \int_s^T \Delta Z_r^\varepsilon \, \drm W_r,}
      \end{align*}
      where we set $\beta^\varepsilon(s) = \left[f_y^\varepsilon(r)-f_y^0(r)\right] \nabla Y_r + \left[f_z^\varepsilon(r)-f_z^0(r)\right] \nabla Z_r$, $\Delta X_s^\varepsilon = \nabla X_s^\varepsilon -  \nabla X_s$, $\Delta Y_s^\varepsilon = \nabla Y_s^\varepsilon -  \nabla Y_s$ and $\Delta Z_s^\varepsilon = \nabla Z_s^\varepsilon -  \nabla Z_s$. The Lemma 2.2 of \cite{MZ02} gives:
      \[ \IE\left[\left|\Delta Y^\varepsilon\right|^{*,2}_{t,T} + \int_t^T \left|\Delta Z_r^\varepsilon\right|^2 \, \drm r\right] \leq C \IE\left[\left|g_x^\varepsilon \Delta X_T^\varepsilon + \left(g_x^\varepsilon - g_x^0\right) \nabla X_T\right|^2 + \int_t^T \left|f_x^\varepsilon(r) \nabla X_r^\varepsilon - f_x^0(r) \nabla X_r + \beta^\varepsilon(r)\right|^2 \, \drm r\right].\]
      First, $\dst \IE\left[\left|g_x^\varepsilon \Delta X_T^\varepsilon + \left(g_x^\varepsilon - g_x^0\right) \nabla X_T\right|^2\right] \leq 2 \left\|\partial_x g\right\|_\infty \IE\left[\left|\Delta X_T^\varepsilon\right|^2\right] + 2 \IE\left[\left|g_x^\varepsilon-g_x^0\right|^4\right]^{\frac{1}{2}} \IE\left[\left|\nabla X_T\right|^4\right]^{\frac{1}{2}} \underset{\varepsilon \rightarrow 0}{\longrightarrow} 0$,
      by Lemma 2.1 of \cite{MZ02}. By dominated convergence, we get $\dst \IE\left[\int_t^T \left|\beta^\varepsilon(r)\right|^2 \, \drm r\right] \underset{\varepsilon \rightarrow 0}{\longrightarrow} 0$, and then
      \begin{align*}
      \IE\left[\int_t^T \left|f_x^\varepsilon(r) \nabla X_r^\varepsilon - f_x^0(r) \nabla X_r\right|^2 \, \drm r\right] 
      \leq&\  2 \IE\left[\left|\Delta X^\varepsilon\right|^{*,4}_{t,T}\right]^{\frac{1}{2}} \IE\left[\left(\int_t^T \left|f_x^0(r)\right|^2 \, \drm r\right)^2\right]^{\frac{1}{2}} \\
      &+ 2 \IE\left[\left|\nabla X^\varepsilon \right|^{*,4}_{t,T}\right]^{\frac{1}{2}} \IE\left[\left(\int_t^T \left|f_x^\varepsilon(r)-f_x^0(r)\right|^2 \, \drm r\right)^2\right]^{\frac{1}{2}} \underset{\varepsilon \rightarrow 0}{\longrightarrow} 0.
      \end{align*}
\item[Step 5:] we conclude that $\partial_x u$ exists and that $\partial_x u(t,x) = \nabla Y_t^{t,x}$, for all $(t,x)$. See Theorem 3.1 of \cite{MZ02}.
\item[Step 6:] we will show that $\partial_x u$ is continuous. Let $\left(t_i,x_i\right) \in [0,T] \times \IR^d$, $i=1,2$, with $t_1 < t_2$. To simplify, we write:
      \[ \Theta^i = \left(X^i, Y^i, Z^i\right) = \left(X^{t_i,x_i}, Y^{t_i,x_i}, Z^{t_i,x_i}\right),\,\,
         f_x^i(r) = \partial_x f\left(\Theta_r^i\right),\,\,
         f_y^i(r) = \partial_y f\left(\Theta_r^i\right),\,\,
         f_z^i(r) = \partial_z f\left(\Theta_r^i\right),\]
      \[ g_x^i = \partial_x g\left(X_T^i\right),\,\,
         b_x^i(r) = \partial_x b\left(X_r^i\right) \text{ and }\,
         \sigma_x^i(r) = \partial_x \sigma\left(X_r^i\right).\]
      We set $\Deltatilde X_r = \nabla X_r^1 - \nabla X_r^2$, $\Deltatilde Y_r = \nabla Y_r^1 - \nabla Y_r^2$, $\Deltatilde Z_r = \nabla Z_r^1 - \nabla Z_r^2$ and for any function $\varphi$, $\Deltatilde_{12}[\varphi] = \varphi^1 - \varphi^2$.
      \begin{align*}
      & \left|\partial_x u\left(t_1,x_1\right) - \partial_x u\left(t_2,x_2\right)\right| \\
      & = \left|\IE\left[g_x^1\nabla X_T^1 + \int_{t_1}^T\left\{f_x^1(r) \nabla X_r^1 + f_y^1(r) \nabla Y_r^1 + f_z^1(r) \nabla Z_r^1\right\} \, \drm r\right] \right. \\
      &\pushright{ \left.- \IE\left[g_x^2\nabla X_T^2 + \int_{t_2}^T\left\{f_x^2(r) \nabla X_r^2 + f_y^2(r) \nabla Y_r^2 + f_z^2(r) \nabla Z_r^2\right\} \, \drm r\right]\right|}\\
      & \leq \IE\left[\left|g_x^1 \nabla X_T^1 - g_x^2 \nabla X_T^2\right|\right] + \IE\left[\int_{t_1}^{t_2} \left\{ \left|f_x^1(r)\right|\left|\nabla X_r^1\right| + \left|f_y^1(r)\right|\left|\nabla Y_r^1\right| + \left|f_z^1(r)\right|\left|\nabla Z_r^1\right|\right\} \, \drm r\right] \\
      &\pushright{ + \IE\left[\int_{t_2}^T\left\{ \left|f_x^1(r) \Deltatilde X_r\right| + \left|\Deltatilde_{12}\left[f_x\right](r)\right| \left|\nabla X_r^2\right| + \left\|\partial_y f\right\|_\infty \left|\Deltatilde Y_r\right| + \left|\Deltatilde_{12}\left[f_y\right](r)\right|\left|\nabla Y_r^2\right| \right.\right. }\\
      &\hspace{9.5cm} \left. \left. + \left\|\partial_z f\right\|_\infty \left|\Deltatilde Z_r\right| + \left|\Deltatilde_{12}\left[f_z\right](r)\right|\left|\nabla Z_r^2\right|\right\}\, \drm r\right]
      \end{align*}
      First of all, $\dst \IE\left[\int_{t_1}^{t_2} \left\{ \left|f_x^1(r)\right|\left|\nabla X_r^1\right| + \left\|\partial_y f\right\|_\infty \left|\nabla Y_r^1\right| + \left\|\partial_z f\right\|_\infty \left|\nabla Z_r^1\right|\right\} \, \drm r\right] \underset{t_1 \rightarrow t_2}{\longrightarrow} 0$, by dominated convergence, because $\dst \IE\left[\left|\nabla X^1\right|^{*,2}_{t,T} + \left|\nabla Y^1\right|^{*,2}_{t,T} + \int_t^T \left|\nabla Z_s^1\right|^2 \, \drm s + \int_t^T \left|Z_s^1\right|^2 \, \drm s\right] < \infty$. Using Lemma 2.2 of \cite{MZ02},
      \begin{align*}
      \IE & \left[\left|\Deltatilde Y\right|^{*,2}_{t_2,T} + \int_{t_2}^T \left|\Deltatilde Z_s\right|^2 \, \drm s\right] \leq C \left\{ \left\|\partial_x g\right\|_\infty^2 \IE\left[\left|\Deltatilde X_T \right|^2\right] + \IE\left[\left|\Deltatilde_{12}\left[g_x\right]\right|^2 \left|\nabla X_T^2\right|^2\right] \right.\\
      & \left. +\, \IE\left[\int_{t_2}^T \left( \left|f_x^1(r)\right|^2 \left|\Deltatilde X_r\right|^2 + \left|\Deltatilde_{12}\left[f_x\right](r)\right|^2\left|\nabla X_r^2\right|^2 + \left|\Deltatilde_{12}\left[f_y\right](r)\right|^2\left|\nabla Y_r^2\right|^2 + \left|\Deltatilde_{12}\left[f_z\right](r)\right|^2\left|\nabla Z_r^2\right|^2 \right) \, \drm r\right]\right\}.
      \end{align*}
      We can adapt the step 2, we replace $(t,x)$ by $\left(t_2,x_2\right)$ and $x+h$ by $X_{t_2}^{t_1,x_1}$ and get:
      \[ \forall p \geq 2,\, \IE\left[\left|X^1-X^2\right|^{*,p}_{t_2,T} + \left|Y^1-Y^2\right|^{*,p}_{t_2,T} + \left(\int_{t_2}^T  \left|Z_s^1 - Z_s^2\right|^2\, \drm s\right)^{\frac{p}{2}}\right] \longrightarrow 0.\]
      By dominated convergence, for $\varphi$ bounded, $\dst\IE\left[\left|\Deltatilde_{12}[\varphi]\right|^{*,p}_{t_2,T}\right] \longrightarrow 0$. Since $\dst \IE\left[\left(\int_{t_2}^T \left|Z_r^i\right|^2 \, \drm r\right)^3\right] < \infty$, we have $\dst \IE\left[\int_{t_2}^T \left\{\left|f_x^1(r) \Deltatilde X_r\right|^2 + \left|\Deltatilde_{12}\left[f_x\right](r) \nabla X_r^2\right|^2\right\} \, \drm r\right] \longrightarrow 0$. To sum up, we have shown that $\dst \left|\partial_x u\left(t_1,x_1\right) - \partial_x u\left(t_2,x_2\right)\right| \longrightarrow 0$ when $t_1 \rightarrow t_2$ and $x_1 \rightarrow x_2$. We can prove it when $t_2 \rightarrow t_1$ and $x_2 \rightarrow x_1$: $\partial_x u$ is continuous on $[0,T] \times \IR^d$.
\item[Step 7:] to show the relation 
$\dst \forall s \in [t,T],\, Z_s^{t,x} = \partial_x u\left(s,X_s^{t,x}\right) \sigma\left(X_s^{t,x}\right) \ \IP\text{-a.s.}$, we can do the same as in the Theorem 3.1 of \cite{MZ02}. We approximate $b$, $\sigma$, $h$ and $g$ by functions $b^\varepsilon$, $\sigma^\varepsilon$, $h^\varepsilon$ and $g^\varepsilon$ which are of class $\Ccal^\infty$ with bounded derivatives (and the bound is independent of $\varepsilon$), and converge uniformly. Then, $\left(\sigma^\varepsilon\right)^{-1}$ converges uniformly to $\sigma^{-1}$. We define a function $f^\varepsilon$ (we want it to approximate $f$ and to be smooth) by $f^\varepsilon(x,y,z) = h^\varepsilon\left(x,y,z\sigma^\varepsilon(x)^{-1}\right)$ and it satisfies $\left|f^\varepsilon(x,y,z) - f(x,y,z)\right| \leq \left\|h^\varepsilon-h\right\|_\infty + H_p |z| \left\|\left(\sigma^\varepsilon\right)^{-1} - \sigma^{-1}\right\|_\infty$. Like before, results from \cite{Rennes03} are useful to adapt the end of the proof in \cite{MZ02}.
\end{description}
\end{proof}

\begin{lem}
Under the same assumptions, for every $T > t$, we have the following inequality: 
\[ \left|\partial_x u(t,x)\right| \leq \Ccal \left(1 + |x|^\mu + \IE\left[\left|g\left(X_T^{t,x}\right)\right|^4\right]^{\frac{1}{4}}\right),\]
where $\Ccal$ depends on $T$, $T-t$, $H_x$, $H_y$, $H_p$, $M_f$, $\left\|\sigma^{-1}\right\|_\infty$, $|b(0)|$, $|\sigma(0)|$, $\|b\|_\lip$, $\|\sigma\|_\lip$ and $\mu$.
\end{lem}

\begin{proof}~\\
Our proof is based on the ideas developed in Theorem 3.3 of \cite{Richou11}. In the following, $\Theta_r^{t,x} = \left(X_r^{t,x}, Y_r^{t,x}, Z_r^{t,x}\right)$.
We can rewrite the equation verified by $Y^{t,x}$ and $Z^{t,x}$, under differential form:
\begin{equation}
\drm Y_s^{t,x} = -f\left(\Theta_s^{t,x}\right) \, \drm s + \sum_{j=1}^d{\left[Z_s^{t,x}\right]^{(j)} \, \drm W_s^j},
\label{eq:EDSR_approach_diff}
\end{equation}
where $\left[Z_s^{t,x}\right]^{(j)}$ stands for the $j^{\text{th}}$ coefficient of the row $Z_s^{t,x}$.
By differentiation of (\ref{eq:EDSR_approach_diff}) w.r.t. $x$ and discounting method,
\begin{align*}
&\drm\left(\erm^{\int_t^s \partial_y f\left(\Theta_r^{t,x}\right) \, \drm r} \nabla Y_s^{t,x}\right) \\
& \hspace{1cm} = \erm^{\int_t^s \partial_y f\left(\Theta_r^{t,x}\right) \, \drm r}\left( - \partial_x f\left(\Theta_s^{t,x}\right) \nabla X_s^{t,x} \, \drm s - \sum_{j=1}^d \partial_{z_j} f\left(\Theta_s^{t,x}\right) \nabla \left[Z_s^{t,x}\right]^{(j)} \, \drm s + \sum_{j=1}^d{\nabla \left[Z_s^{t,x}\right]^{(j)} \, \drm W_s^j}\right).
\end{align*}
We set, for every $j \in [\![1,d]\!]$, $\drm \widetilde{W}_s^j = \drm W_s^j - \partial_{z_j} f\left(\Theta_s^{t,x}\right) \, \drm s$; $\widetilde{W}$ is a Brownian motion under a new probability denoted $\IQ$. Between $s$ and $T$ (for $s \in [t,T]$), the integral form of this equation is:
\begin{align}
\erm^{\int_t^T \partial_y f\left(\Theta_r^{t,x}\right) \, \drm r} \nabla Y_T^{t,x} &- \erm^{\int_t^s \partial_y f\left(\Theta_r^{t,x}\right) \, \drm r} \nabla Y_s^{t,x} \nonumber \\
&= - \int_s^T \erm^{\int_t^r \partial_y f\left(\Theta_u^{t,x}\right) \, \drm u} \partial_x f\left(\Theta_r^{t,x}\right) \nabla X_r^{t,x} \, \drm r + \sum_{j=1}^d \int_s^T \erm^{\int_t^r \partial_y f\left(\Theta_u^{t,x}\right) \, \drm u} \nabla \left[Z_r^{t,x}\right]^{(j)} \, \drm \widetilde{W}_r^j.
\label{eq:EDSR_underQ}
\end{align}
We define $\dst F_s^{t,x} = \erm^{\int_t^s \partial_y f\left(\Theta_r^{t,x}\right) \, \drm r}\nabla Y_s^{t,x} 
+ \int_t^s \erm^{\int_t^r \partial_y f\left(\Theta_u^{t,x}\right) \, \drm u} \partial_x f\left(\Theta_r^{t,x}\right) \nabla X_r^{t,x} \, \drm r$,
and then equation (\ref{eq:EDSR_underQ}) becomes:
\begin{equation}
F_s^{t,x} = F_T^{t,x} - \sum_{j=1}^d \int_s^T \erm^{\int_t^r \partial_y f\left(\Theta_u^{t,x}\right) \, \drm u} \nabla \left[Z_r^{t,x}\right]^{(j)} \, \drm \widetilde{W}_r^j,
\label{eq:martingale_F}
\end{equation}
and it tells us that $F^{t,x}$ is a $\IQ$-martingale. We recall, for every $s \in [t,T]$, that $Y_s^{t,x} = u\left(s,X_s^{t,x}\right)$, $\nabla Y_s^{t,x} = \partial_x u\left(s,X_s^{t,x}\right) \nabla X_s^{t,x}$ and $Z_s^{t,x} = \partial_x u\left(s,X_s^{t,x}\right) \sigma\left(X_s^{t,x}\right) = \nabla Y_s^{t,x} \left(\nabla X_s^{t,x}\right)^{-1} \sigma\left(X_s^{t,x}\right)$.
Indeed, $\nabla X_s^{t,x}$ is invertible. We can show that $\nabla X^{t,x}$ is the solution of the linear SDE (because $b$ and $\sigma$ are $\Ccal^1$ with bounded derivatives):
\[
\drm \nabla X_s^{t,x} = \partial_x b\left(X_s^{t,x}\right) \nabla X_s^{t,x} \, \drm s+ \sum_{j=1}^d \partial_x^j \sigma\left(X_s^{t,x}\right) \nabla X_s^{t,x} \, \drm W_s^j,
\]
whose solution is $\dst \nabla X_s^{t,x} = \exp\left(\int_s^t \left\{ \partial_x b\left(X_r^{t,x}\right) - \frac{1}{2} \sum_{j=1}^d{\partial_x^j \sigma\left(X_r^{t,x}\right) \left[\partial_x^j \sigma\left(X_r^{t,x}\right)\right]^*}\right\} \, \drm r + \sum_{j=1}^d \int_t^s \partial_x^j \sigma\left(X_r^{t,x}\right) \, \drm W_r^j\right)$, where the notation $\partial_x^j \sigma\left(X_s^{t,x}\right)$ stands for the $(d \times d)$-matrix
$\begin{pmatrix} \partial_x \left[\sigma\left(X_s^{t,x}\right)\right]^{(1,j)} \\\hline \vdots \\\hline \partial_x \left[\sigma\left(X_s^{t,x}\right)\right]^{(d,j)} \end{pmatrix}$, and where $\left[\sigma\left(X_s^{t,x}\right)\right]^{(i,j)}$ is the coefficient at line $i$ and column $j$ in the matrix $\sigma\left(X_s^{t,x}\right)$. This way, we see that $\left(\nabla X_s^{t,x}\right)^{-1}$ is solution of the following linear SDE:
\begin{equation}
\drm \left(\nabla X_s^{t,x}\right)^{-1} 
  = \left(\nabla X_s^{t,x}\right)^{-1} \left\{-\partial_x b\left(X_s^{t,x}\right) + \frac{1}{2} \sum_{j=1}^d\partial_x^j \sigma\left(X_r^{t,x}\right) \left[\partial_x^j \sigma\left(X_r^{t,x}\right)\right]^*\right\} \, \drm s - \left(\nabla X_s^{t,x}\right)^{-1} \sum_{j=1}^d \partial_x^j \sigma\left(X_s^{t,x}\right) \, \drm W_s^j.
\label{eq:nablaX-1}
\end{equation}
In the following, for every $s \in [t,T]$, we set $R_s^{t,x} = Z_s^{t,x} \sigma\left(X_s^{t,x}\right)^{-1} = \nabla Y_s^{t,x} \left(\nabla X_s^{t,x}\right)^{-1} = \partial_x u\left(s,X_s^{t,x}\right)$, $\dst \beta_s^{t,x} = \int_t^s \erm^{\int_t^r \partial_y f\left(\Theta_u^{t,x}\right) \, \drm u} \partial_x f\left(\Theta_r^{t,x}\right) \nabla X_r^{t,x} \, \drm r \left(\nabla X_s^{t,x}\right)^{-1}$, $\widetilde{R}_s^{t,x} = F_s^{t,x} \left(\nabla X_s^{t,x}\right)^{-1} = \erm^{\int_t^s \partial_y f\left(\Theta_r^{t,x}\right) \, \drm r} R_s^{t,x} + \beta_s^{t,x}$. Using the integration by parts formula, combining (\ref{eq:martingale_F}) and (\ref{eq:nablaX-1}), one gets:
\begin{align}
\drm \widetilde{R}_s^{t,x}
  = \widetilde{R}_s^{t,x} &\left\{-\partial_x b\left(X_s^{t,x}\right) + \sum_{j=1}^d\partial_x^j \sigma\left(X_r^{t,x}\right) \left[\partial_x^j \sigma\left(X_r^{t,x}\right)\right]^*- \sum_{j=1}^d \partial_{z_j} f\left(\Theta_s^{t,x}\right) \partial_x^j \sigma\left(X_s^{t,x}\right) \right\} \, \drm s \nonumber \\
  & - \sum_{j=1}^d \erm^{\int_t^s \partial_y f\left(\Theta_r^{t,x}\right) \, \drm r} \nabla \left[Z_s^{t,x}\right]^{(j)} \left(\nabla X_s^{t,x}\right)^{-1} \partial_x^j \sigma\left(X_s^{t,x}\right)  \, \drm s \nonumber\\
  & + \sum_{j=1}^d \left\{\erm^{\int_t^s \partial_y f\left(\Theta_r^{t,x}\right) \, \drm r} \nabla \left[Z_s^{t,x}\right]^{(j)} \left(\nabla X_s^{t,x}\right)^{-1} - \widetilde{R}_s^{t,x} \partial_x^j \sigma\left(X_s^{t,x}\right) \right\} \, \drm \widetilde{W}_s^j.
 \end{align}
Then, let us take a new parameter $\lambda \in \IR$; we have:
\begin{align}
\drm \left|\erm^{\lambda s} \widetilde{R}_s^{t,x}\right|^2
  &= \erm^{2\lambda s} \widetilde{R}_s^{t,x} \left\{2\lambda \Irm_d -2\partial_x b\left(X_s^{t,x}\right) + 2\sum_{j=1}^d\partial_x^j \sigma\left(X_r^{t,x}\right) \left[\partial_x^j \sigma\left(X_r^{t,x}\right)\right]^* - 2\sum_{j=1}^d \partial_{z_j} f\left(\Theta_s^{t,x}\right) \partial_x^j \sigma\left(X_s^{t,x}\right) \right\} \left(\widetilde{R}_s^{t,x}\right)^* \, \drm s \nonumber \\
  &\pushright{ - 2\erm^{2\lambda s} \sum_{j=1}^d \erm^{\int_t^s \partial_y f\left(\Theta_r^{t,x}\right) \, \drm r} \nabla \left[Z_s^{t,x}\right]^{(j)} \left(\nabla X_s^{t,x}\right)^{-1} \partial_x^j \sigma\left(X_s^{t,x}\right) \left(\widetilde{R}_s^{t,x}\right)^* \, \drm s \nonumber \hspace{2.5cm}} \\
	&\pushright{ + 2\erm^{2\lambda s} \sum_{j=1}^d \left\{\erm^{\int_t^s \partial_y f\left(\Theta_r^{t,x}\right) \, \drm r} \nabla \left[Z_s^{t,x}\right]^{(j)} \left(\nabla X_s^{t,x}\right)^{-1} - \widetilde{R}_s^{t,x} \partial_x^j \sigma\left(X_s^{t,x}\right) \right\} \left(\widetilde{R}_s^{t,x}\right)^* \, \drm \widetilde{W}_s^j \nonumber \hspace{2.5cm}} \\
	& \pushright{+ \erm^{2\lambda s} \sum_{j=1}^d \left\{\left[\erm^{\int_t^s \partial_y f\left(\Theta_r^{t,x}\right) \, \drm r} \nabla \left[Z_s^{t,x}\right]^{(j)} \left(\nabla X_s^{t,x}\right)^{-1} - \widetilde{R}_s^{t,x} \partial_x^j \sigma\left(X_s^{t,x}\right)\right] \right. \nonumber \hspace{2.5cm}} \\
	& \pushright{ \left. \times \left[\erm^{\int_t^s \partial_y f\left(\Theta_r^{t,x}\right) \, \drm r} \nabla \left[Z_s^{t,x}\right]^{(j)} \left(\nabla X_s^{t,x}\right)^{-1} - \widetilde{R}_s^{t,x} \partial_x^j \sigma\left(X_s^{t,x}\right)\right]^* \right\} \, \drm s. \hspace{2.5cm}}
 \end{align}
If we denote $\gamma = \erm^{\int_t^s \partial_y f\left(\Theta_r^{t,x}\right) \, \drm r} \nabla \left[Z_s^{t,x}\right]^{(j)} \left(\nabla X_s^{t,x}\right)^{-1}$ and $\delta = \widetilde{R}_s^{t,x} \partial_x^j \sigma\left(X_s^{t,x}\right)$, we remark that we have the inequality $-2\gamma\delta^* + |\gamma-\delta|^2 = |\gamma-2\delta|^2 -3|\delta|^2$. Thus,
\begin{align}
\drm \left|\erm^{\lambda s} \widetilde{R}_s^{t,x}\right|^2
  &= \erm^{2\lambda s} \widetilde{R}_s^{t,x} \left\{2\lambda \Irm_d -2\partial_x b\left(X_s^{t,x}\right) - \sum_{j=1}^d\partial_x^j \sigma\left(X_r^{t,x}\right) \left[\partial_x^j \sigma\left(X_r^{t,x}\right)\right]^* - 2\sum_{j=1}^d \partial_{z_j} f\left(\Theta_s^{t,x}\right) \partial_x^j \sigma\left(X_s^{t,x}\right) \right\} \left(\widetilde{R}_s^{t,x}\right)^* \, \drm s \nonumber \\
	&\pushright{ + 2\erm^{2\lambda s} \sum_{j=1}^d \left\{\erm^{\int_t^s \partial_y f\left(\Theta_r^{t,x}\right) \, \drm r} \nabla \left[Z_s^{t,x}\right]^{(j)} \left(\nabla X_s^{t,x}\right)^{-1} - \widetilde{R}_s^{t,x} \partial_x^j \sigma\left(X_s^{t,x}\right) \right\} \left(\widetilde{R}_s^{t,x}\right)^* \, \drm \widetilde{W}_s^j \nonumber \hspace{2.5cm}} \\
	&\pushright{ + \erm^{2\lambda s} \sum_{j=1}^d \left|\erm^{\int_t^s \partial_y f\left(\Theta_r^{t,x}\right) \, \drm r} \nabla \left[Z_s^{t,x}\right]^{(j)} \left(\nabla X_s^{t,x}\right)^{-1} - \widetilde{R}_s^{t,x} \partial_x^j \sigma\left(X_s^{t,x}\right)\right|^2 \, \drm s. \hspace{2.5cm}}
 \end{align}
This way, we can see that, for $\lambda$ great enough (that is to say bigger than something depending only on $H_p$, $\left\|\sigma^{-1}\right\|_\infty$ and on bounds over the derivatives of $b$ and $\sigma$), the process $\left(\left|\erm^{\lambda s} \widetilde{R}_s^{t,x}\right|^2\right)_{s \in [t,T]}$ is a $\IQ$-submartingale. So, we get:
\[
\left|R_t^{t,x}\right|^2 (T-t) \leq \erm^{-2\lambda t} \IE^\IQ\left[\int_t^T{\erm^{2\lambda s} \left|\widetilde{R}_s^{t,x}\right|^2 \, \drm s}\right] \leq \erm^{2\lambda(T-t)} \IE^\IQ\left[\int_t^T \left|\widetilde{R}_s^{t,x}\right|^2 \, \drm s\right].
\]
But $\widetilde{R}_s^{t,x} = \erm^{\int_t^s \partial_y f\left(\Theta_r^{t,x}\right) \, \drm r} R_s^{t,x} + \beta_s^{t,x}$ and $\partial_y f$ is bounded by a constant $H_y$, so
\[
\left|R_t^{t,x}\right|^2 (T-t) \leq \erm^{2(\lambda+H_y)(T-t)} \IE^\IQ\left[\int_t^T \left|R_s^{t,x}\right|^2 \, \drm s\right] + \erm^{2\lambda(T-t)} \IE^\IQ\left[\int_t^T \left|\beta_s^{t,x}\right|^2 \, \drm s\right].
\]
On the one hand, we have: 
\begin{align}
\IE^\IQ\left[\int_t^T \left|R_s^{t,x}\right|^2 \, \drm s\right] 
  &\leq \left\|\sigma^{-1}\right\|^2_\infty \IE\left[\erm^{\int_t^T \partial_z f\left(\Theta_s^{t,x}\right)\, \drm W_s - \frac{1}{2} \int_t^T \left|\partial_z f\left(\Theta_s^{t,x}\right)\right|^2 \, \drm s} \int_t^T \left|Z_s^{t,x}\right|^2 \, \drm s\right] \nonumber \\
  &\leq \left\|\sigma^{-1}\right\|^2_\infty\erm^{\frac{1}{2} H_p^2 \left\|\sigma^{-1}\right\|^2_\infty (T-t)} \IE\left[\left(\int_t^T \left|Z_s^{t,x}\right|^2 \, \drm s\right)^2\right]^{\frac{1}{2}}
\end{align}
We recall equation (\ref{eq:Z_Lp}):
$ \dst \IE\left[\left(\int_t^T \left|Z_r^{t,x}\right|^2 \, \drm r\right)^2\right] \leq C \erm^{4a(T-t)} \left\{\IE\left[\left|Y^{t,x}\right|^{*,4}_{t,T}\right] + C_T\left(1+|x|^{4\mu}\right)\right\}$.
By using the Lemma 2.2 of \cite{MZ02}, we have, where $C$ is a constant:
\[
\IE\left[\left|Y^{t,x}\right|^{*,4}_{t,T}\right] \leq C \IE\left[\left|g\left(X_T^{t,x}\right)\right|^4 + \int_t^T \left|f\left(X_r^{t,x},0,0\right)\right|^4 \, \drm r\right].
\]
But, using proposition \ref{majo_X_Sp} and the assumption made on $f(\bullet,0,0)$, we get:
\[ \IE\left[\int_t^T \left|f\left(X_r^{t,x},0,0\right)\right|^4 \, \drm r\right] \leq C\left(1+|x|^{4\mu}\right),\]
where $C$ only depends on $\mu$, $M_f$, $T-t$, $T$, $|\sigma(0)|$, $\|\sigma\|_\lip$, $|b(0)|$ and $\|b\|_\lip$.
On the other hand, using the Lemma 2.1 of \cite{MZ02}, we can show that $\dst \nabla X^{t,x}$ and $\left(\nabla X^{t,x}\right)^{-1}$ are in $\Scal^p$, for all $p < \infty$. Also, we have $\left|\partial_x f\left(\Theta_r^{t,x}\right)\right| \leq H_x + H_p \left|Z_r^{t,x}\right| \left\|\partial_x \sigma^{-1}\right\|_\infty$. Then, with Cauchy-Schwarz inequality and a priori estimates of \cite{Rennes03}, we get that $\dst \IE^{\IQ}\left[\int_t^T \left|\beta_s^{t,x}\right|^2 \, \drm s\right]$ is bounded.
We recall that $R_t^{t,x} = \partial_x u(t,x)$. We have the conclusion:
\[ \left|\partial_x u(t,x)\right|^2 \leq \Ccal \left(1 + |x|^{2\mu} + \IE\left[\left|g\left(X_T^{t,x}\right)\right|^4\right]^{\frac{1}{2}}\right).\]
\end{proof}

\begin{thm}
We make the following assumptions:
\begin{itemize}
\item $\sigma$ and $b$ are Lipschitz continuous; 
\item the function $\sigma(\bullet)^{-1}$ is bounded;
\item $f \in \Ccal^0\left(\IR^d \times \IR \times \left(\IR^d\right)^*, \IR\right)$ and $\forall x \in \IR^d,\, |f(x,0,0)| \leq M_f \left(1+|x|^\mu\right)$ and $|g(x)| \leq M_g\left(1+|x|^\nu\right)$;
\item we can write $f(x,y,z) = \psi(x,z) - \alpha y$, with $\alpha \in (0,1]$ and if we define $h : (x,p) \mapsto \psi(x,p\sigma(x))$, then $h$ is Lipschitz continuous (with bounds denoted $H_x$ and $H_p$ w.r.t. $x$ and $p$).
\end{itemize}
Under those assumptions, the function $u : (t,x) \mapsto Y_t^{t,x}$ satisfies the following inequality, for every $T > t$:
\[ \left|u(t,x)-u(t,x')\right| \leq \Ccal \left(1 + |x|^{\mu \vee \nu} + |x'|^{\mu \vee \nu}\right)|x-x'|,\]
where $\Ccal$ depends on $T$, $T-t$, $H_x$, $H_p$, $M_f$, $M_g$, $\left\|\sigma^{-1}\right\|_\infty$, $|b(0)|$, $|\sigma(0)|$, $\|b\|_\lip$, $\|\sigma\|_\lip$ and $\mu$.
\label{thm:u_Lipschitz} \label{thm:u_locLipschitz}
\end{thm}

\begin{proof}~\\
We approximate the functions $h$, $g$, $b$ and $\sigma$ by some uniformly converging sequences $\left(h^\varepsilon\right)_{\varepsilon > 0}$, $\left(g^\varepsilon\right)_{\varepsilon > 0}$, $\left(b^\varepsilon\right)_{\varepsilon > 0}$ and $\left(\sigma^\varepsilon\right)_{\varepsilon > 0}$ of functions of class $\Ccal^1$ with bounded derivatives. We can assume that there exists a bound, independent of $\varepsilon$, of all the derivatives of $h^\varepsilon$, $b^\varepsilon$ and $\sigma^\varepsilon$, for every $\varepsilon > 0$. We define $f^\varepsilon(x,y,z) = h^\varepsilon\left(x,z\sigma(x)^{-1}\right) - \alpha y$ and we can assume that $f^\varepsilon(\bullet,0,0) = h^\varepsilon(\bullet,0,0)$ and $g^\varepsilon$ have polynomial growth independently of $\varepsilon$, with constants $M_f$ and $M_g$ and exponents $\mu$ and $\nu$. For every $(t,x) \in [0,T) \times \IR^d$, $\left|\partial_x u^\varepsilon(t,x)\right| \leq \Ccal\left(1+ |x|^{\mu \vee \nu}\right)$, where $\Ccal$ does not depend on $\varepsilon$ or $x$. Using the mean-value theorem, it comes that :
\begin{equation}
\forall t \in [0,T),\, \forall x,x' \in \IR^d,\, \left|u^\varepsilon(t,x) - u^\varepsilon(t,x')\right| \leq \Ccal\left(1+ |x|^{\mu \vee \nu} + |x'|^{\mu \vee \nu}\right) |x-x'|.
\label{eq:Ascoli}
\end{equation}
Then, for every $T > 1$, the set $\left\{ u^\varepsilon(t,\bullet) \middle| \varepsilon > 0, t \in [0,T-1]\right\}$ is equicontinuous; also it is pointwise bounded: we can show it by writing the equation verified by $Y^{\varepsilon,t,x}$ and using the Lemma 2.2 of \cite{MZ02}. Let $K$ be a compact subset of $\IR^d$. By the Arzelà-Ascoli theorem, we get, for a sequence $\left(\varepsilon^{(K)}_n\right)_{n \geq 0}$ which has $0$ as limit:
\[
\forall t \in [0,T-1],\, \restriction{u^{\varepsilon^{(K)}_n}(t,\bullet)}{K} \overset{\|.\|_\infty}{\underset{n \rightarrow \infty}{\longrightarrow}} u_K(t,\bullet).
\]
Let $x \in K$; we have: $u^{\varepsilon^{(K)}_n}(t,x) \underset{n \rightarrow \infty}{\longrightarrow} u_K(t,x)$ and $u^{\varepsilon^{(K)}_n}(t,x) = Y_t^{\varepsilon^{(K)}_n,t,x} \underset{n \rightarrow \infty}{\longrightarrow} Y_t^{t,x}$. So the limit $u_K(t,x) = Y_t^{t,x}$ does not depend on $K$; and the convergence of $u^{\varepsilon_n^{(K)}}(t,\bullet)$ is uniform on $K$.
The following triangle inequality gives the result claimed after taking the limit as $n$ goes to infinity (because $\Ccal$ is independent of $K$): for every $x,x' \in K$ and $t \in [0,T-1]$,
\[
\forall n \in \IN,\, |u(t,x)-u(t,x')| \leq \left|u(t,x)-u^{\varepsilon_n^{(K)}}(t,x)\right| + \Ccal\left(1+ |x|^{\mu \vee \nu} + |x'|^{\mu \vee \nu}\right)|x-x'| + \left|u^{\varepsilon_n^{(K)}}(t,x')-u(t,x')\right|.
\]
\end{proof}

\end{document}